\documentclass[12pt,a4paper]{article}

\usepackage[english]{babel}
\RequirePackage{etoolbox,graphicx,textfit,changepage,setspace,ifthen,fancyhdr,calc,pdfpages,makeidx,color}
\RequirePackage{hyperref}
\RequirePackage[all]{hypcap}
\usepackage{amsmath,amssymb,amsthm,graphicx,xypic}
\usepackage{multirow}
\usepackage{hyperref}
\hypersetup{linktocpage}

\newtheorem{thm}{Theorem}[subsection]
\newtheorem{lem}[thm]{Lemma}
\newtheorem{prop}[thm]{Proposition}
\newtheorem{cor}[thm]{Corollary}

\theoremstyle{definition}
\newtheorem{de}[thm]{Definition}
\newtheorem{example}[thm]{Example}

\theoremstyle{remark}

\numberwithin{equation}{section}

\def\sideremark#1{\ifvmode\leavevmode\fi\vadjust{\vbox to0pt{\vss
 \hbox to 0pt{\hskip\hsize\hskip1em
 \vbox{\hsize3cm\tiny\raggedright\pretolerance10000
 \noindent #1\hfill}\hss}\vbox to8pt{\vfil}\vss}}}

\def\sideremark#1{\relax} 

\newlength{\sirkaustavorig} 
\newlength{\sirkafakultaorig} 
\begin{document}

\begin{titlepage}
\pagestyle{empty}

\changetext{}{11.6mm}{}{}{}
\begin{adjustwidth}{-2mm}{}
\vspace*{-5mm}
 \hspace{-17mm}
\begin{minipage}{28mm}
  \includegraphics[width=28mm]{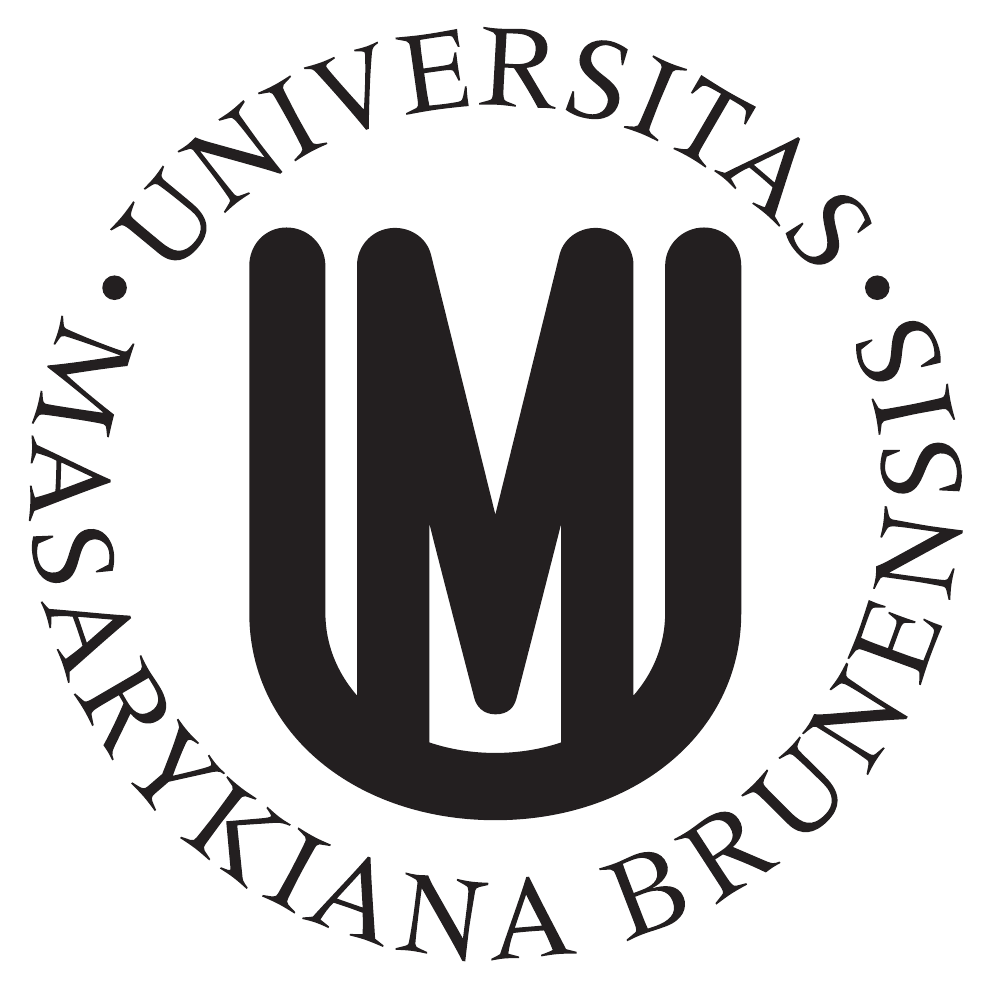}
\end{minipage}\hspace{2mm}
\begin{minipage}{113mm}
  \begin{center}\vspace*{3mm}
    {\fontsize{22pt}{28pt}\selectfont{\textbf{MASARYK UNIVERSITY}}}\\[2mm]
    {\fontsize{21pt}{28pt}\selectfont{\textbf{\textsc{Faculty of Science}}}}\\[1mm]

    \settowidth{\sirkaustavorig}{\fontsize{16pt}{28pt}\selectfont{\textbf{\textsc{Department of Mathematics and Statistics}}}}
   
\settowidth{\sirkafakultaorig}{\fontsize{21pt}{28pt}\selectfont{\textbf{\textsc{P{\v{r}}{\'{i}}rodov{\v{e}}deck{\'{a}}
fakulta}}}}
    \ifthenelse{\lengthtest{\sirkaustavorig > 0.95\sirkafakultaorig}}%
        {\vspace*{-0.5mm}\textbf{\textsc{\scaletowidth{0.95\sirkafakultaorig}{Department of Mathematics and Statistics}}}\\[-1mm]}%
        {\vspace*{-4mm}\fontsize{16pt}{28pt}\selectfont{\textbf{\textsc{Department of Mathematics and Statistics}}}\\[-6mm]}
    \noindent \rule{\textwidth}{2pt}\vspace*{2mm}
  \end{center}
\end{minipage}\hspace{2mm}
\begin{minipage}{28mm}
  \includegraphics[width=28mm]{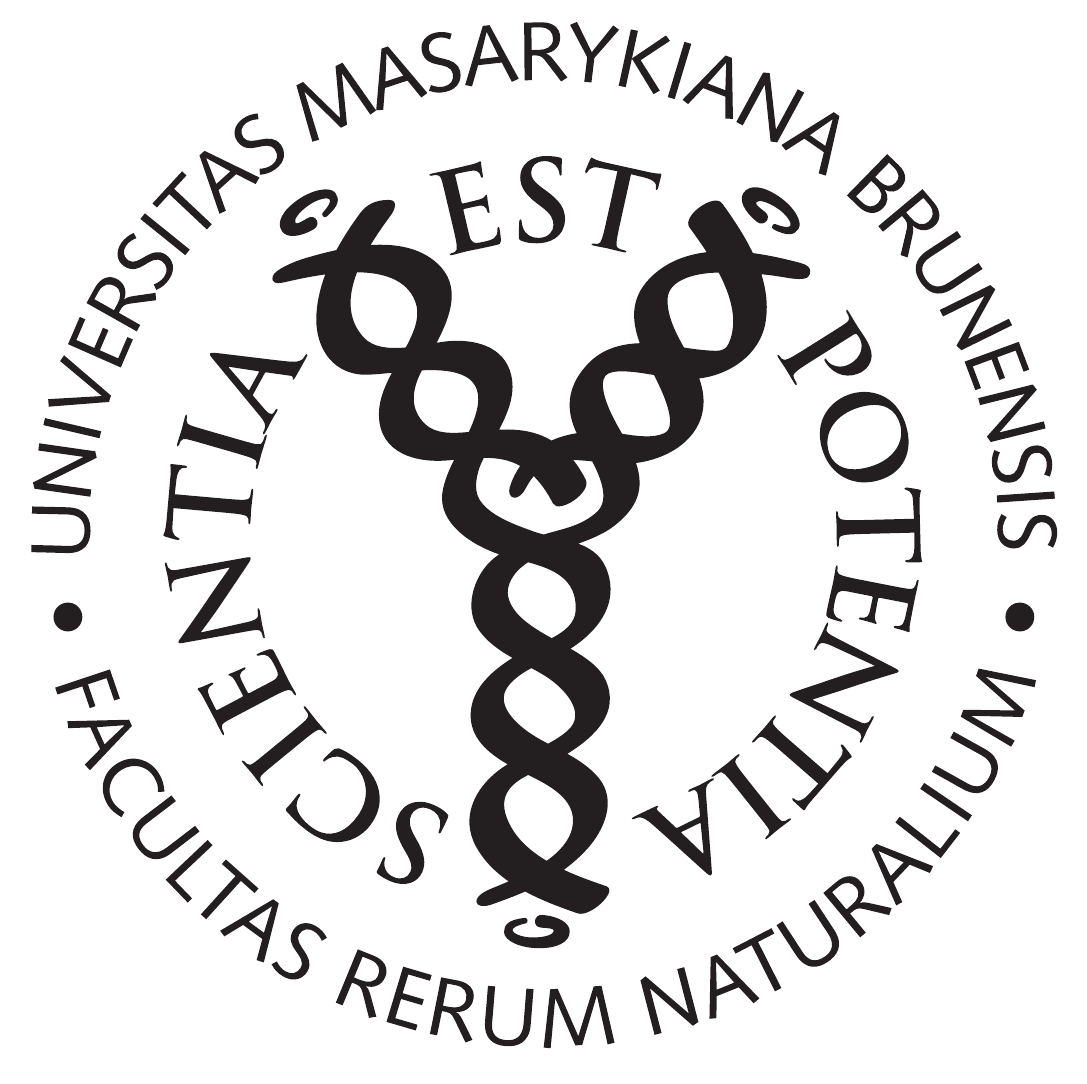}
\end{minipage}
\end{adjustwidth}
\changetext{}{-11.6mm}{}{}{}

 \vfill

\begin{center}
\fontsize{16pt}{1} 
Ph.D. Dissertation
\selectfont 
\bf
\\[24pt]
\fontsize{24pt}{1} 
\selectfont 
Geometric structures invariant to symmetries
\\[36pt]
\rm
\fontsize{20pt}{1} 
\selectfont 
\sc
Jan Gregorovi\v c
\vfill
\fontsize{13pt}{1} 
\selectfont 
Supervisor: Prof. RNDr. Jan Slov\' ak, DrSc.
\\[24pt]
\rule[2pt]{\textwidth}{.5pt}
\rm
Brno, 2012

\end{center}
\end{titlepage}

\thispagestyle{empty}

\section*{Bibliographic entry}
\newlength{\LevySloupecENG}
\settowidth{\LevySloupecENG}{{\bf Degree Programme:} } 

\newlength{\PravySloupecENG}
\setlength{\PravySloupecENG}{\textwidth - \LevySloupecENG - 24 pt} 

\noindent
\begin{tabular}{ p{\LevySloupecENG} p{\PravySloupecENG} }
  {\bf Author:}           &Mgr. Jan Gregorovi\v c \\
                          & Faculty of Science, Masaryk University\\
                          & Department of Mathematics and Statistics \\[6mm]
  {\bf Title of Thesis:}  &Geometric structures invariant to symmetries  \\[6mm]
  {\bf Degree Programme:} &Mathematics (4 years)  \\[6mm]
  {\bf Field of Study:}   &Geometry, Topology and Global Analysis  \\[6mm]
  {\bf Supervisor:}       &Prof. RNDr. Jan Slov\' ak, DrSc. \\
                          & Faculty of Science, Masaryk University\\
                          & Department of Mathematics and Statistics \\[6mm]
  {\bf Academic Year:}    &2011/2012  \\[6mm]
  {\bf Number of Pages:}  & 90 \\[6mm]
  {\bf Keywords:}         &symmetric spaces, geometric structures, Cartan geometries, parabolic geometries, homogeneous spaces, extension functors
\end{tabular}

\cleardoublepage
\thispagestyle{empty}

\section*{Bibliografick{\'{y}} z{\'{a}}znam}
\newlength{\LevySloupecCZ}
\settowidth{\LevySloupecCZ}{{\bf Studijn{\'{i}} program:} } 

\newlength{\PravySloupecCZ}
\setlength{\PravySloupecCZ}{\textwidth - \LevySloupecCZ - 24 pt} 

\noindent
\begin{tabular}{p{\LevySloupecCZ} p{\PravySloupecCZ} }
  {\bf Autor:}            &Mgr. Jan Gregorovi\v c  \\
                          & P{\v{r}}{\'{i}}rodov{\v{e}}deck{\'{a}} fakulta, Masarykova univerzita\\
                          & \'Ustav Matematiky a Statistiky  \\[6mm]
  {\bf N{\'{a}}zev pr{\'{a}}ce:}      &Geometrick\'e struktury invariantn\'i vzhledem k symetri\'im  \\[6mm]
  {\bf Studijn{\'{i}} program:} & Matematika (\v cty\v rlet\'e) \\[6mm]
  {\bf Studijn{\'{i}} obor:}    &Geometrie, Topologie a Glob\'aln\'i anal\'yza \\[6mm]
  {\bf Vedouc{\'{i}} pr{\'{a}}ce:}    &Prof. RNDr. Jan Slov\' ak, DrSc. \\
                          & P{\v{r}}{\'{i}}rodov{\v{e}}deck{\'{a}} fakulta, Masarykova univerzita\\
                          & \'Ustav Matematiky a Statistiky  \\[6mm]
  {\bf Akademick{\'{y}} rok:}   &2011/2012 \\[6mm]
  {\bf Po{\v{c}}et stran:}      &90 \\[6mm]
  {\bf Kl{\'{i}}{\v{c}}ov{\'{a}} slova:}    &symetrick\'e prostory, geometrick\'e struktury, Cartanovy geometrie, parabolick\'e geometie, homogenn\'i prostory, extension funktor
\end{tabular}

\cleardoublepage

\section*{Abstract}

We introduce and discuss (local) symmetries of geometric structures. These symmetries generalize the classical (locally) symmetric spaces to various other geometries. Our main tools are homogeneous Cartan geometries and their explicit description. This allows us to describe the structure of symmetric geometric structures and to provide a general construction of such structures. Since we can view the classical (locally) symmetric spaces as special case, this allows us to classify various geometric structures on semisimple symmetric spaces. Then we investigate the case of symmetric parabolic geometries in detail and obtain classification of symmetric AHS-structures and symmetric parabolic contact geometries in the semisimple cases.

\section*{Abstrakt}

V t\'eto pr\'aci se budeme zab\'yvat (lok\'aln\'imi) symetriemi geometrick\'ych  struktur. Tyto symetrie zobec\v nuj\'i klasick\'e symetrick\'e prostory pro dal\v s\'i geometrie. Hlavn\'imi n\'astroji jsou homogenn\'i Cartanovy geometrie a jejich p\v resn\'y popis. Tyto n\'astroje n\'am umo\v z\v nuj\'i popsat strukturu symetrick\'ych geometrick\'ych struktur a p\v rin\'a\v s\'i obecnou konstrukci t\v echto geometi\'i. To, \v ze jsou symetrick\'e prostory speci\'aln\'im p\v r\'ipadem t\v echto struktur, n\'am umo\v z\v nuje klasifikovat n\v ekter\'e geometrick\'e struktury na polojednoduch\'ych symetrick\'ych prostorech. D\'ale podrobn\v e prozkoum\'ame symetrie parabolick\'ych geometri\'i a v polojednoduch\'em p\v r\'ipad\v e klasifikujeme symetrick\'e AHS-struktury a symetrick\'e parabolick\'e kontaktn\'i struktury.

\thispagestyle{empty} 
\newpage
\thispagestyle{empty} 
\newpage
\quad
\vfill

\begin{center}
\copyright Jan Gregorovi\v c, Masaryk University, 2012
\end{center}
\thispagestyle{empty} 
\newpage

\tableofcontents

\thispagestyle{empty} 
\newpage
\setcounter{page}{1}
\section*{Preface}

Symmetric spaces are fascinating geometric structures studied in different areas of mathematics for more than hundred years. One of the viewpoints on symmetric spaces is that a symmetric space is a smooth connected manifold equipped with special diffeomorphisms, called symmetries, chosen at each point in a such way that the symmetries preserve this chosen structure. This allows a direct generalization to other geometric structures just by assuming that the symmetries preserve the other geometric structure.

The first and third chapter of this work contain description of several geometric structures, for which we want to define the symmetries. This is done in the second chapter and at the end end of the third chapter. The second chapter also explains, how the classical symmetric spaces fit in our approach, and investigates geometric structures on symmetric spaces. In the fourth and fifth chapters, we investigate and construct examples of symmetric parabolic geometries in the simplest cases. In fact, we obtain classification of non-flat symmetric AHS-structures and parabolic contact structures in the case when the groups generated by the symmetries are semisimple, and we explicitly construct many of them.

For the convenience of the reader, the appendices display tables containing  the classification of simple symmetric spaces and tables containing the gradings associated to the parabolic geometries, which we are investigating.

\bf Acknowledgments.\rm I would like to thank my supervisor Jan Slov\'ak for all-round help,
support and encouragement. I would also like to thank Lenka Zalabov\'a for many discussions about the topic. This work was supported by the grant GACR 201/09/H012.

\newpage
\section{Geometric structures}

In this chapter, we summarize the description and basic properties of two types of geometric structures, namely, $P_0$-structures and Cartan connections. The section \ref{1.1} is based on the approach to geometric structures in \cite{odk2} and the description of their automorphism in \cite{odk1}. The section \ref{1.2} is devoted to Cartan geometries and their automorphism and contains results from \cite{odk3}. The section \ref{1.3} deals with special class of Cartan geometries, where the group of automorphisms of the Cartan geometry acts transitively. In this case, there is an explicit description of such geometries, which is contained in \cite{odk4} in the global case and some further details can be found in \cite{odk5}. The description in the local case is then a direct generalization of the global case. This description will provide a general construction of symmetric geometric structures. The section \ref{1.4} contains known results about relations between homogeneous spaces, which provide examples of our construction. The last section \ref{1.5} deals with the case of affine geometries and first order $P_0$-structures, where our approaches to geometric structures coincide. These results can also be found in \cite{odk3}.


\subsection{$P_0$-structures}\label{1.1}

Let $M$ be a smooth connected manifold of dimension $n$ and let $P^1M$ be the bundle of linear frames over $M$, then $P^1M$ is a principal $GL(n,\mathbb{R})$-bundle over $M$. The fibre of $P^1M$ is the set of all possible frames of the tangent bundle $T_xM$ and $GL(n,\mathbb{R})$ is identified with the set of all transition maps between the frames, as soon as one of the frames is fixed.

\begin{de}
Let $P_0$ be a Lie subgroup of $GL(n,\mathbb{R})$ and $\mathcal{P}_0$ a principal $P_0$-bundle over $M$, then a $P_0$-structure is a $P_0$-equivariant inclusion $i: \mathcal{P}_0\to P^1M$.
\end{de}

In other words, a $P_0$-structure is a choice of a subbundle of $P^1M$ with transition maps from $P_0$. More generally, we can define:

\begin{de}
Let $P_0$ be a Lie group, $j: P_0\to GL(n,\mathbb{R})$ a Lie group homomorphism and $\mathcal{P}_0$ a principal $P_0$-bundle over $M$, then a $P_0$-structure is a principal bundle morphism $i: \mathcal{P}_0\to P^1M$ over $j$.
\end{de}

Obviously, $i(\mathcal{P}_0)$ is an underlying $j(P_0)$-structure. Further, there are the following important examples of $P_0$-structures.

\begin{example}
Restriction $\mathcal{P}_0|_{U}$ of $P_0$-structure $\mathcal{P}_0$ to an open subset $U$ of $M$ is again a $P_0$-structure.
\end{example}

\begin{example}\label{1.1.3}
Let $K$ be a Lie group with Lie subgroup $H$. Let $\mathfrak{h}$ be the Lie algebra of $H$ and $\mathfrak{k}$ be the Lie algebra of $K$. Then the restriction of $Ad$ action of $H$ to the quotient $\mathfrak{k}/\mathfrak{h}$ defines a Lie algebra homomorphism\linebreak $\underline{Ad}: H\to GL(\mathfrak{k}/\mathfrak{h})$. Then we can form the semidirect product $\mathcal{P}_0=\mathfrak{k}/\mathfrak{h}\ltimes_{\underline{Ad}} H$, which is a principal $H$-bundle over $\mathfrak{k}/\mathfrak{h}$. Since $P^1(\mathfrak{k}/\mathfrak{h})\cong \mathfrak{k}/\mathfrak{h}\ltimes Gl(\mathfrak{k}/\mathfrak{h})$, we see, that the choice of a frame $\beta$ of $\mathfrak{k}/\mathfrak{h}$ provides an $H$-structure $\mathcal{P}_0$ over $i_\beta\circ \underline{Ad}$, where $i_\beta: Gl(\mathfrak{k}/\mathfrak{h})\to Gl(n,\mathbb{R})$ is the isomorphism induced by the frame $\beta$.

If $(K,H)$ is a reductive pair i.e. there is an $Ad(H)$-invariant complement of $\mathfrak{h}$ in $\mathfrak{k}$, which we will denote again $\mathfrak{k}/\mathfrak{h}$, then $Ad: H\to GL(\mathfrak{k}/\mathfrak{h})$ and we can form the semidirect product $\mathcal{P}_0=\mathfrak{k}/\mathfrak{h}\ltimes_{Ad} H$. Then the choice of a frame $\beta$ of $\mathfrak{k}/\mathfrak{h}$ provides an $H$-structure $\mathcal{P}_0$ over $i_\beta\circ Ad$.
\end{example}

A diffeomeorphism $\phi: M\to M'$ between two smooth manifolds $M$ and $M'$  induces a principal fibre bundle isomorphism $P^1\phi: P^1M\to P^1M'$. If there is a $P_0$-structure $\mathcal{P}_0'$ on $M'$, then $(P^1\phi)^{-1}i'(\mathcal{P}_0')$ is a $j(P_0)$-structure on $M$.

\begin{de}
Let $i:\mathcal{P}_0\to P^1M,\ i':\mathcal{P}_0'\to P^1M'$ be $P_0$-structures on $M$ and $M'$. Then we say, that a diffeomorphism $\phi: M\to M'$ satisfying $$(P^1\phi)^{-1}i'(\mathcal{P}_0')=i(\mathcal{P}_0)$$ is an isomorphism of $P_0$-structures $\mathcal{P}_0$ and $\mathcal{P}_0'$. 

We will call a diffeomorphism between open subsets $U,\ U' $ of $M,\ M'$ a local isomorphism of $P_0$-structures, if it is an isomorphism of $P_0$-structures $\mathcal{P}_0|_U$ and $\mathcal{P}_0'|_{U'}$.

We say that a $P_0$-structure $\mathcal{P}_0$ is (locally) equivalent to a $P_0$-structure $\mathcal{P}_0'$, if there is a covering $U_i$ of $M$ by open sets together with maps\linebreak $\phi_{U_i}:U_i\to M'$ such, that each $\phi_{U_i}$ is a local isomorphism of $P_0$-structures $\mathcal{P}_0|_{U_i}$ and $\mathcal{P}_0'|_{\phi_{U_i}(U_i)}$.
\end{de}

A vector field $X$ on $M$ defines for small $t\in \mathbb{R}$ a locally defined diffeomorphism $\phi_t$ by sending $x\mapsto Fl^{X}_t(x)$, where $Fl_t^X$ is the flow of the vector field $X$. Then $\phi_t$ is a local one parameter group of locally defined diffeomorphism i.e. $\phi_0=id_M$ and $\phi_{s+t}=\phi_t\circ \phi_s$ for small $t$ and $s$. If the vector field $X$ is complete, then $\phi_t$ are globally defined diffeomorphisms, for sufficiently small $t$ .

\begin{de}
Let $\mathcal{P}_0$ be a $P_0$-structure on $M$, then a vector field is an infinitesimal automorphism, if its flow generates a local one parameter group of local automorphisms of $P_0$-structure $\mathcal{P}_0$.
\end{de}

If the infinitesimal automorphism $X$ is complete, then it generates a local one parameter group of automorphisms of $\mathcal{P}_0$, which is contained in the connected component of the identity of the group of automorphisms of $\mathcal{P}_0$, by the very definition of the flow.

Next we will remind the prolongations of $P_0$-structures.

Let $\mathfrak{p}_0\subset \mathfrak{gl}(n,\mathbb{R})$ be the Lie algebra of $P_0\subset Gl(n,\mathbb{R})$. We define the first prolongation $\mathfrak{p}_1$ as the space of all tensors $t\in S^2\mathbb{R}^n$ such, that the map $v\mapsto t(v,v_1)$ is in $\mathfrak{p}_0$, for all fixed $v_1\in \mathbb{R}^n$ . In fact, $\mathfrak{p}_1$ is the kernel of the differential $\partial: \mathbb{R}^{n*}\otimes \mathfrak{p}_0\to \mathbb{R}^{n*}\wedge \mathbb{R}^{n*}\otimes \mathbb{R}^{n}$. Then we can define subgroup $P_1\subset Gl(\mathbb{R}^n+\mathfrak{p}_0)$ induced by the elements $t\in \mathfrak{p}_1$ as $v\mapsto v+t(.,v)$ for $v\in \mathbb{R}^n$ and $id_{\mathfrak{p}_0}$ on the rest. It is obvious that $\mathfrak{p}_1$ is the Lie algebra of $P_1$. Iterating this process we get the $k$-th prolongation $\mathfrak{p}_k$ and $P_k$.

\begin{de}
If $\mathfrak{p}_k={0}$ for some $k$, then the $P_0$-structures are said to be of finite type of order $k$. If there is no such $k$, then they are said to be of infinite type.
\end{de}

Let $\mathcal{P}_0$ be a $P_0$-structure on $M$ and let $\theta_u: T_u\mathcal{P}_0\to \mathbb{R}^{n}$ be the soldering form on $\mathcal{P}_0$ i.e. map given as the composition of projection to $TM$ and the evaluation of coordinates in the frame $u$. Let $H$ be a complement to vertical subspace in $T_u\mathcal{P}_0$, then $\theta_u$ provides an isomorphism of $\mathbb{R}^{n}$ and $H$ and $T^H_u(X,Y):=d\theta(u)(\theta_u^{-1}(X),\theta_u^{-1}(X))$ is an element of $\mathbb{R}^{n*}\wedge \mathbb{R}^{n*}\otimes \mathbb{R}^{n}$. In \cite{odk2}, it is shown that for any two complements $H_1, H_2$ the difference $T^{H_1}_u(X,Y)-T^{H_2}_u(X,Y)\in \partial(\mathbb{R}^{n*}\otimes \mathfrak{p}_0)$. So there is the unique $T_u:=[T_u^{H}]$ in $\mathbb{R}^{n*}\wedge \mathbb{R}^{n*}\otimes \mathbb{R}^{n}/\partial(\mathbb{R}^{n*}\otimes \mathfrak{p}_0)$ and we define:

\begin{de}
For a $P_0$-structure $\mathcal{P}_0$ on $M$, the map $T:\mathcal{P}_0\to \mathbb{R}^{n*}\wedge \mathbb{R}^{n*}\otimes \mathbb{R}^{n}/\partial(\mathbb{R}^{n*}\otimes \mathfrak{p}_0)$ is called the torsion of the $P_0$-structure $\mathcal{P}_0$.
\end{de}

If we follow \cite{odk2}, then a choice of a complement of $\partial(\mathbb{R}^{n*}\otimes \mathfrak{p}_0)$ in $\mathbb{R}^{n*}\wedge \mathbb{R}^{n*}\otimes \mathbb{R}^{n}$ provides a construction of a $P_1$-structure $\mathcal{P}_1$ on the manifold $\mathcal{P}_0$ called the first prolongation. Iterating this process we get the $k$-th prolongation as $P_k$-structure $\mathcal{P}_k$ on $\mathcal{P}_{k-1}$. The following proposition reduces the problem of equivalence of the $P_0$-structures, the proof can be found in \cite[pp. 336]{odk2}:

\begin{prop}
For any fixed choice of complements, the $P_0$--structures $\mathcal{P}_0$ and $\mathcal{P}_0'$ are locally equivalent if and only if the $k$-th prolongations $\mathcal{P}_k$ and $\mathcal{P}_k'$ are locally equivalent $P_k$-structures.
\end{prop}

An automorphism $\phi_0$ of $\mathcal{P}_0$ induces an automorphism $\phi_1=P^1\phi_0|_{\mathcal{P}_1}$ of $\mathcal{P}_1$. So inductively $\phi_k$ is an automorphism of the $k$-th prolongation. On the other hand, an automorphism of the $k$-th prolongation induces by composing with the projection an underlying mapping, which is clearly an automorphism of the $(k-1)$-st prolongation. Thus the following holds, for detailed proof look in \cite[pp. 22]{odk1}.

\begin{prop}
The group of automorphisms of a $P_0$-structure $\mathcal{P}_0$ is the same as the group of automorphisms of the $k$-th prolongation $\mathcal{P}_k$.

The Lie algebra of infinitesimal automorphisms of a $P_0$-structure $\mathcal{P}_0$ is the same as the Lie algebra of infinitesimal automorphisms of the $k$-th prolongation $\mathcal{P}_k$.
\end{prop}

If $P_0$-structures are of finite type of order $k+1$, then there is an $e$-structure on the $k$-th prolongation $\mathcal{P}_k$, which is equivalent to a cross section $\Omega: \mathcal{P}_k\to P^1\mathcal{P}_k$. Then we can define a map $\omega=\theta\circ T_u\Omega: T_u\mathcal{P}_k\to \mathbb{R}^n+\mathfrak{p}_0+\dots+\mathfrak{p}_k$, where $\theta$ is the soldering form on $P^1\mathcal{P}_k$. There is the following trivial example:

\begin{example}
We start with the semidirect product $\mathcal{P}_0=\mathbb{R}^n \ltimes P_0$ over $\mathbb{R}^n$, which is trivially a $P_0$-structure. If $P_0$-structures are of finite type of order $k+1$, then $\mathcal{P}_k=\mathbb{R}^n \ltimes P$, where $P=(\dots(P_0\ltimes P_1)\ltimes \dots )\ltimes P_k$ and $\ltimes$ is a semidirect product of Lie groups induced by the standard actions. If we denote $\mathfrak{g}$ the vector space $\mathbb{R}^n+\mathfrak{p}_0+\dots+\mathfrak{p}_k$, then we extend the structure of the Lie algebra of $P$ to a Lie algebra structure on $\mathfrak{g}$ by defining the missing Lie brackets by $[v_1,v_2]=0$ for $v_1,v_2\in \mathbb{R}^n$ and $[X,v_1]=X(.,\dots,.,v_1)$ for $X\in \mathfrak{p}_0+\dots+\mathfrak{p}_k$,. Then $\omega$ is a $\mathfrak{g}$-valued one form on $\mathcal{P}_k$, which is a linear isomorphism of $T_u\mathcal{P}_k$ and $\mathfrak{g}$ for all $u\in \mathcal{P}_k$. It can be shown, that $\omega$ is $P$-equivariant in this case.
\end{example}

The example might indicate, that there could be a well defined right action of $P$ on $\mathcal{P}_k$ compatible with $\omega$. Generally this is not true. However, if we collect general results from \cite{odk2}, we can formulate the following proposition, using the notation of the previous example:

\begin{prop}
Let $\mathcal{P}_0$ be a $P_0$-structure on $M$ of finite type of order $k+1$. Then:
\begin{enumerate}
\item $p: \mathcal{P}_k\to M$ is a fibre bundle over $M$ with fibre $P_x$ diffeomorphic to $P\subset Gl(\mathbb{R}^n+\mathfrak{p}_0+\dots+\mathfrak{p}_{k-1})$;

\item $\omega$ is a $\mathfrak{g}$-valued one form on $\mathcal{P}_k$;

\item $\omega$ is a linear isomorphism of $T_u\mathcal{P}_k$ and $\mathfrak{g}$ for all $u\in \mathcal{P}_k$.
\end{enumerate}

Moreover, an automorphism $\phi$ of $\mathcal{P}_0$ induces a fibre bundle morphism $\phi_k$ such that the pullback $(\phi_k)^*\omega=\omega$. An infinitesimal automorphism $X$ induces a vector field on $\mathcal{P}_k$ such, that $\mathfrak{L}_X\omega=0$, where $\mathfrak{L}$ is the Lie derivative.
\end{prop}

If there is a well defined right action of $P$ on $\mathcal{P}_k$ compatible with $\omega$, the proposition leads to the notion discussed in the next section.

\subsection{Cartan geometries}\label{1.2}

The Cartan geometries are the following geometric structures, for more details and proofs look in \cite{odk3}.

\begin{de}
Let $G$ be a Lie group with Lie algebra $\mathfrak{g}$, and $P$ be a Lie subgroup of $G$. Let $p: \mathcal{G}\to M$ be a principal fibre bundle with structure group $P$. We say that $\mathfrak{g}$-valued one form $\omega$ on $\mathcal{G}$ is a Cartan connection if:
\begin{enumerate}
\item $\omega$ is equivariant with respect to the action $r$ of $P$ i.e. $(r^p)^*\omega=Ad(p^{-1})\circ \omega$ for all $p\in P$;

\item $\omega$ reproduces the fundamental vector fields of the action of $P$;

\item $\omega$ is a linear isomorphism of $T_u\mathcal{G}$ and $\mathfrak{g}$ for all $u\in \mathcal{G}$.
\end{enumerate}
The pair $(p: \mathcal{G}\to M,\omega)$ is then called Cartan geometry of type $(G,P)$.
\end{de}

If $\beta$ is a frame of $\mathfrak{g}$, then $\omega^{-1}_u(\beta)$ is a frame of $T_u\mathcal{G}$ and we obtain a cross section $\sigma_{\beta}: \mathcal{G}\to P^1\mathcal{G}$. Further, $T_up\circ\omega^{-1}_u$ induces a map $\mathcal{G}\times \mathfrak{g}\to TM$. Since $\omega^{-1}(X)$ is vertical for $X\in \mathfrak{p}$, we get a vector bundle isomorphism of $\mathcal{G}\times_{\underline{Ad}} \mathfrak{g}/\mathfrak{p}$ and $TM$, where the action $\underline{Ad}$ of $P$ is the restriction of the adjoint action on the quotient space. Let $P^1$ be the kernel of $\underline{Ad}: P \to GL(\mathfrak{g}/\mathfrak{p})$ and $P_0:=P/P^1$. Then $\mathcal{P}_0:=\mathcal{G}/P^1 \cong \operatorname{Im}(P^1p\circ \sigma_{\beta})\subset P^1M$ is a principal fibre bundle with structure group $P_0$ i.e. $P_0$-structure on $M$. So we have shown that:

\begin{lem}\label{1.2.2}
Let $(p: \mathcal{G}\to M,\omega)$ be a Cartan geometry of type $(G,P)$. There is a $P$-structure $\mathcal{G}$ on $M$ over $\underline{Ad}: P \to GL(\mathfrak{g}/\mathfrak{p})$ and there is the underlying $P_0$-structure $\mathcal{P}_0$ on $M$.
\end{lem}

Now, $\mathcal{G}\to \mathcal{P}_0$ is a principal fibre bundle with structure group $P_1$. Since $P_1\subset P$, $\omega$ is clearly a Cartan connection of type $(G,P_1)$ on $\mathcal{G}\to \mathcal{P}_0$. So we can iterate this process for the projection from $\mathcal{G}\to \mathcal{P}_i$ to obtain a sequence $\mathcal{P}_i:=\mathcal{G}/P^{i+1} \subset P^1\mathcal{P}_{i-1}$ of principal fibre bundles with structure groups $P_i:=P/P^{i+1}$. In some cases this will be a prolongation of the $P_0$-structure $\mathcal{P}_0$ (for the rightly chosen complements). For example, this will be the case for all the geometries dealt with in the chapter four of this work, except for the case of projective structures, because $\mathcal{P}_0=P^1M$ in that case.

Now we define morphisms of Cartan geometries.

\begin{de}
A morphism of Cartan geometries $\phi$ between $(p: \mathcal{G}\to M,\omega)$ and $(p': \mathcal{G}'\to M',\omega')$ of the same type $(G,P)$ is a principal bundle morphism $\phi: \mathcal{G}\to \mathcal{G}'$ such, that $\phi^*\omega'=\omega$. 

We say that the Cartan geometry $(p: \mathcal{G}\to M,\omega)$ is locally equivalent to the Cartan geometry $(p': \mathcal{G}'\to M',\omega')$ if there is a covering $U_i$ of $M$ and injective morphisms $\phi_{U_i}$ of  Cartan geometries $(p: \mathcal{G}|_{U_i}\to U_i,\omega|_{U_i})$ and $(p': \mathcal{G}'\to M',\omega')$. We say that they are locally isomorphic if they are locally equivalent in both directions.

We denote $\operatorname{Aut}(\mathcal{G},\omega)$ the group of all automorphisms of $(p: \mathcal{G}\to M,\omega)$.

We will call an isomorphism of Cartan geometries between restrictions $\mathcal{G}|_U$ and $\mathcal{G}|_{U'}$ to two open subsets $U,\ U' $ of $M$ a local automorphism of Cartan geometry $(p: \mathcal{G}\to M,\omega)$ and we denote $\operatorname{Aut}_{loc}(\mathcal{G},\omega)$ the pseudogroup of local automorphisms.
\end{de}

Further, we define infinitesimal version of automorphisms.

\begin{de}
Let $(p: \mathcal{G}\to M,\omega)$ be a Cartan geometry of type $(G,P)$. Then a right invariant vector field $X$ on $\mathcal{G}$ is an infinitesimal automorphism of $(p: \mathcal{G}\to M,\omega)$ if $\mathfrak{L}_X\omega=0$. We denote $\operatorname{Inf}(\mathcal{G},\omega)$ the Lie algebra of all infinitesimal automorphisms of $(p: \mathcal{G}\to M,\omega)$ and we denote $\operatorname{Inf}_{loc}(\mathcal{G},\omega)$ the sheaf of locally defined infinitesimal automorphisms.
\end{de}

The following proposition will be essential for later considerations. It goes back to \cite{odk19}, see also \cite[pp. 96-98]{odk3}:

\begin{prop}\label{1.2.4}
Let $(p: \mathcal{G}\to M,\omega)$ be a Cartan geometry of type $(G,P)$ over a connected manifold $M$ and $U$ an open subset of $M$. Then the dimension of the Lie algebra $\operatorname{Inf}(\mathcal{G}|_U,\omega|_U)$ is at most $\operatorname{dim}(G)$ and $\operatorname{Aut}_{loc}(\mathcal{G},\omega)$ is generated by $\operatorname{Inf}_{loc}(\mathcal{G},\omega)$. The group $\operatorname{Aut}(\mathcal{G},\omega)$ is a Lie group with the Lie algebra consisting of all complete infinitesimal automorphisms. 
\end{prop}

Now, the left multiplication by elements $g_1,g_2\in G$ cover the same map on $G/P$ if and only if $g_1=g_2n$, where $n\in G$ is an element with trivial action on $G/P$. Such $n$ form the maximal normal subgroup $N$ of $G$ contained in $P$. The following proposition shows that there is a similar rigidity in the case of Cartan geometries, for proof look in \cite[pp. 75]{odk3}:

\begin{prop}\label{1.2.6}
Let $N$ be the maximal normal subgroup of $G$ contained in $P$. Then if $\phi_1$ and $\phi_2$ are two morphisms of the Cartan geometries $(p: \mathcal{G}\to M,\omega)$ and $(p: \mathcal{G}'\to M',\omega')$ of type $(G,P)$ covering the same diffeomorphism $M\to M'$. Then there is a smooth map $\psi: \mathcal{G}\to N$ such, that $\phi_1(u)=\phi_2(u)\cdot \psi(u)$ for all $u\in \mathcal{G}$.
\end{prop}

So we define:

\begin{de}
The Cartan geometry from the previous proposition is effective if $N$ is trivial and infinitesimally effective if $M$ is connected and $N$ is discrete.
\end{de}

In both cases, $\phi_1=\phi_2\cdot n$ for some $n\in N$ if they cover the same base map. If there is no danger of confusion, we will denote the base map of a morphism with the same symbol. 

Now we define the curvature, which characterizes each Cartan geometry.

\begin{de}
Let $(p: \mathcal{G}\to M,\omega)$ be a Cartan geometry of type $(G,P)$. The curvature form $K$ is a $\mathfrak{g}$-valued two form on $\mathcal{G}$ defined as $$K(\nu,\eta)=d\omega(\nu,\eta)+[\omega(\nu),\omega(\eta)].$$ It is not hard to verify that the curvature vanishes on vertical arguments and so we can define the curvature (function) $\kappa: \mathcal{G}\to \bigwedge^2 (\mathfrak{g}/\mathfrak{p})^*\otimes \mathfrak{g}$ by $$\kappa(u)(X,Y):=K(\omega^{-1}(X),\omega^{-1}(Y))=[X,Y]-\omega([\omega^{-1}(X),\omega^{-1}(Y)](u)).$$

We say that the geometry is torsion free if the images of $\kappa(u)$ are in $\mathfrak{p}$ and we say that the geometry is flat if $\kappa(u)=0$ for all $u\in \mathcal{G}$.
\end{de}

The definition implies the formula $$\kappa=\kappa'\circ \phi$$ for each morphism $\phi$ of Cartan geometries with curvatures $\kappa$ and $\kappa'$. This provides a strong condition on local equivalence of Cartan geometries. If the geometry is flat it leads to complete solution:

\begin{prop}\label{flat1}
Let $\omega_G$ be the Maurer-Cartan form of $G$. Then $(G\to G/P,\omega_G)$ is a flat Cartan geometry and a Cartan geometry $(p: \mathcal{G}\to M,\omega)$ of type $(G,P)$ is locally isomorphic to $(G\to G/P,\omega_G)$ if and only if $(p: \mathcal{G}\to M,\omega)$ is flat.
\end{prop}

We will say, that $(G\to G/P,\omega_G)$ is a homogeneous model (of Cartan geometry of type $(G,P)$). The automorphisms of the homogeneous model are the following:

\begin{prop}\label{flat2}
Let $G/P$ be connected. Then the automorphisms of the homogeneous model $(G\to G/P,\omega_G)$ are exactly the left multiplications by elements of $G$ and any morphism of two restrictions of the homogeneous model to open subsets can be uniquely extended to an automorphism of the homogeneous model i.e. $\operatorname{Aut}(G,\omega_G)=\operatorname{Aut}_{loc}(G,\omega_G)$.
\end{prop}

The second part of the proposition is known as the Liouville theorem and offers alternative description of a flat Cartan geometry as follows:

The proposition \ref{flat1} provides us a covering $U_i$ of $M$ and isomorphisms $\phi_{U_i}$ of the Cartan geometry restricted to $U_i$ to a restriction of the homogeneous model. Viewing it as atlas of $M$, the transition functions are then automorphisms of two restrictions of the homogeneous model. Thus they are restrictions of left multiplications by elements of $G$ by the proposition \ref{flat2}.

Conversely, suppose we have given an atlas for a manifold $M$ such, that the images of the charts are open subsets in the homogeneous model and the transition functions are restrictions of left multiplications by elements of $G$. Then we can pull back the appropriate restrictions of the homogeneous model to the domains of the charts and glue them via the isomorphism provided by left multiplications to a principal $P$-bundle over M. The resulting Cartan geometry on $M$ is flat by the construction.

\subsection{(Locally) homogeneous Cartan geometries and extensions}\label{1.3}

Now, we shall investigate the structure of homogeneous Cartan geometries. For further details and proofs look in \cite{odk3}, \cite{odk4} and \cite{odk5}.

\begin{de}
Let $(p: \mathcal{G}\to M,\omega)$ be a Cartan geometry of type $(G,P)$. We say that the geometry is $K$-homogeneous if there is a Lie subgroup $K$ of $\operatorname{Aut}(\mathcal{G},\omega)$, which acts transitively on $M$. If we denote $H$ the stabilizer of a point, then $M$ is the homogeneous space $K/H$. In the notation, we shall omit the prefix $K$ if the group is clear from the context.

We say that a Cartan geometry $(p: \mathcal{G}\to M,\omega)$ of type $(G,P)$ is locally homogeneous if the Cartan geometry is locally isomorphic to a homogeneous Cartan geometry of type $(G,P)$.
\end{de}

There is a similar description of locally homogeneous Cartan geometries $(p: \mathcal{G}\to M,\omega)$ as in the flat case.

More explicitly, there is a homogeneous Cartan geometry $(\mathcal{G}' \to K/H,\omega')$ of type $(G,P)$, a covering $U_i$ of $M$, and local isomorphisms $\phi_{U_i}$ of the Cartan geometry restricted to $U_i$ to $(\mathcal{G}' \to K/H,\omega')$, which gives an atlas of $M$ with values in $K/H$, such that the transition maps are local automorphisms of the homogeneous Cartan geometry $(\mathcal{G}' \to K/H,\omega')$.

Conversely, suppose we have given an atlas $(U_i,\phi_{U_i})$ of $M$ such, that the images of $\phi_{U_i}$ are open subsets of $K/H$ and the transition maps are local isomorphisms of the homogeneous Cartan geometry $(\mathcal{G}' \to K/H,\omega')$. Then we can use as in the flat case the pullbacks and glue them to get a locally homogeneous Cartan geometry.

Now, the following is obvious:

\begin{cor}
Let $(p: \mathcal{G}\to M,\omega)$ be a locally homogeneous Cartan geometry, then $\operatorname{Aut}_{loc}(\mathcal{G},\omega)$ acts transitively on $M$.
\end{cor}

So clearly, there is the following solution of the equivalence problem of locally homogeneous Cartan geometries:

\begin{lem}
Locally homogeneous Cartan geometries are locally equi\-valent if and only if they are locally isomorphic.
\end{lem}

Now, let $(p: \mathcal{G}\to K/H,\omega)$ be a homogeneous Cartan geometry of type $(G,P)$. Then $K$ acts on $\mathcal{G}$ by a fiber preserving action, which commutes with the principal action of $P$. Thus for a fixed $u_0\in \mathcal{G}$ such, that $p(u_0)=eH$, there is a unique Lie group homomorphism $i: H\to P$ defined by $hu_0=u_0i(h)$. Since the action of $K$ is fiber transitive, we get that $\mathcal{G}=K\times_i P$ is the associated bundle to $K\to K/H$ obtained by the action $i$ of $H$ on $P$. The insertion $j(k)=ku_0$ of $K$ to $\mathcal{G}$ induces a map $\alpha=\omega_{u_0} \circ T_ej: \mathfrak{k}\to \mathfrak{g}$. It is easy to check that the pair $(i,\alpha)$ has the following properties.

\begin{de}
Let $P$ be a Lie subgroup of a Lie group $G$ and let $H$ be a Lie subgroup of a Lie group $K$. We say, that pair a $(i,\alpha)$ is an extension of $(K,H)$ to $(G,P)$ if it satisfies:
\begin{itemize}
\item $i:H\to P$ is a Lie group homomorphism
\item $\alpha: \mathfrak{k}\to \mathfrak{g}$ is a linear mapping extending $T_ei:\mathfrak{h}\to \mathfrak{p}$
\item $\alpha$ induces a vector space isomorphism of $\mathfrak{k}/\mathfrak{h}$ and $\mathfrak{g}/\mathfrak{p}$
\item $Ad(i(h))\circ \alpha=\alpha \circ Ad(h)$ for all $h\in H$ i.e. $\alpha$ is a homomorphism of the representations $Ad(H)$ and $Ad(i(H))|_{\operatorname{Im}(\alpha)}$
\end{itemize}
\end{de}

In the case of Cartan geometry $(p: \mathcal{G}\to M,\omega)$ of type $(K,H)$, the extension $(i,\alpha)$ provides a Cartan connection $\omega_\alpha$ of type $(G,P)$ on $\mathcal{G}\times_iP$ by setting $\omega_\alpha|_{T(\mathcal{G}\times_i\{e\})}=\alpha \circ \omega$ and extending it to the entire $\mathcal{G}\times_iP$ by the $P$-action:

\begin{prop}
An extension $(i,\alpha)$ of $(K,H)$ to $(G,P)$ induces a functor from Cartan geometries of type $(K,H)$ to Cartan geometries of type $(G,P)$, mapping $( \mathcal{G}\to M,\omega)$ to $(\mathcal{G}\times_iP\to M,\omega_\alpha)$. 
\end{prop}

We denote $[k,p]$ the elements of  $K\times_iP$ represented by $(k,p)\in K\times P$.

We see, that the extension provides the structure of the (locally) homogeneous Cartan geometries
:

\begin{cor}\label{1.3.6}
A homogeneous geometry $(p: \mathcal{G}\to K/H,\omega)$ of type $(G,P)$ is given by an extension of the homogeneous model $(K\to K/H,\omega_K)$ and a locally homogeneous geometry of type $(G,P)$ is locally isomorphic to an extension of the homogeneous model $(K\to K/H,\omega_K)$. The curvature $\kappa$ is given by $$\kappa([e,e])(X,Y)=[\alpha(X),\alpha(Y)]-\alpha([X,Y])$$ for $X,Y\in \mathfrak{k}/\mathfrak{h}$.
\end{cor}

The next proposition computes the infinitesimal automorphisms of a (locally) homogeneous Cartan geometry. Our formulation expands the formula in \cite[Theorem 2.5]{odk4} in a form suitable for our computations later on.

\begin{prop}\label{lab_1}
Let $(p: \mathcal{G}\to M,\omega)$ be a (locally) homogeneous Cartan geometry. Let $U$ be a simply connected, connected neighborhood of $p(u)\in M$. Let us further identify $\operatorname{Inf}(\mathcal{G}|_U,\omega|_U)$ as a subset in $\mathfrak{g}$ using the values at $u$.  Then $\operatorname{Inf}(\mathcal{G}|_U,\omega|_U)$ is the set of $X\in \mathfrak{g}$, such that $\sum_i R^i(X)=0$, where $R^i: \mathfrak{g}\to \bigwedge^{(i+2)} \mathfrak{k}^*\otimes\mathfrak{g}$ is defined inductively as:
$$R^0(X_1,X_2)(X)=[\kappa(X_1,X_2),X]+\kappa([\alpha(X_1),X],\alpha(X_2))-\kappa([\alpha(X_2),X],\alpha(X_1))$$
 \[
\begin{split}
\iota_{X_{i+2}}(R^i(X))=&-R^{i-1}([\alpha(X_{i+2}),X]+\kappa(X,\alpha(X_{i+2})))\\
&-[\alpha(X_2),R^{i-1}(X)]-\kappa(R^{i-1}(X),\alpha(X_2)).
\end{split}
\]
Using our identification, the bracket of $X,Y\in \operatorname{Inf}(\mathcal{G}|_U,\omega|_U)$ is given by $\kappa(u)(X,Y)-[X,Y]\in \mathfrak{g}$.
\end{prop}

Also we can solve the problem of local isomorphism, and thus of the local equivalence, of the (locally) homogeneous Cartan geometries.

\begin{prop}\label{1.4.7}
Let $(i,\alpha)$ and $(\hat{i},\hat{\alpha})$ be two extension of $(K,H)$ to $(G,P)$. Then the extended geometries are locally isomorphic if and only if there are $p_0\in P$ and a Lie algebra automorphism $\sigma$ of $K$ preserving $H$ such, that $\hat{i}(\sigma(h))=p_0i(h)p_0^{-1}$ and $\hat{\alpha}=Ad_{p_0^{-1}}\circ \alpha\circ T\sigma$. The local isomorphism is given by $(k,p)\mapsto(\sigma(k),p_0p)$.
\end{prop}
\begin{proof}
Let $\Phi: K\times_i P\to K\times_{\hat{i}} P$ be a morphism of principal $P$-bundles such, that $(\Phi)^*\omega_\alpha=\omega_{\hat{\alpha}}$ and $\Phi([e,e])=[k_0,p_0]$. We know, that $\omega_\alpha(X)([k,p])=Ad(p^{-1})\circ \alpha \circ \omega(Tr^{p^{-1}} \circ Tl^{k^{-1}}(X))([e,e])$, thus $\Phi$ is uniquely determined by $\Phi([e,e])$ and $T_{[e,e]}\Phi$. If $\Phi([e,e])=[k_0,e]$, then $\hat{i}(k_0^{-1}hk_0)=i(h)$, which corresponds to an inner Lie algebra automorphism $\sigma$ of $K$ preserving $H$. If $\Phi([e,e])=[e,p_0]$, then $\hat{i}(h)=p_0i(h)p_0^{-1}$. Since $\omega_K(T_{[e,e]}\Phi \omega_K^{-1}(X))\in \mathfrak{k}$ for $X\in \mathfrak{k}$, $T_{[e,e]}\Phi$ corresponds to an automorphism $\sigma$ of $K$ preserving $H$ and the claim follows.
\end{proof}

\subsection{Examples of extensions}\label{1.4}

We introduce a few simple examples of extensions, which we will use later. First, we notice, that we can compose two extensions:

\begin{lem}\label{1.4.11}
Let $(i,\alpha)$ be an extension of $(K,H)$ to $(K',H')$ and $(\hat{i},\hat{\alpha})$ an extension of $(K',H')$ to $(G,P)$. Then $(\hat{i}\circ i, \hat{\alpha} \circ \alpha)$ is an extension of $(K,H)$ to $(G,P)$.
\end{lem}

\begin{example}\label{1.4.8}
\bf Extension to effective geometry \rm

Let $N$ be the maximal normal subgroup of $G$ contained in $P$. Then $G'=G/N$ and $P'=P/N$ are Lie groups and mapping $g\mapsto gN$ is a homomorphism of Lie groups $G\to G'$, which trivially defines extension of $(G,P)$ to $(G',P')$. Of course, the geometry of type $(G',P')$ is effective and the lemma \ref{1.4.11} implies:

\begin{cor}
Any extension of $(K,H)$ to $(G,P)$ induces extension of \linebreak $(K,H)$ to effective $(G',P')$.
\end{cor}
\end{example}

\begin{example}\label{1.4.9}
\bf Extension from simply connected models \rm

Let $K_c$ be connected, simply connected Lie group with Lie algebra $\mathfrak{k}$ and $H_c$ connected subgroup of $K_c$ with Lie algebra $\mathfrak{h}$. Let $(K,H)$ be a pair with the same Lie algebras, then $id_\mathfrak{k}|_{H_c}$ defines a Lie group homomorphism $i_c: H_c\to H$ and $(i_c,id_\mathfrak{k})$ is an extension of $(K_c,H_c)$ to $(K,H)$. Thus the homogeneous model of $(K_c,H_c)$ is locally isomorphic to homogeneous model of $(K,H)$. Further, the lemma \ref{1.4.11} implies:

\begin{cor}
Any extension of $(K,H)$ to $(G,P)$ induces extension of \linebreak $(K_c,H_c)$ to $(G,P)$.
\end{cor}
\end{example}

\begin{example}\label{1.4.10}
\bf Extension by an involution fixing a point \rm

Let $(K,H)$ be effective pair and let $\sigma$ be an involution of $K$ fixing point $eH$ in $K/H$. Consider the semidirect product $K'=K\ltimes_\sigma \mathbb{Z}_2$ and $H'=H\ltimes_\sigma \mathbb{Z}_2$. If we take $i:H\to H'$ as the inclusion and $\alpha=id_{\mathfrak{k}}$, then $(i,\alpha)$ is extension of $(K,H)$ to $(K',H')$. 

\begin{lem}
The pair $(K',H')$ is effective if and only if $\sigma|_H$ is an outer automorphism of $H$.
\end{lem}
\begin{proof}
If $\sigma(k)=h_0kh_0, \ h_0^2=e$, for some $h_0\in H$, then 
$$(g,0)(h_0,1)(g^{-1},0)=(gh_0\sigma(g^{-1}),1)=(gh_0^2g^{-1}h_0,1)=(h_0,1)$$ 
and 
$$(g,1)(h_0,1)(\sigma(g^{-1}),1)=(gh_0\sigma(g^{-1}),1)=(h_0,1).$$ Thus the group generated by $(h_0,1)$ is normal subgroup of $K'$ contained in $H'$.

If $N\cong \mathbb{Z}_2$ is normal subgroup of $K'$ contained in $H'$, it contains element of the form $(h_0,1), h_0^2=e$. Then $(gh_0\sigma(g^{-1}),1)\in N$ and $gh_0\sigma(g^{-1})=h_0$. Thus $h_0gh_0=\sigma(g)$ i.e. the involution is inner and $h_0\in H$.
\end{proof}

Now assume that $\sigma$ is an outer automorphism of $H$ and there is an extension of $(K,H)$ to $(G,P)$. Then there is an extension of $(K',H')$ to $(G,P)$ if and only if there is $g_0\in P$ such, that $g_0^2=id$ and $i(\sigma(h))=g_0i(h)g_0$. Then $i(h,1)=hg_0$ extends $i$ from $H\to P$ to $H'\to P$.
\end{example}

\begin{example}\label{1.4.12}
\bf Extensions associated to some locally homogeneous geometries \rm

Let $\mathcal{G}\to M$ be a locally homogeneous geometry of type $(G,P)$ locally isomorphic to an extension $(K\times_i P,\omega_\alpha)$ of the homogeneous model $(K\to K/H,\omega_K)$. From the description of the locally homogeneous geometries, there is an atlas $(U_i,\phi_{U_i})$ with values in $K/H$. Let us assume that the transition maps are in $\operatorname{Aut}(K,\omega_K)$. Then instead of dealing with pullbacks of $\omega_\alpha$, we can pullback $\omega_K$ and due to our assumption we can glue the pullbacks to get Cartan geometry $\mathcal{G}'\to M$ of type $(K,H)$. Then extension $(i,\alpha)$ clearly gives us a Cartan geometry of type $(G,P)$ on $M$, which is locally isomorphic to the initial one. Thus we have shown:

\begin{cor}\label{1.4.13}
A locally homogeneous Cartan geometry locally isomorphic to an extension of $(K\to K/H,\omega_K)$ such, that the transition maps are in $\operatorname{Aut}(K,\omega_K)$, is an extension of a flat Cartan geometry of type $(K,H)$.
\end{cor}
\end{example}

\subsection{Affine geometries and first order $P_0$-structures}\label{1.5}

An affine connection $\nabla$ on $M$ can be described as a principal connection form $\gamma$ on $P^1M$ i.e. as a $Gl(n,\mathbb{R})$-equivariant $\mathfrak{gl}(n,\mathbb{R})$-valued one form $\gamma$ on $P^1M$. If we denote $\theta$ the soldering form of $P^1M$, then $\omega_\nabla=\gamma+\theta$ is a one form on $P^1M$ with values in the Lie algebra $\mathfrak{a}(n,\mathbb{R}):=\mathbb{R}^n+\mathfrak{gl}(n,\mathbb{R})$ of the Lie group $A(n,\mathbb{R})=\mathbb{R}^n\ltimes Gl(n,\mathbb{R})$ of affine transformations of $\mathbb{R}^n$. The matrix form of $A(n,\mathbb{R})$ is
$$
\left(  \begin{array}{cc}   
1&0 \\ X&A
\end{array}  \right) \subset Gl(n+1,\mathbb{R}),
$$
where $X\in\mathbb{R}^n$ and $A\in Gl(n,\mathbb{R})$. It is easy to show the following (check \cite[Section 1.3.5]{odk3} for the proof):

\begin{lem}\label{1.5.1}
The one form $\omega_\nabla=\gamma+\theta$ corresponding to the affine connection $\nabla$ is a Cartan connection and $(P^1M\to M,\omega_\nabla)$ is a Cartan geometry of type $(A(n,\mathbb{R}),Gl(n,\mathbb{R}))$. The curvature $\kappa_\nabla$ of the Cartan connection has values in $\mathfrak{a}(n,\mathbb{R})$ and decomposes to torsion $T$ (the $\mathbb{R}^n$ part) and curvature $R$ (the $\mathfrak{gl}(n,\mathbb{R})$ part) of the connection $\nabla$. 

The correspondence between $\omega_\nabla$ and $\nabla$ induces an equivalence of categories of Cartan geometries of type $(A(n,\mathbb{R}),Gl(n,\mathbb{R}))$ and affine connections.
\end{lem}

Further, we define the following class of Cartan geometries:

\begin{de}
We say, that $(K,H)$ is an affine type of Cartan geometry if there is an extension of $(K,H)$ to $(A(n,\mathbb{R}),Gl(n,\mathbb{R}))$. We say that Cartan geometries of affine type are affine geometries.
\end{de}

The next proposition shows, why we call them affine geometries.

\begin{prop}\label{1.5.3}
Let $(p:\mathcal{P}\to M,\omega)$ be a Cartan geometry of type $(K,H)$. Then it is an affine geometry if and only if the pair $(K,H)$ is reductive (cf. \ref{1.1.3}). Moreover, extensions $(i,\alpha)$ of $(K,H)$ to $(A(n,\mathbb{R}),Gl(n,\mathbb{R}))$ are in one to one correspondence with choices of a frame of $\mathfrak{k}/\mathfrak{h}$. All the choices of a frame are equivalent and there is (up to equivalence) unique affine connection induced by the geometry of affine type.
\end{prop}
\begin{proof}
Let $(i,\alpha)$ be an extension of $(K,H)$ to $(A(n,\mathbb{R}),Gl(n,\mathbb{R}))$. Then, since $i(Ad(H))\subset Gl(\alpha(\mathfrak{k}/\mathfrak{h}))$, the preimage of $\alpha$ of $\mathbb{R}^n$ is $Ad(H)$-invariant complement to $\mathfrak{h}$ in $\mathfrak{k}$. Thus the pair $(K,H)$ is reductive.

Conversely, if the pair $(K,H)$ is reductive, then $Ad(H)\subset GL(\mathfrak{k}/\mathfrak{h})$ and the choice of a frame $\beta$ of $\mathfrak{k}/\mathfrak{h}$ clearly defines an extension $(i_\beta,\alpha_\beta)$ of $(K,H)$ to $(A(n,\mathbb{R}),Gl(n,\mathbb{R}))$.

Since the preimage of $\alpha$ of the standard frame of $\mathbb{R}^n$ is a frame of $\mathfrak{k}/\mathfrak{h}$ and the transition maps between two frames are in $Gl(n,\mathbb{R})$, all extensions are equivalent due to proposition \ref{1.4.7}.
\end{proof}

So as in the example \ref{1.1.3}, the choice of a frame $\beta$ of $\mathfrak{k}/\mathfrak{h}$ provides an $H$-structure on $M$ over $i_\beta: H\to GL(n,\mathbb{R})$. Thus it is obvious, that:

\begin{cor}\label{1.5.4}
Let $P_0\subset Gl(n,\mathbb{R})$ and $(p:\mathcal{P}\to M,\omega)$ be an affine geometry of type $(K,H)$. Then the choice of a frame $\beta$ of $\mathfrak{k}/\mathfrak{h}$ such that $i_\beta(H)\subset P_0$ provides a $P_0$-structure $\mathcal{P}\times_H P_0$ and induces extension $(i_\beta,\alpha_\beta)$ of $(K,H)$ to $(\mathbb{R}^n\ltimes P_0,P_0)$. The induced $P_0$-structures are equivalent if and only if the induced extensions are equivalent.
\end{cor}

We show that we can describe every first order $P_0$-structure in this way.

Let $P_0\subset Gl(n,\mathbb{R})$. Then for a first order $P_0$-structure $\mathcal{P}$, the first prolongation defines a $\mathbb{R}^n+\mathfrak{p}_0$-valued one form $\omega$ on $\mathcal{P}$. Since there is right $P_0$-action on $\mathcal{P}$, one can check that this defines a Cartan connection.

\begin{lem}\label{1.5.5}
Let $\mathcal{P}$ be a first order $P_0$-structure on $M$ and $\omega$ the one form given by the first prolongation. Then $(p:\mathcal{P}\to M,\omega)$ is a Cartan connection of affine type $(\mathbb{R}^n\ltimes P_0,P_0)$. The Cartan connection is torsion-free if and only if the $P_0$-structure is torsion-free.
\end{lem}








\newpage
\section{Symmetries of geometric structures and \\ symmetric spaces}

This chapter is devoted to symmetries of geometric structures generalizing the symmetries of symmetric spaces. The section \ref{2.1} contains definitions of (local) symmetries of geometric structures introduced in the first chapter and contains main properties of automorphisms of such symmetric geometric structures. These definitions generalize the approach to symmetric spaces introduced in \cite{odk13} in the global case and \cite{odk14} in the local case. The sections \ref{2.2} and \ref{2.3} compare our approach to symmetric spaces with other equivalent approaches, which can be found for example in \cite{odk6} and \cite{odk7}. The section \ref{2.4} is devoted to investigation of the structure of locally symmetric geometric structures and contains the main result in theorem \ref{3.1.3}, which allows us to construct all such symmetric geometries (corollary \ref{2.4.4}). The results are compilation of \cite{odk7?} and \cite{odk14}. The section \ref{2.5} is then a summary of a construction of first order geometric structures on symmetric spaces. The section \ref{2.6} is a collection of various results on affine structures on symmetric spaces from the literature and it is shown, how to obtain them using our approach.

\subsection{Symmetries and local symmetries}\label{2.1}

We are interested in special automorphisms of $P_0$-structures from the following example, which we will call symmetries:

\begin{example}\label{1.1.5}
For $x\in M$, let $S_x: M\to M$ be an automorphism of a $P_0$-structure $\mathcal{P}_0$ over $j: P_0\to Gl(n,\mathbb{R})$ such that $S_xx=x$ and $S_x^2=id_M$. Then $P^1S_x(u)=u\cdot j(p_0)$ for some $p_0\in P_0$ and $(P^1S_x)^2(u)=u\cdot j(p_0)^2=u$. Thus the action of $T_xS_x$ on $T_xM$ and its matrix $j(p_0)$ in the frame $u$ have only eigenvalues $\pm 1$. Consequently, $T_xM$ decomposes to the $\pm 1$ eigenspaces $T_x^\pm M$.
\end{example}

Now, we define symmetric $P_0$-structures:

\begin{de}\label{2.1.2}
Let $M$ be a connected smooth manifold and let $\mathcal{P}_0$ be a $P_0$-structure over $j: P_0\to Gl(n,\mathbb{R})$ on $M$. Then a smooth system of involutive symmetries (we will say only system of symmetries) on $\mathcal{P}_0$ is a smooth map $S: M\times M \to M$ satisfying the following five conditions:
\begin{itemize}
\item[GS1] If we denote $S_xy:=S(x,y)$, then $S_x$ is an automorphism of $\mathcal{P}_0$.
\item[GS2] $S_xx=x$ for all $x\in M$.
\item[GS3] $S_x^2=id_M$ for all $x\in M$.
\item[GS4] $S_x$ is an automorphism of the system of symmetries $S$ i.e.
$$S(S_xy,S_xz)=S_{S_xy}S_xz=S_xS_yz=S_xS(y,z)$$
for all $x,y,z\in M$.
\item[GS5] Let $p_0\in P_0$ be such, that $P^1S_x(u)=u\cdot j(p_0)$ for some $u$ in the fiber over $x$. Then there is no $p\in P_0,\ j(p)^2=id$ such, that the $-1$ eigenspace of $j(p_0)$ is a proper subspace in the $-1$ eigenspace of $j(p)$. 
\end{itemize}
We will say that $S_x$ is a symmetry (at $x$). Moreover,  $P_0$-structure $\mathcal{P}_0$ is said to be symmetric, if there exists non-trivial system of symmetries $S$ of $\mathcal{P}_0$. If  $(P^1M,S)$ is a symmetric $Gl(n,\mathbb{R})$-structure, then we say that $(M,S)$ is a symmetric space.
\end{de}

The axiom [GS5] restricts the possible symmetric $P_0$-structures to only those with maximal $T^-M$, otherwise any $P_0$-structure would admit trivial system of symmetries with $S_x=id_M$. If $P_0=Gl(n,\mathbb{R})$, then  [GS5] is just the condition $T_xS_x=-id_{T_xM}$, so our definition generalizes the well-known symmetric spaces as defined in \cite{odk6}.

Local version of the previous definition is the following:

\begin{de}\label{2.1.3}
Let $\mathcal{P}_0$ be a $P_0$-structure on a connected smooth manifold $M$ over $j: P_0\to Gl(n,\mathbb{R})$. Then a system of local symmetries on $\mathcal{P}_0$ is a smooth map $S: N \to M$, where $N$ is some open neighborhood of the diagonal in $M\times M$, satisfying the following five conditions:
\begin{itemize}
\item[GS1] If we denote $S_xy=S(x,y)$, then $S_x$ is an automorphism of $\mathcal{P}_0$ on the definition domain $U_x$ of $S_x$.
\item[GS2] $S_xx=x$ for all $x\in M$.
\item[GS3] $S_x^2=id_{U_x}$ for all $x\in M$.
\item[GS4] $S_x$ is a local automorphism of the system of symmetries $S$ i.e. $$S(S_xy,S_xz)=S_{S_xy}S_xz=S_xS_yz=S_xS(y,z)$$ for $x,y,z\in W$, where $W$ is some open neighborhood of diagonal in $M\times M\times M$. 
\item[GS5] Let $p_0\in P_0$ be such, that $P^1S_x(u)=u\cdot j(p_0)$ for some $u$ in the fiber over $x$. Then there is no $p\in P_0,\ j(p)^2=id$ such, that the $-1$ eigenspace of $j(p_0)$ is a proper subspace in the $-1$ eigenspace of $j(p)$.
\end{itemize}
We will say that $S_x$ is a (local) symmetry (at $x$) and that $P_0$-structure $\mathcal{P}_0$ is locally symmetric, if there exists a non-trivial system of local symmetries $S$ of $\mathcal{P}_0$.  If $(P^1M,S)$ is a locally symmetric $Gl(n,\mathbb{R})$-structure, then we say that $(M,S)$ is a locally symmetric space.
\end{de}

If $(\pi:\mathcal{P}_0\to M,S)$ is a (locally) symmetric $P_0$-structure with torsion $T$, then \begin{align*}
T_u(X,Y)&=(S_{\pi(u)})^*T_u(X,Y)=T_{u\cdot j(p_0)}(X,Y)\\&=j(p_0)(T_u(j(p_0)(X),j(p_0)(Y))).
\end{align*} Thus the torsion provides an obstruction on the possible $\pm 1$ eigenspaces of $j(p_0)$. On the other hand, even if the torsion allows existence of some suitable $p\in Gl(n,\mathbb{R})$, it does have to be contained in $j(P_0)$. For example $-id_{T_xM}\notin Sl(2n+1,\mathbb{R})$. This leads to the useful, but nearly forgotten, notion of a reflexion space introduced by O. Loos in \cite{odk13}: A  reflexion space is a pair $(M,S)$ satisfying [GS2]-[GS4]. A local version of reflexion spaces was investigated by the author in \cite{odk14}.

Now we define symmetric Cartan geometries. It is clear, that an automorphism $\phi$ of a Cartan geometry restricted to the underlying $P_0$-structure (cf. lemma \ref{1.2.2}) is an automorphism of the $P_0$-structure and the restriction of $\phi$ to $M$ is a diffeomorphism of $M$. Thus we can easily adapt the definition of a symmetric $P_0$-structures for Cartan geometries:

\begin{de}\label{2.1.4}
Let $(p: \mathcal{G}\to M,\omega)$ be a Cartan geometry of type $(G,P)$. We say that it is a symmetric Cartan geometry if there is a system of symmetries $S$ of the underlying $P_0$-structure and the symmetries $S_x$ are base maps of automorphisms of the Cartan geometry.

We say that it is a locally symmetric Cartan geometry if there is a system of local symmetries $S$ of the underlying $P_0$-structure and the local symmetries $S_x$ are base maps of local automorphisms of the Cartan geometry.
\end{de}

Since the infinitesimal automorphisms of a Cartan geometry project correctly on vector fields on $M$, we get the following corollary of proposition \ref{1.2.4}:

\begin{cor}\label{2.1.7}
The group generated by symmetries of a symmetric Cartan geometry is a Lie group. The pseudo-group generated by local symmetries is generated by a sheaf of Lie subalgebras of infinitesimal automorphisms.
\end{cor}



Since extension is a functor, we get the following important corollary:

\begin{cor}
An extension of a (locally) symmetric Cartan geometry is a (locally) symmetric Cartan geometry.
\end{cor}

Let $(\mathcal{P}_0,S)$ be a  locally symmetric $P_0$-structure, $W$ the neighborhood from the definition and consider a neighborhood  $W_x$ of $x$ such, that $W_x\times W_x\times W_x\subset W$ and $V_x=\{S_yz: y,z\in W_x\}$. The conditions [GS1]-[GS5] hold for all points of $W_x$ and [GS1]-[GS3] hold for all points of $V_x$. Thus we can differentiate those formulas along curves in $W_x$ or $V_x$. For example $$R_x(X)(y):=\frac12\frac{d}{ds}|_{s=0} S_{c(s)}S_xy,$$ where $c(s)$ is a curve in $W_x$ satisfying $c(0)=x, \ c'(0)=X$. Then $R_x$ defines a mapping from $T_xM$ to vector fields on $V_x$. We will show that $R_x(X)$ is in fact an infinitesimal automorphism on $V_x$:

\begin{prop}\label{1.1.10}
Let $X\in T_x^-M$ (cf. example \ref{1.1.5}). Then $R_x(X)$ defined in the previous paragraph is an infinitesimal automorphism of a locally symmetric $P_0$-structure $\mathcal{P}_0$, defined on $V_x$ with value $X$ at $x$. Its flow induces local one parameter subgroup $\phi_t$ generated by local symmetries and induces isomorphism of the Lie algebras of vector fields generated by the values of $R_x(X)$ and $R_{\phi_t(x)}(X)$, for  $X\in T^-M$. Moreover, if the $P_0$-structure $\mathcal{P}_0$ is symmetric, then the vector fields $R_x(X)$ are complete.
\end{prop}
\begin{proof}
Let $\phi_t$ be a local one parameter subgroup of the local diffeomorphisms given as a flow of $R_x(X)$. We show, that $\phi_t$ are generated by symmetries. Let 
$\gamma(t)=\phi_{-t}(S_{\phi_t(p)}\phi_t(q)).$ Then 
 \[ 
\begin{split}
\gamma'(t)&=\frac12\frac{d}{ds}|_{s=0}(-S_{c(s)}S_x\gamma(t)+\phi_{-t}S_{S_{c(s)}S_x\phi_t(p)}\phi_t(q)+\phi_{-t}S_{\phi_t(p)}S_{c(s)}S_x\phi_t(q))
\end{split}
\]

Differentiating $S_{c(s)}S_yz=S_{S_{c(s)}y}S_{c(s)}z$ we obtain 
\[\frac{d}{ds}|_{s=0}S_{c(s)}S_yz=\frac{d}{ds}|_{s=0} (S_{S_{c(s)}y}S_{x}z+S_{S_{x}y}S_{c(s)}z).\]

Thus 
\[ 
\begin{split}
\frac{d}{ds}|_{s=0} S_{c(s)}S_xS_yz&=\frac{d}{ds}|_{s=0} S_{c(s)}S_{S_xy}S_xz\\
&=\frac{d}{ds}|_{s=0} (S_{S_{c(s)}S_xy}S_{x}S_xz+S_{S_{x}S_xy}S_{c(s)}S_xz)\\
&=\frac{d}{ds}|_{s=0} (S_{S_{c(s)}S_xy}z+S_{y}S_{c(s)}S_xz).
\end{split}\]

Then since $\phi_t$ is flow of $R_x(X)$,
\[ 
\begin{split}
\gamma'(t)&=\frac12\frac{d}{ds}|_{s=0}(-S_{c(s)}S_x\gamma(t)+\phi_{-t}S_{c(s)}S_xS_{\phi_t(p)}\phi_t(q))\\
&=-R_x(X)(\gamma(t))+(\phi_t)^*R_x(X)(\gamma(t))=0.
\end{split}
\]

So we have shown that $\gamma(t)$ is constant and
\[ \phi_{-t}(S_{\phi_t(p)}\phi_t(q))=\gamma(0)=S_pq.\]

Differentiating $S_{c(t)}S_{c(t)}y=y$ we obtain
\[ \frac{d}{ds}|_{s=0}(S_xS_{c(s)}y+S_{c(s)}S_xy)=0.\]

Then
\[2(S_x)^*R_x(X)(y)=\frac{d}{ds}|_{s=0}S_xS_{c(s)}S_xS_xy=-\frac{d}{ds}|_{s=0}S_{c(s)}S_xy=-2R_x(X)(y).\]

Finally, since $(S_x)^*Fl^{R_x(X)}_t(x)=Fl^{(S_x)^*R_x(X)}_t$,
\[ 
S_{\phi_t(x)}S_x=\phi_tS_x\phi_{-t}S_x=\phi_t\phi_t=\phi_{2t}.
\]

Thus $R_x(X)$ is an infinitesimal automorphism and we want to evaluate it at $x$. Differentiating $S_{c(t)}c(t)=c(t)$ we obtain
\[ \frac{d}{ds}|_{s=0} (S_{c(s)}x+S_xc(s))=\frac{d}{ds}|_{s=0} c(s).\]

So if $X\in T_x^-M$, then 
\[R_x(X)(x)=\frac12\frac{d}{ds}|_{s=0} S_{c(s)}S_xx=\frac12\frac{d}{ds}|_{s=0}(c(s)-S_xc(s))=X,\]
and if $X\in T_x^+M$, then
\[R_x(X)(x)=\frac12\frac{d}{ds}|_{s=0} S_{c(s)}S_xx=\frac12\frac{d}{ds}|_{s=0}(c(s)-S_xc(s))=0.\]

Differentiating $S_{y}S_{c(s)}z=S_{S_yc(s)}S_{y}z$ we obtain
 \[\frac{d}{ds}|_{s=0}S_{y}S_{c(s)}z=\frac{d}{ds}|_{s=0} S_{S_yc(s)}S_{y}z.\]

Thus $TS_y$ maps $X\in T_x^\pm M$ to $T_{S_yx}^\pm M$ and
$$\phi_{-2t}^*(R_x(X)(y))=\frac12\frac{d}{ds}|_{s=0} S_{S_{\phi_t(x)}S_xc(s)}S_{S_{\phi_t(x)}S_xx}y=R_{\phi_{2t}(x)}(TS_{\phi_t(x)}TS_x(X))(y)$$
i.e. $\phi_{t}$ induces isomorphism of the Lie algebras generated by $R_x$ and $R_{\phi_t(x)}$.

The above computation gives us equality $$TS_xR_y(X)(S_xz)=R_{S_xy}(TS_x(X))(z),$$ which we will use later.

If the symmetries are globally defined, we extend $\phi_{t}$ from $t\in (a,b)$ to $t\in (2a,2b)$ by $\phi_{2t}=S_{\phi_{t}(x)}S_x$ and the last claim follows.
\end{proof}

\subsection{Affine locally symmetric spaces}\label{2.2}

There are several other definitions of locally symmetric spaces than \ref{2.1.3}. First we start with the following definition:

\begin{de}\label{2.1.1}
Let $M$ be a connected smooth manifold. We say that $M$ is an affine locally symmetric space if there is a torsion-free affine connection $\nabla$, whose curvature $R$ satisfies $\nabla R=0$.
\end{de}

We show that a (locally) symmetric space is an affine locally symmetric space:

\begin{prop}
Let $(M,S)$ be a locally symmetric space, then there is a unique linear torsion-free affine connection $\nabla$ invariant to all symmetries and $(M,\nabla)$ is an affine locally symmetric space.
\end{prop}
\begin{proof}
We define $\nabla_XY(x)=[R_x(X),Y](x)$ (cf. proposition \ref{1.1.10}). One can easily check, that $\nabla$ is an affine connection. Then $$((S_x)^*\nabla_X Y)(y)=TS_x(\nabla_{(S_x)^*X}(S_x)^*Y)(S_xy)=[R_{y}(X),Y](y),$$ where we used equality $(S_x)^*R_y(X)(z)=R_{S_xy}((S_x)^*(X))(z)$ (cf. proof of  proposition \ref{1.1.10}). Thus $\nabla$ is invariant to all symmetries $S_x$.

Now the curvature $R$ of the locally symmetric space satisfies $$(S_x)^*(\nabla R)=\nabla R.$$ This is a tensor field of type $(4,1)$ and $(S_x)^*$ acts there as $-id$, thus $\nabla R=0$. Since the same holds for any tensor field invariant to symmetries with odd number of arguments, the torsion and the difference between $\nabla$ and any other affine connection invariant to all symmetries vanish.
\end{proof}

The connection $\nabla$ from the proposition is called the canonical connection of the (locally) symmetric space.

As a corollary of proposition \ref{1.1.10} and lemma \ref{1.5.1} we obtain:

\begin{cor}
Let $(M,S)$ be a (locally) symmetric space. Then the corresponding affine geometry is (locally) homogeneous.
\end{cor}

Since $(S_x)^*$ act as an involution on the infinitesimal automorphisms (cf. proposition \ref{1.1.10}), the example \ref{1.4.10} and corollary \ref{1.4.13} imply:

\begin{cor}
A (locally) symmetric space is an extension of a flat Cartan geometry.
\end{cor}

We will answer the question which flat Cartan geometries have extensions to symmetric spaces in the next section.

\subsection{Homogeneous symmetric spaces and classification of symmetric spaces}\label{2.3}

Another definition of a (locally) symmetric space is the following.

\begin{de}
Let $(K,H,h)$ be the triple satisfying:
\begin{itemize}
\item $K$ is a Lie group with Lie subgroup $H$ such that $K/H$ is connected
\item $h\in H$, $h^2=id$ and $H$ is open subgroup of the centralizer of $h$ in $K$
\item the maximal normal subgroup of $K$ contained in $H$ is trivial.
\end{itemize}
Then we call this triple a homogeneous symmetric space and say that a flat Cartan geometry of type $(K,H,h)$ is a homogeneous locally symmetric space.
\end{de}

Since $H$ is open subgroup of the centralizer of $h$ in $K$, $\mathfrak{h}$ is the $1$-eigenspace of $Ad(h)$, while the $-1$-eigenspace of $Ad(h)$ can be identified with $\mathfrak{k}/\mathfrak{h}$ as the $Ad(H)$-invariant complement to $\mathfrak{h}$ in $\mathfrak{k}$.

\begin{lem}
Let $(K,H,h)$ be a homogeneous symmetric space, then \linebreak $(K,H)$ is an affine type of Cartan geometry.
\end{lem}



So assume, $(K,H)$ is an affine type of Cartan geometry. Then the extension of flat affine geometry of type $(K,H)$ to $(A(n,\mathbb{R}),Gl(n,\mathbb{R}))$ is torsion free if and only if the Lie bracket $[X,Y]\in \mathfrak{h}$ for $X,Y\in \mathfrak{k}/\mathfrak{h}$ (due to formula in corollary \ref{1.3.6}). If this is the case, then we consider $\sigma: \mathfrak{k}/\mathfrak{h}+\mathfrak{h}\to \mathfrak{k}/\mathfrak{h}+\mathfrak{h}$ defined as $-id$ on $\mathfrak{k}/\mathfrak{h}$ and $id$ on $\mathfrak{h}$. Since $\sigma([X,Y])=[\sigma(X),\sigma(Y)]$, $\sigma$ is an involutive automorphism of $\mathfrak{k}$. Now if we modify the geometry as in the examples \ref{1.4.8} and \ref{1.4.10}, then we obtain a homogeneous (locally) symmetric space:

\begin{prop}
A locally symmetric space $(M,S)$ is a homogeneous locally symmetric space and a symmetric space $(M,S)$ is a homogeneous symmetric space.
\end{prop}

Finally, we show that all our definitions of (locally) symmetric spaces are equivalent.

Assume that $(K,H,h)$ is a homogeneous symmetric space. Then the description of flat Cartan geometries in \ref{1.2} implies, that for a flat Cartan geometry of affine type $(K,H)$, there is an atlas of $M$, such that the images of the charts are open subsets in $K/H$ and the transition functions are restrictions of left multiplications by elements of $K$. Then we can define $$S_{fH}gH=fhf^{-1}gH$$ in these charts. Really, in a different chart $S_{kfH}kgH=kfh(kf)^{-1}kgH=kfhf^{-1}gH$ i.e. the $S$ is well defined. Since $S_x$ are covered by left multiplications by elements of $K$, they are automorphisms of any of the above Cartan geometries restricted to the charts, and since the pair $(K,H)$ is effective, they are unique. Then easy computations check [GS2]-[GS5] and we obtain:

\begin{prop}\label{2.1.9}
A homogeneous (locally) symmetric space is a (locally) symmetric space.
\end{prop}

Now we see, that the classification of symmetric spaces is equivalent to the classification of homogeneous symmetric spaces. We restrict ourselves to the case of semisimple symmetric spaces:

\begin{de}
We say that a homogeneous symmetric space $(K,H,h)$ is semisimple if $K$ is semisimple.
\end{de}

The semisimple homogeneous symmetric spaces are classified in \cite{odk9}:

\begin{prop}
Each semisimple homogeneous symmetric space is a finite product of symmetric spaces of the following two types (which are called simple):

\begin{itemize}
\item $((\bar{H}\times \bar{H})\ltimes \mathbb{Z}_2,H\ltimes \mathbb{Z}_2,h)$, where $\bar{H}$ is a simple Lie group, $h:=((e,e),1)$ swaps the first two factors, and $H$ is the diagonal in $\bar{H}\times \bar{H}$.

\item $(K,H,h)$ is a homogeneous symmetric space with a simple Lie group  $K$. The Lie algebras $\mathfrak{k}$ and $\mathfrak{h}$ of all possible cases can be found in the table in the appendix \ref{appA}.
\end{itemize}
\end{prop}





\subsection{Local structure of symmetric $P_0$-structures}\label{2.4}

We will continue our investigation of the structure of (locally) symmetric $P_0$-structures using notation of section \ref{2.1}.

First we notice, that we have shown that $R_x(X)(x)=0$ for $X\in T^+M$ and  $R_x(X)(x)=X$ for $X\in T^-M$  in the proof of proposition \ref{1.1.10}. Thus:

\begin{lem}
Let $(\mathcal{P},S)$ be a (locally) symmetric $P_0$-structure, then there is the decomposition of $TM$ to the two complementary distributions $T^+M$ and $T^-M$. The distribution $T^+M$ is integrable.
\end{lem}
\begin{proof}
Differentiating $S_{S_{x(s)}x}y=S_{x(s)}S_{x}S_{x(s)}y$ we obtain
\[\frac{d}{ds}|_{s=0}S_{S_{x(s)}x}y=\frac{d}{ds}|_{s=0}S_{x(s)}S_{x}S_{x}y+\frac{d}{ds}|_{s=0}S_{x}S_{x}S_{x(s)}y=2\frac{d}{ds}|_{s=0}S_{x(s)}y.\]
Then $R_x(R_x(X)(x))(y)=\frac14 \frac{d}{ds}|_{s=0} S_{S_{x(s)}x}S_xy=\frac12 \frac{d}{ds}|_{s=0} S_{x(s)}S_xy=R_x(X)(y)$. Thus $R_x(X)(y)=0$ if and only if $X\in T^+M$. 

There is an action $\mu$ of $W_x\times W_x$ on vector fields $(\xi,0)$ on $W_x\times W_x$ defined by $\mu(x,y)(\xi,0):=(R_{x}(\xi)(y),0)$. Then the vector field $(\xi,0)$ is $\mu(x,y)$-related to $(0,0)$ if and only if $\xi\in T^+W_x$. Since the bracket of $\mu$-related vector fields is $\mu$-related, $T^+M$ is an integrable distribution.
\end{proof}

We will denote $F_x$ the leaf of $T^+M$ trough $x$. Since $TS_y$ preserves $T^+M$, $S_y(F_x\cap W_x)\subset F_{S_yx}$. Since $R_y(X)=0$ for any $y\in F_x\cap W_x$ and $X\in T_yF_x$, $S_x|_{W_x}=S_y|_{W_x}$ for $y\in F_x\cap W_x$. Thus the only infinitesimal automorphisms we can obtain by composing symmetries are generated (by bracket of vector fields) by $R_x(X), \ X\in T^-M$. Let $\mathfrak{k}_x$ be the Lie algebra generated by $R_x(X), \ X\in T^-M$. Then $[R_x(X),R_x(Y)]$ is bilinear in $X,\ Y$ and the difference of $[[R_x(X),R_x(Y)],R_x(Z)]$ and $R_x([[R_x(X),R_x(Y)],R_x(Z)](x))$ is in the stabilizer of $x$, thus the difference is zero due to $(S_x)^*$ action. So we proved, that $\mathfrak{k}_x$ is finite dimensional, actually spanned by $R_x(X)$ and $[R_x(X),R_x(Y)]$.

Remind that $S_y$ is isomorphism of Lie algebras $\mathfrak{k}_x$ to $\mathfrak{k}_{S_yx}$. If we denote $N_x$ the orbit of the pseudo-group generated by symmetries trough $x$, then $\mathfrak{k}_{S_yx}(S_yx)=T_{S_yx}N_x$. So there are subgroups $P_x$ of $P_0$ such, that the restriction of the (locally) symmetric $P_0$-structure to $N_x$ is a (locally) symmetric $P_x$-structure. Thus we will investigate (locally) symmetric $P_0$-structures under one of the following conditions:

\begin{prop}\label{2.4.2}
The following conditions are equivalent for any (locally) symmetric $P_0$-structure on $M$:
\begin{itemize}
\item The (pseudo)-group generated by symmetries acts transitively on $M$.
\item $TM$ is the only integrable distribution containing $T^-M$.
\item The torsion $T(X,Y)$ of the $P_0$-structure spans $T^+M$ for $X,Y\in T^-M$.
\end{itemize} 
\end{prop}
\begin{proof}
We have shown in the previous paragraph that the (pseudo)-group generated by symmetries acts transitively on $M$ if and only if $TM$ is the only integrable distribution containing $T^-M$. Since $$T(X,Y)(x)=[R_x(X),R_x(Y)](x)$$ for $X,Y\in T^-M$ and $[R_x(X),R_x(Y)](x)\in T^+_xM$, $T(X,Y)$ spans $T^+M$ if and only if $TM$ is the only integrable distribution containing $T^-M$. 
\end{proof}

\begin{de}
We say that a (locally) symmetric $P_0$-structure is of maximal torsion, if it satisfies the above conditions.
\end{de}

Next we show, that there is a Cartan connection corresponding to a (locally) symmetric $P_0$-structure of maximal torsion.

\begin{thm}\label{3.1.3}
For each locally symmetric $P_0$-structure of maximal torsion on $M$, there is a unique flat Cartan geometry of type $(K,H)$ on $M$ such, that:
\begin{itemize}
\item $K/H$ is connected and simply connected;
\item the maximal normal subgroup of $K$ contained in $H$ is trivial;
\item there is $h\in K$ such, that $h^2=id$ and $H$ is contained in the centralizer of $h$ in $K$ and the symmetries are locally given by $S_{kH}fH=khk^{-1}fH$.
\end{itemize}
\end{thm}
\begin{proof}
The Lie second fundamental theorem provides us an atlas $(U_i,\phi_i)$ corresponding to $\mathfrak{k}_x,\ x\in M,$ satisfying the conditions from the characterization of the locally homogeneous Cartan geometries in the corollary \ref{1.4.13}. Thus there is a flat Cartan geometry on $M$ of type $(K,H)$, where $K$ is connected, simply connected Lie group locally generated by the symmetries and $H$ is the image of the connected, simply connected group with Lie algebra consisting of elements of $\mathfrak{k}_x$ vanishing at $x$ of Lie group homomorphism induced by inclusion of Lie algebras. Of course, $K/H$ is simply connected. Since $S_{eH}$ is a local automorphism, it induces an involution of $K$ and we extended the model according to the example \ref{1.4.10}. Then we can factor out the maximal normal subgroup of $K$ contained in $H$, and $K/H$ stays simply connected and $h$ is the class of $S_{eH}$. Since $K$ is locally generated by symmetries, $S_{kH}=kS_{eH}k^{-1}$ near $eH$.
\end{proof}

The locally symmetric Cartan geometries of maximal torsion are locally homogeneous and $S_x\in H$ i.e. they satisfy conditions of corollary \ref{1.4.13} and one can alway use extension from example \ref{1.4.9}. Thus:

\begin{cor}\label{2.4.4}
A locally symmetric Cartan geometry of maximal torsion is an extension of a flat Cartan geometry of type $(K,H)$ such, that:
\begin{itemize}
\item $K/H$ is connected;
\item the maximal normal subgroup of $K$ contained in $H$ is trivial;
\item there is $h\in K$ such, that $h^2=id$ and $H$ is contained in the centralizer of $h$ in $K$ and the symmetries are locally given by $S_{kH}fH=khk^{-1}fH$.
\end{itemize}
The Cartan geometry is symmetric if and only if it is extension of homogeneous model of $(K,H)$.
\end{cor}

Of course, if we denote $L$ the centralizer of $h$ in $K$ and  $N$ is the maximal normal subgroup of $K$ contained in $L$, then $(K/N,L/N,h)$ is a homogeneous symmetric space.

If we restrict the projection $\pi: K/H\to K/L$ to the images of the atlas $(U_i,\phi_i)$, then we get submersions $\pi \circ \phi_i: U_i\to K/L$. Since the vector fields $R_x(X)$ are $L$-invariant, $S\circ \pi=\pi \circ S$. Furthermore $T^+M$ is tangent to $\phi_i^{-1}(L/H)$. So we have shown:

\begin{cor}\label{2.4.6}
For a (locally) symmetric $P_0$-structure of maximal torsion on $M$, the (local) leaf space of the foliation $F_x$ tangent to $T^+M$ is a (locally) symmetric space of type $(K/N,L/N,h)$.
\end{cor}

Although $H\subset L$ can be almost arbitrary, we will investigate the existence of the $P_0$-structures only for symmetric spaces.

\subsection{Symmetric $P_0$-structures on symmetric spaces}\label{2.5}

We know that the symmetric affine geometry corresponding to a symmetric space is torsion-free and of maximal torsion. We show that the torsion is the only obstruction to be an extension of a symmetric space.

\begin{prop}\label{2.5.1}
For each locally symmetric affine geometry, the following conditions are equivalent:
\begin{itemize}
\item $T_xS_x=-id_{T_xM}$
\item the affine geometry is torsion-free and of maximal torsion
\item the affine geometry is an extension of a locally symmetric space.
\end{itemize}
\end{prop}
\begin{proof}
We have shown that $T_xS_x=-id_{T_xM}$ implies, that the affine geometry is torsion-free and of maximal torsion. Then the Cartan geometry from corollary \ref{2.4.4} is (due to torsion freeness) a locally symmetric space. Of course, $T_xS_x=-id_{T_xM}$ on each extension of a locally symmetric space.
\end{proof}

Since (local) automorphisms of a first order $P_0$-structure induce (local) automorphisms of the corresponding Cartan connection, the previous proposition and corollary \ref{1.5.4} and \ref{1.5.5} imply:

\begin{cor}\label{2.5.3}
For $P_0\subset Gl(n,\mathbb{R})$ and homogeneous symmetric space $(K,H,h)$, there is a bijection between:
\begin{itemize}
\item extensions of affine type $(K,H)$ to $(\mathbb{R}^n\ltimes P_0,P_0)$

\item frames $\beta$ of $\mathfrak{k}/\mathfrak{h}$ such, that the inclusion $i_\beta(Ad(H))$ induced be the frame $\beta$, cf. \ref{1.1.3}, is contained in $P_0$.
\end{itemize}

Every such extension of a homogeneous locally symmetric space of type $(K,H)$ is a torsion-free symmetric $P_0$-structures of maximal torsion. 

If the $P_0$-structures are of the first order, all  torsion-free symmetric $P_0$--structures of maximal torsion are constructed this way.
\end{cor}

Now the proposition \ref{1.4.7} determines the equivalence classes of such extensions.

\begin{lem}\label{2.5.4}
Two frames of $\mathfrak{k}/\mathfrak{h}$ determine the same homomorphism $i: H\to Gl(n,\mathbb{R})$ if and only if the transition map between them commutes with $i(H)$.

Two frames of $\mathfrak{k}/\mathfrak{h}$ determine equivalent $P_0$-structures if and only if the transition map between them is composition of elements of $P_0$ and outer automorphisms of the Lie group $Ad(H)$ induced by automorphisms of $K$.
\end{lem}




\subsection{Examples of $P_0$--structures on locally symmetric spaces}\label{2.6}

The following examples can be found in various literature, for example in \cite{odk18}, \cite{odk10}, \cite{odk11} and \cite{odk12}, but most of the results can be deduced from, what we have proven and the structure of semisimple symmetric spaces. The results for simple symmetric spaces are also summarized in the table in appendix \ref{appA}.

\begin{example}
Pseudo-Riemannian structures of signature $(p,q)$ ($p+q=n$) are the $O(p,q)$-structures. They are of the first order, thus we can construct all symmetric pseudo-Riemannian structures from locally symmetric spaces, which are in the literature known as the pseudo-Riemannian symmetric spaces. 

On semisimple symmetric space $(K,H,h)$, we can restrict the Killing form $B: \mathfrak{k}\otimes \mathfrak{k}\to \mathbb{R}$ to the $H$-invariant complement $\mathfrak{k}/\mathfrak{h}$ of $\mathfrak{h}$. This defines non-degenerate $Ad(H)$-invariant scalar product on $T_eK/H$ and choice of basis of $\mathfrak{k}/\mathfrak{h}$ ortonormal with respect to this scalar product defines a  pseudo-Riemannian structure on locally symmetric space of type $(K,H,h)$. Thus we have proved:

\begin{lem}
There is a pseudo-Riemannian structure  of signature  $(n-p,p)$ on each semisimple locally symmetric space of type $(K,H,h)$, where $p=\operatorname{dim}(C)-\operatorname{dim}(C\cap H)$ and $C$ is the maximal compact subgroup of $K$.
\end{lem}

The Killing form $B$ defines inclusion of $i_B: H\to Gl(n,\mathbb{R})$, which is unique up to conjugation of $O(n-p,p)$. But any other $i: H\to Gl(n,\mathbb{R})$ is conjugate to $i_B$ in $Gl(n,\mathbb{R})$, thus all possible extension $\alpha$ are determined by $i_B(H)$-invariant elements of $Gl(n,\mathbb{R})$, which we will denote $End^H(K/H)$. We will denote $Sym^H(K/H)$ or  $Asym^H(K/H)$ the $i_B(H)$-invariant elements of $Gl(n,\mathbb{R})$, which are symmetric or antisymmetric with respect to $B$, respectively.

\begin{prop}
Let $(K,H,h)$ be simple symmetric space. Then one of the following claims holds:
\begin{itemize}
\item $\mathfrak{h}$ is real, simple and $ad(\mathfrak{h})$ is irreducible on $\mathfrak{k}/\mathfrak{h}$, $$End^H(K/H)=Sym^H(K/H)\cong \mathbb{R};$$
\item $\mathfrak{h}$ is complex, simple and $ad(\mathfrak{h})$ is irreducible on $\mathfrak{k}/\mathfrak{h}$, $$End^H(K/H)=Sym^H(K/H)\cong \mathbb{C};$$
\item $\mathfrak{h}$ consist of a real, simple factor and one dimensional center and $ad(\mathfrak{h})$ has two dual factors in $\mathfrak{k}/\mathfrak{h}$, $$Asym^H(K/H)\cong Sym^H(K/H)\cong \mathbb{R},$$ and $End^H(K/H)\cong \mathbb{C}$ if the center is compact and $End^H(K/H)\cong \mathbb{R}\times \mathbb{R}$ if the center is non-compact;
\item $\mathfrak{h}$ consist of a complex, simple factor and two dimensional center and $ad(\mathfrak{h})$ has two dual factors in $\mathfrak{k}/\mathfrak{h}$, $$Asym^H(K/H)\cong Sym^H(K/H)\cong \mathbb{C},$$ and $End^H(K/H)\cong \mathbb{C}\times \mathbb{C}$;
\end{itemize}
\end{prop}
\begin{proof}
The result can be found in \cite[Chapter 5]{odk10}. But if one looks in the classification in the appendix \ref{appA}, one can see that these are the only possibilities.
\end{proof}

Notice, that the simple quaternionic case is not possible at all. Now we can determine all possible pseudo-Riemannian structures on locally symmetric spaces.

\begin{prop}\label{2.6.4}
All non-equivalent pseudo-Riemannian structures on semisimple locally symmetric spaces of type $(K,H,h)$ are in bijection with $\prod_i Sym^{H_i}(K_i/H_i)$ for all simple factors $(K_i,H_i,h)$ up to an outer automorphisms of $Ad(H_i)$ induced by an automorphisms of $K_i$.
\end{prop}
\begin{proof}
There is the unique inclusion $i_B$ of $H$ to $O(n,n-p)$ up to conjugation of  $O(n,n-p)$. Moreover, the representation $Ad(H)$ is completely reducible, thus there is frame such, that $i_B=\prod_i i_{B_i}$ of simple factors. So all possible $\alpha$ restricted to simple factors $(K_i,H_i,h)$ are determined by $X\in End^{H_i}(K_i/H_i)$. Let $\beta$ be orthonormal basis with respect to $B_i$, then $X\beta$ is orthonormal basis with respect to $B_i(X^{-1}.,X^{-1}.)$.  Since $B_i(X^{-1}.,X^{-1}.)=B_i$, the claim follows from lemma \ref{2.5.3}.
\end{proof}
\end{example}

\begin{example}
Conformal structures of signature $(p,q)$ are the $CO(p,q)$-structures, where $CO(p,q)=O(p,q)\times \mathbb{R}^+$ is the group of angle-preserving transformations of a vector space with inner product of signature $(p,q)$. They are of order two and we deal with them in the fourth chapter.
\end{example}



\begin{example}
Almost complex structures are the $Gl(n,\mathbb{C})$-structures and they are of infinite type. 

\begin{lem}
There is a complex structure on a semisimple locally symmetric space of type $(K,H,h)$ if and only if $End^H(K/H)$ is complex. 
\end{lem}
\begin{proof}
The complex structure is given by $J=\bigl( \begin{smallmatrix} 0&-E\\ E&0 \end{smallmatrix} \bigr)\in Gl(2n,\mathbb{R})$, which has to commute with $i(H)$ i.e. $J\in End^H(K/H)$. Since $J^2=-id$ we get the claim.
\end{proof}
\end{example}

\begin{example}\label{2.6.8}
Pseudo-Hermitian structures of signature $(p,q)$,  $p+q=2n$, are the $U(p,q)$-structures. They are of the first order and their classification on simple symmetric spaces is the following:

\begin{lem}
There is a pseudo-Hermitian structure on a semisimple locally symmetric space of type $(K,H,h)$  if and only if $End^H(K/H)$ is complex and $Asym^{H_i}(K_i/H_i)\neq 0$ on each simple factor.
\end{lem}
\begin{proof}
The claim follows from the fact, that the complex structure is antisymmetric with respect to metric of the pseudo-Hermitian structure.
\end{proof}
\end{example}

\begin{example}\label{2.6.10}
Straight complex pseudo-Riemannian structures are the $O(n,\mathbb{C})$-structures, where $O(n,\mathbb{C})=O(n,n)\cap Gl(n,\mathbb{C})$. They are of the first order and their classification on semisimple symmetric spaces is the following:

\begin{lem}
There is a straight complex structure on a semisimple locally symmetric space of type $(K,H,h)$  if and only if $Sym^{H_i}(K_i/H_i)$ is complex on each simple factor.
\end{lem}
\begin{proof}
The claim follows from the fact, that the complex structure is symmetric with respect to $O(n,\mathbb{C})$ metric.
\end{proof}
\end{example}

\begin{example}
Almost para--complex structures are the $Gl(n,\tilde{\mathbb{C}})$--struc\-tures, where $\tilde{\mathbb{C}}$ are the para--complex numbers, which we will view as a real two--by--two matrices $\bigl( \begin{smallmatrix} a&b\\ b&a \end{smallmatrix} \bigr)$, and $Gl(n,\tilde{\mathbb{C}})\subset Gl(2n,\mathbb{R})$. They are of infinite type. On the other hand, one can also view $\tilde{\mathbb{C}}\cong \mathbb{R}\times \mathbb{R}$ as $\bigl( \begin{smallmatrix} \frac{a+b}{2}&0\\ 0&\frac{a-b}{2} \end{smallmatrix} \bigr)$. 

\begin{lem}
There is a paracomplex structure on a semisimple locally symmetric space of type $(K,H,h)$ if and only if there is decomposition of $End^{H}(K/H)$ to isomorphic $H$-modules of the same dimension.
\end{lem}
\begin{proof}
The paracomplex structure is given by $J=\bigl( \begin{smallmatrix} E&0\\ 0
&-E \end{smallmatrix} \bigr)\in Gl(2n,\mathbb{R})$, which has to commute with $i(H)$ i.e. $J\in End^H(K/H)$ and the claim clearly follows.
\end{proof}
\end{example}

\begin{example}
Pseudo-para-Hermitian structures of signature $(n,n)$ are the $\tilde{U}(n)$-structures, where $\tilde{U}(n)$ is the paracomplex analog of $U(n)$. They are of the first order and their classification on semisimple symmetric spaces is the following:

\begin{lem}
There is pseudo-para-Hermitian structure on a semisimple locally symmetric space of type $(K,H,h)$  if and only if $End^{H_i}(K_i/H_i)$ is $\mathbb{R}\times \mathbb{R}$ or $\mathbb{C}\times \mathbb{C}$.
\end{lem}
\begin{proof}
The claim follows from the fact, that the para-complex structure is antisymmetric with respect to $\tilde{U}(n)$ metric.
\end{proof}
\end{example}

\begin{example}\label{2.6.16}
Almost symplectic structures are the $Sp(2n,\mathbb{R})$-structures and they are of infinite type.

\begin{lem}
There is a symplectic structure on a semisimple locally symmetric space of type $(K,H,h)$  if and only if $Asym^{H_i}(K_i/H_i)\neq 0$ for each simple factor.
\end{lem}
\begin{proof}
If $Asym^{H_i}(K_i/H_i)\neq 0$ for each simple factor, then the metric and (para)-complex structure defines symplectic structure.

If there is a symplectic structure, then due to complete reducibility of $Ad(H)$, there are symplectic structures on each simple factor, which have to be compatible with metric. Thus there is (para)-complex structure and the claim follows.
\end{proof}
\end{example}

\begin{example}\label{2.6.18}
Pseudo--para--quaternionic--K\"ahler structures are the \linebreak $Sp(2,\mathbb{R})\times Sp(2n,\mathbb{R})$--structures and pseudo--quaternionic--K\"ahler structures are the $Sp(1)\times Sp(p,q)$--structures. They are of order one and their classification on semisimple symmetric spaces reads as follows (for details on gradings look in the next chapter):

\begin{lem}
Let $(K,H,h)$ be a semisimple homogeneous symmetric \linebreak space. Then $Ad(H)\subset Sp(1)\times Sp(p,q)$ or $Ad(H)\subset Sp(2,\mathbb{R})\times Sp(2n,\mathbb{R})$ if and only if there is a complex contact grading $\mathfrak{k}_{i,\mathbb{C}}$ on the complexification of $\mathfrak{k}$ such, that $\mathfrak{k}_{-2,\mathbb{C}}+\mathfrak{k}_{0,\mathbb{C}}+\mathfrak{k}_{2,\mathbb{C}}$ coincide with complexification of $\mathfrak{h}$. Thus $\mathfrak{k}$ is simple and if $\mathfrak{h}$ contains $Sp(2,\mathbb{R})\cong Sl(2,\mathbb{R})$, then $Ad(H)\subset Sp(2,\mathbb{R})\times Sp(2n,\mathbb{R})$ and if $\mathfrak{h}$ contains $Sp(1)\cong SU(2)$, then $Ad(H)\subset Sp(1)\times Sp(p,q)$.
\end{lem}
\end{example}

\newpage
\section{Parabolic geometries and symmetries}

This chapter is devoted to parabolic geometries and their symmetries. Most of the general results is coming from \cite{odk3}. The section \ref{3.1} contains description of parabolic geometries as Cartan geometries of special type and their connection with gradings of semisimple Lie algebras. The section \ref{3.2} introduces the notion of infinitesimal flag structures and their relation with parabolic geometries. The section \ref{3.3} picks a suitable normalization condition to get a one to one correspondence with parabolic geometries and the underlying structures. The section \ref{3.4} then generalizes symmetries to infinitesimal flag structures and shows that we can obtain similar results as in the case of general Cartan connections in \ref{2.4}, namely, the structural results in the theorem \ref{3.4.4} and construction in the theorem \ref{3.5.1}.


\subsection{Parabolic geometries and gradings}\label{3.1}

Let $G$ be a semisimple Lie group with Lie algebra $\mathfrak{g}$. A $k$-grading of $\mathfrak{g}$ is the decomposition of $\mathfrak{g}=\mathfrak{g}_{-k}\oplus \dots \oplus \mathfrak{g}_k$ to a direct sum of vector subspaces such, that
\begin{itemize}
\item $\mathfrak{g}_{-k}\neq 0$, $\mathfrak{g}_k\neq 0$ 
\item $[X,Y]\in \mathfrak{g}_{i+j}$ for $X\in \mathfrak{g}_i$ and $Y\in \mathfrak{g}_j$, when we assume $\mathfrak{g}_i= 0$ for $|i|>k$
\item the subalgebra $\mathfrak{g}_-:=\mathfrak{g}_{-k}\oplus \dots \oplus \mathfrak{g}_{-1}$ is generated (as a Lie algebra) by $\mathfrak{g}_{-1}$
\end{itemize}

The grading $\mathfrak{g}_i$ defines a filtration $\mathfrak{g}^i:=\mathfrak{g}_{i}\oplus \dots \oplus \mathfrak{g}_{k}$ of $\mathfrak{g}=\mathfrak{g}^{-k}$ and we denote $\mathfrak{p}:=\mathfrak{g}^0=\mathfrak{g}_{0}\oplus \dots \oplus\mathfrak{g}_{k}$ and $\mathfrak{p}_+:=\mathfrak{g}^1$. We say that a Lie subgroup $P$ of $G$ is the parabolic subgroup corresponding to the grading $\mathfrak{g}_i$ of $\mathfrak{g}$ if $\mathfrak{p}$ is the Lie algebra of $P$. In particular, the adjoint action of $p\in P$ preserves the filtration $\mathfrak{g}^i$ and there is a Lie subgroup $G_0$ of $P$ consisting of $p\in P$ that preserves the gradation $\mathfrak{g}_i$. Of course, the Lie algebra of $G_0$ is $\mathfrak{g}_0$. Moreover, $G_0\ltimes \mathfrak{p}_+$ is diffeomorphic to $P$ via $(g,X)\mapsto g\cdot \exp(X)$.

\begin{de}
Let $G$ be a semisimple Lie group and $P$ its parabolic subgroup corresponding to the grading $\mathfrak{g}_i$ of $\mathfrak{g}$. Then a Cartan geometry $(p: \mathcal{G}\to M,\omega)$ of type $(G,P)$ is called a parabolic geometry of type $(G,P)$.
\end{de}

We can identify $TM=\mathcal{G}\times_P \mathfrak{g}/\mathfrak{p}$ and the kernel of the action of $P$ on $\mathfrak{g}/\mathfrak{p}$ is $G_{k}:=\exp(\mathfrak{g}_{k})$, thus the structure group from lemma \ref{1.2.2} is $P_0:=P/G_{k}$ and there is the underlying $P_0$-structure $\mathcal{G}/G_{k}\to M$.  So we see, that the parabolic geometries are not of affine type and they are of second order if they are prolongation of the underlying $P_0$-structure, because the kernel of the action of $P$ on $\mathfrak{g}/\mathfrak{g}_k$ is trivial. They are infinitesimally effective, if there is no simple ideal of $\mathfrak{g}$ in $\mathfrak{g}_0$. Since according to example \ref{1.4.8} we can always pass to effective example, we will always assume that $(G,P)$ is effective.

But $P_+=\exp(\mathfrak{p}_+)$ is also a normal subgroup of $P$, so there is a finer structure induced by the parabolic geometry. In particular, the orbit space $\mathcal{G}_0:=\mathcal{G}/P_+$ is a principal $G_0$-bundle over $M$ and the filtration $\mathfrak{g}^i$ induces a filtration $T^i\mathcal{G}:= \omega^{-1}(\mathfrak{g}^i)$ of $T\mathcal{G}$ by smooth subbundles, which descends to filtrations $T^i\mathcal{G}_0$ of $T\mathcal{G}_0$ and $T^iM$ of $TM$. The Cartan connection $\omega$ descends for each $i$ to a smooth section $\omega^0_i$ of the bundle $\Omega^1(T^iM,\mathfrak{g}_i)$ of partially defined differential forms.

If we define $gr_i(TM):=T^iM/T^{i+1}M$, then we see, that $\mathcal{G}_0\times_{G_0}\mathfrak{g}_i \cong gr_i(TM)$ via mapping $(u,X)\mapsto [T_up (\sigma)]$, where $\sigma\in T^i_u\mathcal{G}_0$ is any vector satisfying $\omega^0_i(u)(\sigma)=X$ and $[\ ]$ denotes the class in $gr_i(TM)$. Together $gr(TM)\cong \mathcal{G}_0\times_{Ad(G_0)}\mathfrak{g}_-$.

Now there might be two Lie brackets on $gr(TM)$. The first Lie bracket is always induced point-wise by the Lie bracket of $\mathfrak{g}_-$. Moreover,  if $[\sigma,\nu]$ is a section $T^{i+j}\mathcal{G}_0$ for all sections $\sigma$ of $T^i\mathcal{G}_0$ and $\nu$ of $T^j\mathcal{G}_0$, there is the second Lie bracket $\mathfrak{L}(X,Y)$ induced by bracket of vector fields on $TM$ composed with projections on $gr(TM)$. These brackets can be different. Let us remind that the manifold $M$ is called filtered manifold if $\mathfrak{L}$ is defined. We are interested in the following case:

\begin{de}
We say that a parabolic geometry is regular if $M$ is a filtered manifold and the two above Lie brackets coincide i.e.
$$\omega^0_{i+j}(u)([\sigma,\nu])=[\omega^0_i(u)(\sigma),\omega^0_j(u)(\nu)]$$ for all $i,j<0$ such that $i+j\geq -k$, all sections $\sigma$ of $T^i\mathcal{G}_0$ and $\nu$ of $T^j\mathcal{G}_0$ and all $u\in \mathcal{G}_0$.
\end{de}

Let $(p: \mathcal{G}\to M,\omega)$ be a regular parabolic geometry of type $(G,P)$. Let $gr(P^1M)$ be the bundle of frames of $gr(TM)$ compatible with gradation. Then, similarly to the forms $\omega^0_i$, the soldering form of $P^1M$ descends to partially defined forms on $gr(P^1M)$, and our forms $\omega^0_i$ can be understood as graded soldering forms as follows:

Since $gr(TM)\cong \mathcal{G}_0\times_{Ad(G_0)}\mathfrak{g}_-$, the $gr(P^1M)$ is a principal $\operatorname{Aut}_{gr}(\mathfrak{g}_-)$-bundle, where $\operatorname{Aut}_{gr}(\mathfrak{g}_-)$ is the group of grading preserving automorphisms of $\mathfrak{g}_-$. Clearly $gr(P^1M)\cong \mathcal{G}_0\times_{Ad(G_0)}\operatorname{Aut}_{gr}(\mathfrak{g}_-)$ and $\mathcal{G}_0\times_{Ad(G_0)}{Ad(G_0)}\subset gr(P^1M)$ i.e. $\mathcal{G}_0$ is an analog of $G_0$-structure for $gr(P^1M)$.

\subsection{Regular infinitesimal flag structures}\label{3.2}

There is an abstract version of the underlying structure that we found in the previous section. Again $G$ will be a semisimple Lie group and $P$ a parabolic subgroup of $G$ corresponding to the grading $\mathfrak{g}_i$ of $\mathfrak{g}$.

\begin{de}
A regular infinitesimal flag structure of type $(G,P)$ consists of:
\begin{itemize}
\item a filtration $T^iM$ of $TM$ on a filtered manifold $M$ such, that the bundle $(gr(TM),\mathfrak{L})$ of Lie algebras $(gr(T_xM),\mathfrak{L}_x)$ is locally trivial and modeled on $\mathfrak{g}_-$
\item a morphism of principal bundles $\iota: \mathcal{G}_0\to gr(P^1M)$ over $Ad: G_0 \to \operatorname{Aut}_{gr}(\mathfrak{g}_-)$.
\end{itemize}

A morphism between regular infinitesimal flag structures $(T^iM,\mathcal{G}_0)$ and $(T^iM',\mathcal{G}_0')$ of type $(G,P)$ is a filtration preserving diffeomorphism $f: M\to M'$ such, that $$gr(P^1f)(\iota(\mathcal{G}_0))=\iota'(\mathcal{G}_0').$$
\end{de}

In the previous section, we associated the regular infinitesimal flag structure to a regular parabolic geometry, and so we have summarized the following result (see \cite{odk3} for more details).

\begin{prop}
The mapping sending $(p: \mathcal{G}\to M,\omega)$ to $(T^iM,\mathcal{G}_0)$ is a functor between a category of regular parabolic geometries of type $(G,P)$ and a category of regular infinitesimal flag structures of type $(G,P)$.
\end{prop}

Moreover, since we assume that $(G,P)$ is effective, a morphism of regular infinitesimal flag structure induces unique morphism of Cartan geometries according to proposition \ref{1.2.6}.

Now let us consider a regular infinitesimal flag structure $(T^iM,\mathcal{G}_0)$ of type $(G,P)$. Then a choice of an isomorphism of $\mathcal{G}_0\times_{G_0}\mathfrak{g}_-=gr(TM)\cong TM$ allows to extend the graded soldering form $(\omega^0_{-k},\dots,\omega^0_{-1})$ to the one form $\theta$ such, that each $\theta_j$ vanishes on the complementary subspace to $T^jM$. Then a choice of a principal connection $\gamma$ on $\mathcal{G}_0$ defines the Cartan connection $\omega_0=\gamma+\theta$ on $\mathcal{G}_0$ of affine type $(P^{op},G_0)$, where $P^{op}$ is the Lie subgroup of $G$ diffeomorphic to $\mathfrak{g}_-\ltimes G_0$ via $(X,g)\mapsto \exp(X)\cdot g$. Then any extension of $(P^{op},G_0)$ to $(G,P)$, which preserves the filtration $\mathfrak{g}^i$, provides a Cartan geometry of type $(G,P)$. For example, the inclusion $P^{op}\subset G$ provides such an extension. So we see, that there are many Cartan connections corresponding to a regular infinitesimal flag structure and we want to choose a suitable representant under some normalization conditions.

\subsection{Normal parabolic geometries}\label{3.3}

Let $(p: \mathcal{G}\to M,\omega)$ be a regular parabolic geometry of type $(G,P)$. Let us denote $AM:=\mathcal{G}\times_P \mathfrak{g}$ the adjoint tractor bundle. Then we can view the curvature $\kappa: \mathcal{G}\to \bigwedge^2 (\mathfrak{g}/\mathfrak{p})^*\otimes \mathfrak{g}$ as a element of $\Omega^2(M,AM)$ of differential two forms with values in $AM$. Of course, the filtration $\mathfrak{g}^i$ induces a filtration $A^iM$ of $AM$. The following holds (see \cite[3.1.8]{odk3}):

\begin{lem}
A parabolic geometry $(p: \mathcal{G}\to M,\omega)$ of type $(G,P)$ is regular if and only if $\kappa(T^iM,T^jM)\subset A^{i+j+1}M$ and the parabolic geometry is torsion free if and only if $\kappa(TM,TM)\subset A^{0}M$.
\end{lem}

In \cite[3.1.10]{odk3}, one can find the description of the parabolic geometries inducing the same regular infinitesimal flag structure, and this description can be related with cohomology groups of certain chain complexes naturally associated to the pair $(G,P)$ and appearing in the prolongation process in \ref{1.1}. This offers a natural normalization condition, which uses the Kostant codifferential $\partial^*: \bigwedge^p(\mathfrak{g}/\mathfrak{p})^*\otimes \mathfrak{g}\to \bigwedge^{p-1}(\mathfrak{g}/\mathfrak{p})^*\otimes \mathfrak{g}$. Since we are interested in the case $p=2$, where we will use the following formula for $\partial^*$:

\begin{lem}\label{3.3.2}
Let $\phi\in \bigwedge^2(\mathfrak{g}/\mathfrak{p})^*\otimes \mathfrak{g}$. Let $X_i\in \mathfrak{g}$ be such that $X_i+\mathfrak{p}$ is a frame of $\mathfrak{g}/\mathfrak{p}$ and let $Z_i$ be the dual (with respect to the Killing form of $\mathfrak{g}$) frame of $\mathfrak{p}_+$. Then for all $X\in \mathfrak{g}$ is:
$$(\partial^*\phi)(X+\mathfrak{p})=2\sum_i [Z_i,\phi(X,X_i)]-\sum_i\phi([Z_i,X],X_i).$$
\end{lem}

Then the normalization condition is:

\begin{de}
We say that a parabolic geometry $(p: \mathcal{G}\to M,\omega)$ of type $(G,P)$ is normal if $\partial^*\kappa(u)=0$ for all $u\in \mathcal{G}$.
\end{de}

For a normal parabolic geometry we can project $\kappa(u)$ to $$\kappa_H(u)\in \operatorname{Ker}(\partial^*)/\operatorname{Im}(\partial^*).$$ In fact it can be shown, that $\kappa_H\in \mathcal{G}_0\times_{G_0}H^2(\mathfrak{g}_-,\mathfrak{g})$, where $H^2(\mathfrak{g}_-,\mathfrak{g})$ is the second cohomology group, which can be easily computed. We call $\kappa_H$ the harmonic curvature and it can be shown, that it contains all informations about the whole curvature. In particular:

\begin{lem}
A normal regular parabolic geometry $(p: \mathcal{G}\to M,\omega)$ of type $(G,P)$ is flat if and only if $\kappa_H$ vanishes, and is torsion free if and only if $\kappa_H$ has values only in $\mathfrak{p}$.
\end{lem}

The result of \cite{odk3} is that we can associate a normal regular parabolic geometry to a regular infinitesimal flag structure (we note that we assume, that $(G,P)$ is effective): 

\begin{prop}
There is a normal regular parabolic geometry for a given regular infinitesimal structure. If the grading does not contain factor corresponding to projective or contact projective structures, then there is an equivalence between a category of the normal regular parabolic geometries of type $(G,P)$ and a category of the regular infinitesimal flag structure of type $(G,P)$.
\end{prop}

Since any homogeneous parabolic geometry is an extension of a flat Cartan geometry, there is the following corollary:

\begin{cor}
For any homogeneous regular parabolic geometry, there is a homogeneous normal regular parabolic geometry over the same regular infinitesimal flag structure. In particular, if there is an extension of a flat Cartan geometry to a parabolic geometry, there is extension to a normal geometry.
\end{cor}

For projective or contact projective structures, we need more data to distinguish between different structures and we need to have more assumptions on morphisms of the regular infinitesimal flag structure in order to be automorphisms of the parabolic geometry. In particular, there are distinguished classes of connections and the morphisms  have to preserve these classes of connections, as we will see later.

\subsection{Symmetries of parabolic geometries}\label{3.4}

We can generalize the definition \ref{2.1.2} for regular infinitesimal flag structures, the modified conditions are indicated by '.

\begin{de}
We say that a regular infinitesimal flag structure \linebreak $(T^iM,\mathcal{G}_0)$ of type $(G,P)$ is symmetric if there is a smooth map $S: M\times M \to M$ satisfying the following five conditions:
\begin{itemize}
\item[GS1'] $S_x$ is an automorphism of $(T^iM,\mathcal{G}_0)$.
\item[GS2] $S_xx=x$ for all $x\in M$.
\item[GS3] $S_x^2=id_M$ for all $x\in M$.
\item[GS4] $S_x$ is an automorphism of the system of symmetries $S$ i.e.
$$S(S_xy,S_xz)=S_{S_xy}S_xz=S_xS_yz=S_xS(y,z)$$
for all $x,y,z\in M$.
\item[GS5'] Let $g_0\in G_0$ be such, that $gr(P^1S_x)(u)=u\cdot g_0$ for some $u$ in the fiber over $x$. Then there is no $g\in G_0,\ g^2=id$ such, that the $-1$ eigenspace of $Ad(g_0)$ is a proper subspace in the $-1$ eigenspace of $Ad(g)$. 
\end{itemize}

We define locally symmetric regular infinitesimal flag structures using the same localization as in definition \ref{2.1.3}.
\end{de}

We note that the maximality condition [GS5'] heavily restricts, which $M$ can be symmetric, because in all cases of interest there are always a nontrivial $g_0\in G_0$ such, that $g_0^2=id$. Since the symmetries $S_x$ are morphisms of the underlying regular infinitesimal flag structure, they define morphisms of the underlying $P_0$-structure. If they are automorphisms of the corresponding normal parabolic geometry (which are unique due to our assumption that $(G,P)$ is effective), then the geometry is (locally) symmetric according to the definition \ref{2.1.4}. In particular, we can say that a parabolic geometry is (locally) symmetric if the underlying regular infinitesimal flag structure is (locally) symmetric.

\begin{cor}
Assume, that the grading does not contain factor corresponding to projective or contact projective structures. Then there is the unique (locally) symmetric normal Cartan geometry of type $(G,P)$ equivalent to a given (locally) symmetric regular infinitesimal flag structure of type $(G,P)$.
\end{cor}

Now, since the parabolic geometry is regular, the condition \linebreak $\kappa(T^iM,T^jM)\subset A^{i+j+1}M$ together with the condition, that $\mathfrak{g}_{-1}$ generates $\mathfrak{g}_{-}$, imply, that the only integrable distribution containing $T^{-1}M$ is $TM$. Thus there is the following proposition equivalent to the proposition \ref{2.4.2}.

\begin{prop}\label{3.4.3}
The following conditions are equivalent for any (locally) symmetric parabolic geometry $(p: \mathcal{G}\to M,\omega)$ of type $(G,P)$:
\begin{itemize}
\item The (pseudo)-group generated by symmetries acts transitively on $M$.
\item $TM$ is the only integrable distribution containing $T^-M$.
\item $T_{p(u)}p \circ \omega_u^{-1}([X,Y]-\tau(u)(X,Y))$ spans  $T^{-1}M\cap T_{p(u)}^+M$ for $u\in \mathcal{G}$ and $X,Y\in \mathfrak{g}$ such, that $T_{p(u)}p \circ \omega_u^{-1}(X),T_{p(u)}p \circ \omega_u^{-1}(Y)\in T_{p(u)}^-M$, where $\tau$ is the torsion of the parabolic geometry.
\end{itemize} 
\end{prop}
\begin{proof}
Since the axioms $[GS2]-[GS4]$ holds for the (locally) symmetric parabolic geometries, the first two conditions are equivalent due to proposition \ref{2.4.2}. The only thing we have to show is that the expression  $T_{p(u)}p\circ \omega_u^{-1}([X,Y]-\tau(u)(X,Y))$ is equivalent to the bracket of vector fields. That follows directly from the regularity of the geometry and the detailed proof can be found in \cite[Proposition 3.1.8]{odk3}.
\end{proof}

We again say that geometries satisfying the above conditions are of maximal torsion.

Finally, the following theorem follows from the corollary \ref{2.4.4}:

\begin{thm}\label{3.4.4}
Let $(p: \mathcal{G}\to M,\omega)$ be a locally symmetric regular normal parabolic geometry of type $(G,P)$ of maximal torsion. Then $(p: \mathcal{G}\to M,\omega)$ is a locally homogeneous Cartan geometry, in particular the pseudo-group of local diffeomorphisms generated by the local symmetries acts transitively on $M$. Moreover, there are unique $K$ and $H$ satisfying the following:
\begin{itemize}
\item $(p: \mathcal{G}\to M,\omega)$ is an extension of flat Cartan geometry of type $(K,H)$;
\item $K/H$ is connected, simply connected;
\item the maximal normal subgroup of $K$ contained in $H$ is trivial;
\item there is $h\in K$ such, that $h^2=id$ and $H$ is contained in  the centralizer of $h$ in $K$ and the symmetries are locally given by $S_{kH}fH=khk^{-1}fH$.
\end{itemize}
\end{thm}

We can summarize the general construction of symmetric parabolic geometries in the next theorem:

\begin{thm}\label{3.5.1}
Let $K$ be a Lie group with Lie subgroup $H$ satisfying:
\begin{itemize}
\item there is $h\in H$ such, that $h^2=e$;
\item $H$ is contained in the centralizer of $h$ in $K$ and $K/H$ is connected;
\item the maximal normal subgroup of $K$ contained in $H$ is trivial.
\end{itemize}
Let $(i,\alpha)$ be an extension of $(K,H)$ to a parabolic geometry of type $(G,P)$ such, that there is no $p\in P, p^2=e$ such that the $-1$ eigenspace of $Ad(i(h))$ in $\mathfrak{g}$ is proper subspace of the $-1$ eigenspace of  $Ad(p)$ in $\mathfrak{g}$. Then the extension of a flat geometry of type $(K,H)$ is a locally symmetric parabolic geometry of type $(G,P)$. It is a symmetric parabolic geometry if and only if it is an extension of the homogeneous model of $(K,H)$.
\end{thm}
\begin{proof}
It is enough to show, that there is a system of symmetries on $K\to K/H$ given by the multiplication by elements of $K$. Let us define $$S_{fH}gH:=fhf^{-1}gH.$$
Since $H$ commutes with $h$, $S$ is well a defined smooth map $M\times M\to M$. It is easy computation to check, that $S$ satisfies [GS2]-[GS4]. Due to our assumptions, the $[GS5]$ is satisfied. Since the extension is a functor, we get the claim.
\end{proof}






\newpage
\section{Symmetric parabolic geometries I. - AHS-structures}

In this chapter, we look on the case of one graded parabolic geometries. These geometries were investigated by L. Zalabova in \cite{odk16} under more general assumptions on their symmetries. In fact, we can use our construction to construct most of such geometries. In the section \ref{4.1} we summarize our construction in the case of one graded parabolic geometries in the theorem \ref{4.1.4}, and we show under which assumption we construct all of them, see the proposition \ref{4.1.2}. In the sections \ref{4.2}, \ref{4.3}, \ref{4.4} and \ref{4.5} we deal with non-flat symmetric one graded parabolic geometries in detail and provide classification of those with semisimple groups (locally) generated by symmetries. 


\subsection{One graded parabolic geometries and symmetries}\label{4.1}

Firstly, we summarize basic results on one graded parabolic geometries from \cite{odk3}. These structures are called almost hermitian symmetric structures (AHS-structures shortly).

Let $G$ be a semisimple Lie group and $P$ parabolic subgroup corresponding to the grading $\mathfrak{g}=\mathfrak{g}_{-1}\oplus \mathfrak{g}_0\oplus \mathfrak{g}_1$. The structure of one gradings is the following according to \cite{odk3}:

\begin{prop}\label{4.1.1}
Let $\mathfrak{g}=\mathfrak{g}_{-1}\oplus \mathfrak{g}_0\oplus \mathfrak{g}_1$ be one graded. Then
\begin{itemize}
\item $\mathfrak{g}$ is sum of one graded simple Lie algebras $\mathfrak{g}^{(j)}$;
\item the decomposition of $\mathfrak{g}_0$-module $\mathfrak{g}_{-1}$ to irreducible components is given by $\mathfrak{g}_{-1}=\oplus_j \mathfrak{g}_{-1}^{(j)}$;
\item the only non isomorphic gradings of simple complex and real algebras of non exceptional type are those in appendix \ref{appB};
\item the maximal compact subgroup $K$ of $G$ acts transitively on $G/P$ and we denote $L:=K\cap P\subset G_0$ the stabilizer of a point.
\end{itemize}
\end{prop}

We will denote $\sigma$ the involution of $\mathfrak{g}$ defined as $(-1)^i$ on $\mathfrak{g}_i$. We know that, we can always extend $(G,P)$ with $\sigma$ so that $\sigma$ becomes inner involution, according to example \ref{1.4.10}. Now $\sigma$ restricts to an involution of $K$ with fixed point set $L$. If $\sigma$ is not an inner involution of $K$, then we extend $K$ as in example \ref{1.4.10} and, finally, we obtain symmetric space $(K/N,L/N,h)$, where $N$ is the maximal normal (discrete) subgroup of $K$ contained in $L$. If $G$ was complex, then $(K/N,L/N,h)$ is a Hermitian symmetric space (cf. example \ref{2.6.8}) and we obtain all Hermitian symmetric spaces this way. So, indeed, we can view AHS-structures as generalization of Hermitian symmetric spaces.

The following proposition is a result similar to the proposition \ref{2.5.1}.

\begin{prop}\label{4.1.2}
The following conditions are equivalent for each locally symmetric AHS-structure:
\begin{itemize}
\item the involution $\sigma$ of $\mathfrak{g}$ (defined as $(-1)^i$ on $\mathfrak{g}_i$) is inner i.e. $\sigma=Ad(g_0)$ for some $g_0\in G_0$;
\item the AHS-structure is torsion-free and of maximal torsion;
\item the AHS-structure is an extension of a locally symmetric space.
\end{itemize}
\end{prop}
\begin{proof}
There are only two possible components of harmonic curvature of the one graded parabolic geometries in general dimensions. Torsion $\tau: \mathfrak{g}_{-1}\times \mathfrak{g}_{-1}\to \mathfrak{g}_{-1}$ and curvature $\zeta: \mathfrak{g}_{-1}\times \mathfrak{g}_{-1}\to \mathfrak{g}_{0}$ (except some low dimensional geometries, where the harmonic component is of the type $\mathfrak{g}_{-1}\times \mathfrak{g}_{-1}\to \mathfrak{g}_{1}$). If $Ad(g_0)=\sigma$, then due to the maximality condition [GS5] is $T^-M=T^{-1}M=TM$. Then $-\tau(X,Y)=\tau(-X,-Y)=\tau(X,Y)=0$. Thus the parabolic geometry is torsion-free and of maximal torsion. 

A torsion-free Cartan geometry of maximal torsion from corollary \ref{3.4.4} is a locally symmetric space, due to propositions \ref{3.4.3}, \ref{2.4.2} and torsion-freeness. 

For each extension $(i,\alpha)$ of a locally symmetric space $(K,L,h)$, there are some $g_0\in G_0$ and $Z\in \mathfrak{g}_1$ such that  $i(h)=g_0\exp(Z)$. Since $\underline{Ad}(i(h))=Ad(g_0)$ acts as $-id$ on $T^{-1}M$, it acts as $-id$ on $\mathfrak{g}_{-1}$. Then we obtain that $Ad(g_0)=\sigma$  from the bracket generating property and assumption that $(G,P)$ is effective.
\end{proof}

So we will assume, that there is $g_0\in G_0$ such, that $Ad(g_0)=\sigma$. Since we assume, that $(G,P)$ is effective, $g_0$ is unique.

The inclusion of the maximal compact subgroup $K$ to $G$ provides extension of $(K,L)$ to $(G,P)$. Since $L$ is the centralizer of $g_0$ in $K$, we can define (local) symmetries (cf. theorem \ref{3.5.1}) in charts in $K/L$ instead of in $G/P$, where it is not possible, because $P$ does not commute with $g_0$. Thus:

\begin{cor}
Every flat AHS-structure of type $(G,P)$ is locally symmetric.
\end{cor}

Now we describe, how the extensions $(i,\alpha)$ of a homogeneous symmetric space $(K,L,h)$ to a one graded parabolic geometry $(G,P)$ can look like.

Generally, $i(h)=g_0\exp(Z)$ for $g_0\in G_0$ and $Z\in \mathfrak{g}_1$, and $$e=i(h)^2=g_0\exp(Z)g_0\exp(Z)=g_0^2\exp(Ad(g_0)Z)\exp(Z)=g_0^2\exp(-Z+Z)$$ so really $Ad(g_0)=\sigma$. Now, we change the extension by conjugation by $\exp(\frac12Z)\in P$ and get equivalent extension according to proposition \ref{1.4.7}. Then $i(h)=\exp(\frac12Z)g_0\exp(Z)\exp(-\frac12Z)=g_0\exp(Ad(g_0)(\frac12Z))\exp(\frac12Z)=g_0$, thus we can assume that $i(h)=g_0$ without loss of generality.

Since $h$ commutes with $L$, the equality $$g_0p_0\exp(Y)=p_0\exp(Y)g_0$$ has to be satisfied for all $p_0\exp(Y)\in i(L)$. Thus $Y=0$ for all $p_0\exp(Y)\in i(L)$ and $i(L)\subset G_0$.

Now $\alpha(Ad(h)X)=Ad(i(h))\alpha(X)$ for $X$ in the $-1$ eigenspace of $Ad(h)$, thus $$-\alpha(X)=Ad(g_0)\alpha(X).$$  Let us decompose $\alpha(X)=\alpha(X)_{-1}+\alpha(X)_0+\alpha(X)_1$ according to the grading of $\mathfrak{g}$.
Then the comparison of both sides provides us restriction $\alpha(X)_{0}=0.$

Altogether, we have proven the following:

\begin{thm}\label{4.1.4}
Let $(\mathcal{G}\to M,\omega)$ be a (locally) symmetric AHS-structure.  Then $(\mathcal{G}\to M,\omega)$ is (up to equivalence) an extension $(i,\alpha)$ of a (locally) symmetric space $(K,L,h)$ such that
\begin{itemize}
\item $i(h)=g_0$;
\item $i(L)\subset G_0$;
\item $\alpha$ is given for $X$ in the $-1$ eigenspace of $Ad(h)$ as follows:

$\alpha(X)_{-1}$ is an arbitrary isomorphism of the adjoint representations \linebreak $Ad(L)$ and $Ad(i(L))$, $\alpha(X)_0=0$, and $\alpha(X)_{1}$ is an arbitrary morphism of the adjoint representations.
\end{itemize}
\end{thm}

There is a similar result to  the corollary \ref{2.5.3} and lemma \ref{2.5.4}.

\begin{cor}\label{4.1.5}
For one graded parabolic geometry of type $(G,P)$ and homogeneous symmetric space $(K,L,h)$, there is a bijection between:
\begin{itemize}
\item extensions of $(K,L)$ to $(G,P)$ such, that $i(h)=g_0$

\item couples $\beta, b_2$, where $\beta$ is a frame of  the $-1$ eigenspace of $Ad(h)$ such, that the inclusion $i_\beta(Ad(L))$ (cf. \ref{1.1.3}) induced by the frame $\beta$ is contained in $G_0$, and $b_2$ is an endomorphism of $\mathfrak{g}_1$ commuting with $i_\beta(Ad(L))$.
\end{itemize}

Two frames of the $-1$ eigenspace of $Ad(h)$ determine the same homomorphism $i: L\to G_0$ if and only if the transition map between them commutes with $i(L)$.

For fixed frame $\beta$, all endomorphisms of $\mathfrak{g}_1$ commuting with $i_\beta(Ad(L))$ always determine the same underlying regular infinitesimal flag structure.

Two frames of the $-1$ eigenspace of $Ad(h)$ determine equivalent parabolic geometries of type $(G,P)$ (except projective structures)  if and only if the transition map between them is composition of elements of $P$ and outer automorphisms of the Lie group $Ad(L)$ induced by automorphisms of $K$.
\end{cor}

If the only possible component of the harmonic curvature is the torsion, then the torsion-free AHS-structures of maximal torsion are flat i.e. the only possible non-flat AHS-structures (with non-complex simple $\mathfrak{g}$) are projective, conformal, quaternionic and para-qutermionic structures. We investigate them in the next sections.

\subsection{Projective structures}\label{4.2}

The projective structures correspond to the following grading of $\mathfrak{g}=\mathfrak{sl}(n+1,\mathbb{R})$:
$$
\left( \begin{array}{cc}
a & Z  \\
X & A  
  \end{array} \right),
$$
where $A\in \mathfrak{gl}(n,\mathbb{R})$, $a=-tr(A)$, $X\in \mathbb{R}^n$ and $Z\in (\mathbb{R}^n)^*$.

The corresponding effective homogeneous model has $G=PGl(n+1,\mathbb{R})$ and
$$
g_0=\left( \begin{array}{cc}
-1 & 0  \\
0 & E  
  \end{array} \right).
$$

These parabolic geometries are special, because they are not given by the underlying regular infinitesimal flag structure. According to \cite{odk3}, projective structures are equivalent to a chosen projective class $[\nabla]$ of affine connections on $M$. Precisely, for any two $\nabla,\ \nabla'$ in the projective class, there is the one form $\Upsilon$ on $M$ satisfying $\nabla_XY-\nabla'_XY=\Upsilon(X)Y+\Upsilon(Y)X$. Automorphisms of projective structures are then diffeomorphisms $\phi$ such, that $(\phi)^*\nabla$ is an affine connection in the projective class.

We show that there is a unique projective structure on any (locally) symmetric space. In fact, it is the projective class of the canonical connection $\nabla$ of the (locally) symmetric space. Thus in view of proposition \ref{4.1.2} all (locally) symmetric projective geometries, are of this form.

\begin{prop}
There is (up to equivalence) unique regular normal projective structure on any locally symmetric space. The canonical connection $\nabla$ of the (locally) symmetric space is in the corresponding projective class.
\end{prop}
\begin{proof}
Let $(K,L,h)$ be a homogeneous symmetric space. The choice of any frame $\beta$ of  the $-1$ eigenspace of $Ad(h)$ provides $i_\beta: Ad(L)\to Gl(n,\mathbb{R})=G_0$. Then $\alpha$ is of the form:

$$
\left( \begin{array}{cc}
A & X   
  \end{array} \right)\mapsto
\left( \begin{array}{cc}
-tr(A)k & (b_2X)^T  \\
X & A-tr(A)(1-k)E
  \end{array} \right),
$$
where $X$ is in $-1$ eigenspace of $Ad(h)$, $A\in ad(\mathfrak{l})$, $b_2$ is a matrix commuting with $i_\beta(Ad(L))$ and $k$ real number (depending on the dimensions) such, that the $tr(A)$ has the right action.

We know from corollary \ref{1.4.7} that all $i_\beta$ are equivalent and the extensions are equivalent if the matrices $b_2$ represent the same endomorphism.

Now for fixed frame $\beta$, we compute $b_2$ from the normality conditions using formula and notation from lemma \ref{3.3.2} on $\kappa([e,e])(X_1,X_2)$:
\[
\begin{split}
0&=(\partial^*\kappa)([e,e])(a_iX^i+\mathfrak{p})=\sum_i [Z_i,[\alpha(a_jX^j),\alpha(X^i)]-\alpha([a_jX^j,X^i])]\\
&=\sum_{i,j} a_j([Z_i,[X^j+b_2(X^j),X^i+b_2(X^i)]-\alpha([X^j,X^i])])\\
&=\sum_{i\neq j}a_j([Z_i,[X^j,b_2(X^i)]-\alpha([X^j,X^i])]),
\end{split}
\]
where $X^i$ is vector in $\mathfrak{g}_{-1}$ with $1$ on i-th row and rest $0$, $Z^i\in \mathfrak{g}_{1}$ is covector with $1$ on i-th column and rest $0$. So we get system of linear equations and we know, there always has to be at least one solution. The homogeneous part is $\sum_{i\neq j}a_j([Z_i,[X^j,b_2(X^i)])=0$ and, because $\mathfrak{g}$ is semisimple, zero is the only possible solution of the homogeneous part. Thus there is unique $b_2$ and the first part of the claim follows.

Consider a parabolic geometry $(\mathcal{G}\to M,\omega)$ of projective type. Then any global $G_0$ equivariant section of $\mathcal{G}_0\to \mathcal{G}$ pulls $\omega$ back to a $\mathfrak{g}$-valued one form and since its image in $\mathfrak{g}_0$ is $G_0$-invariant, it defines an affine connection $\nabla$. It can be shown that all such $\nabla$ form the projective class of affine connections. Since we are in the case of homogeneous parabolic geometry and $i(L) \subset G_0$, $\mathcal{G}_0=K\times_i G_0$ and we define a global section of $\mathcal{G}=K\times_i P$ as inclusion $(k,g)\mapsto (k,g)$. This section is invariant to all symmetries, thus the corresponding torsion-free affine connection is invariant to all symmetries i.e. it is the canonical connection of the (locally) symmetric space.
\end{proof}

\subsection{Conformal structures}\label{4.3}

The conformal structures correspond to the following grading of $\mathfrak{g}=\mathfrak{so}(p+1,q+1)$:
$$
\left( \begin{array}{ccc}
a & Z & 0 \\
X & A & -I_{p,q}Z^T\\
0 & -X^TI_{p,q} & -a
  \end{array} \right),
$$
where $A\in \mathfrak{so}(p,q)$, $a\in \mathbb{R}$, $X\in \mathbb{R}^{p+q}$, $Z\in (\mathbb{R}^{p+q})^*$ and is diagonal matrix $I_{p,q}$ with $\pm 1$ on diagonal according to the signature $(p,q)$.

The corresponding effective model has $G=PO(p+1,q+1)$ and
$$
g_0=\left( \begin{array}{ccc}
-1 & 0 & 0 \\
0 & E & 0\\
0 & 0 & -1
  \end{array} \right).
$$

We show that there is a conformal structure on any semisimple symmetric space. Thus in view of proposition \ref{4.1.2} all (locally) symmetric conformal geometries, whose symmetries (locally) generate semisimple groups, are of this form.

\begin{prop}
There is a regular normal conformal structure on any semisimple (locally) symmetric space. The equivalence classes of such structures are parametrized by torus of dimension equal to the number of simple factors with straight complex structure (cf. example \ref{2.6.10}).
\end{prop}
\begin{proof}
Let $(K,L,h)$ be a semisimple homogeneous symmetric space. We know from the proposition \ref{2.6.4},  what are the possible inclusions of $i_\beta: Ad(L)\to O(p,q)\subset G_0$. If there is no straight complex structure on any factor, then we know from corollary \ref{4.1.5} that all $i_\beta$ are equivalent, because $\beta$ are unique up to real multiples on each simple factors and $G_0=CO(p,q)$. If there is a straight complex structure on some factor, then the additional freedom in the choice of $i_\beta$ on this factor is given by complex multiples. Thus the possible mappings $\alpha$ are:
$$
\left( \begin{array}{cc}
A & X   
  \end{array} \right)\mapsto
\left( \begin{array}{ccc}
0 & (b_2X)^T & 0 \\
b_1X & A & -I_{p,q}b_2X\\
0 & -(b_1X)^TI_{p,q} & 0
  \end{array} \right),
$$
where $A\in ad(\mathfrak{l})$ and $X$ is in $-1$ eigenspace of $Ad(h)$, $b_2$ is matrix commuting with $Ad(i_\beta(L))$ and $b_1$ is a rotation in $\operatorname{Re}(X), \operatorname{Im}(X)$ plane on any straight complex factors and identity elsewhere and the claim again follows from  the corollary \ref{4.1.5}.
\end{proof}

\subsection{Quaternionic structures}\label{4.4}

The quaternionic structures correspond to the following grading of $\mathfrak{g}=\mathfrak{sl}(n+1,\mathbb{H})$:
$$
\left( \begin{array}{cc}
a & Z \\
X & A 
  \end{array} \right),
$$
where $A\in \mathfrak{gl}(n,\mathbb{H})$, $a\in \mathbb{H}$, $\operatorname{Re}(a)+\operatorname{Re}(tr(A))=0$, $X\in \mathbb{H}^{p+q}$ and $Z\in (\mathbb{H}^{p+q})^*$.

The corresponding effective model has $G=PGl(n+1,\mathbb{H})$ and
$$
g_0=\left( \begin{array}{cc}
-1 & 0 \\
0 & E
  \end{array} \right).
$$

\begin{example}
Quaternionic structure on $SO^*(2n+2)/ SO^*(2)\times SO^*(2n)$.

If we look in the table in appendix B, $SO^*(2n)$ acts by a quaternionic representation i.e. there is $i: SO^*(2n)\to Gl(n,\mathbb{H})$. Further $SO^*(2)$ acts by multiples of $-k\in Sp(1)$ from left i.e. $ SO^*(2)\times SO^*(2n)$ sits in $G_0:=P(Sp(1)\times Gl(n,\mathbb{H}))$. In fact, we immediately get a flat quaternionic structure on $SO^*(2n+2)/ SO^*(2)\times SO^*(2n)$ just by inclusion of $SO^*(2n+2)$ to $PGl(n+1,\mathbb{H})$.
\end{example}

We show that there is a quaternionic structure on any pseudo--quater\-nionic--K\"ahler symmetric space (cf. example \ref{2.6.18}) and and there are no quaternionic structures on other semisimple symmetric spaces except the previous example. Thus in view of proposition \ref{4.1.2} all (locally) symmetric quaternionic geometries, whose symmetries (locally) generate semisimple groups, are of this form.

\begin{prop}
There is a regular normal quaternionic structure on any pseudo--quaternionic--K\"ahler locally symmetric space and there are no quaternionic structures on other semisimple locally symmetric spaces except $SO^*(2n+2)/ SO^*(2)\times SO^*(2n)$. 
\end{prop}
\begin{proof}
Let $(K,L,h)$ be a semisimple homogeneous symmetric space and assume that the image of $i$ is contained in $Gl(n,\mathbb{H})$. Then the representation of $Ad(i(L))$ is of quternionic type and there is no such in the classification of semisimple symmetric spaces. The same is true in the case that the image of $i$ is contained in the part given by $a\in \mathbb{H}$. So the image of $i$ has intersection with both parts, but this implies that the representation of $Ad(i(L))$ is irreducible and going through the list of simple symmetric spaces we check that the only possibilities are pseudo-quaternionic-K\"ahler symmetric spaces (where $Sp(1)\times Sp(p,q)$ trivially sits in $G_0$) and the previous example. 
\end{proof}

\subsection{Para-quaternionic structures}\label{4.5}

The para-quaternionic structures correspond to the following grading of $\mathfrak{g}=\mathfrak{sl}(n+2,\mathbb{R})$:
$$
\left( \begin{array}{cc}
a & Z \\
X & A 
  \end{array} \right),
$$
where $A\in \mathfrak{gl}(n,\mathbb{R})$, $a\in \mathfrak{gl}(2,\mathbb{R})$, $\operatorname{Re}(tr(a))+\operatorname{Re}(tr(A))=0$, $X\in \mathbb{R}^{n}\otimes (\mathbb{R}^{2})^*$ and $Z\in \mathbb{R}^{2}\otimes (\mathbb{R}^{n})^*$.

The corresponding effective model has $G=PGl(n+2,\mathbb{R})$ and
$$
g_0=\left( \begin{array}{ccc}
-1 & 0& 0 \\
0 & -1& 0 \\
0 & 0 & E
  \end{array} \right).
$$

\begin{example}
Para-quaternionic structures on $SO(k+1,l+1)/SO(1,1)\times SO(k,l)$ and  $SO(k,l+2)/SO(2)\times SO(k,l)$.

First we notice that $SO(n+2)/SO(2)\times SO(n)$ is equivalent to the homogeneous model of $(G,P)$,  according to the proposition \ref{4.1.1}. Since $SO(2)\times SO(n)$ sits in $G_0:=P(Gl(2,\mathbb{R})\times Gl(n,\mathbb{R}))$, it does not matter, which signature the matrices have and we immediately get a flat para-quaternionic structure on $SO(k+1,l+1)/SO(1,1)\times SO(k,l)$ and  $SO(k,l+2)/SO(2)\times SO(k,l)$ just by inclusion.
\end{example}

We show that there is a para-quaternionic structure on any pseudo-para-quaternionic-K\"ahler symmetric space (cf. example \ref{2.6.18}) and there are no para-quaternionic structures on other  semisimple symmetric spaces except the previous examples. Thus in view of proposition \ref{4.1.2} all (locally) symmetric para-quaternionic geometries, whose symmetries (locally) generate semisimple groups, are of this form.

\begin{prop}
There is a regular normal para-quaternionic structure on any pseudo-para-quaternionic-K\"ahler locally symmetric space and there are no para-quaternionic structures on other semisimple locally symmetric spaces except $SO(k+1,l+1)/SO(1,1)\times SO(k,l)$ and  $SO(k,l+2)/SO(2)\times SO(k,l)$.
\end{prop}
\begin{proof}
Let $(K,L,h)$ be a semisimple homogeneous symmetric space and assume that the image of $i$ is contained in $Sl(n,\mathbb{R})$. Then the representation of $Ad(i(L))$ decomposes to two copies of standard representation of $\mathfrak{sl}(n,\mathbb{R})$ and there is no such in the classification of semisimple symmetric spaces. The same is true in the case that the image of $i$ is contained in the part given by $a\in \mathfrak{gl}(2,\mathbb{R})$. So the image of $i$ has intersection with both parts, but this implies that the representation of $Ad(i(L))$ is irreducible and going through the list of simple symmetric spaces we check that the only possibilities are pseudo-para-quaternionic-K\"ahler symmetric spaces (where $Sp(2,\mathbb{R})\times Sp(2n,\mathbb{R})$ trivially sits in $P(Gl(2,\mathbb{R})\times Gl(2n,\mathbb{R}))$) and those in previous example. 
\end{proof}

\newpage
\section{Symmetric parabolic geometries II. \\--  Parabolic contact geometries}

We investigate the symmetric parabolic contact geometries in this chapter. The symmetries in this case are point-wise investigated in \cite{odk17}, again under more general assumptions. The results here are published in \cite{odk5}. The section \ref{5.1} contains main results on structure of symmetric parabolic contact geometries in the general case in proposition \ref{5.1.1} and in the semisimple case theorem \ref{5.2.1}. The details on construction in the semisimple case are contained in the proposition \ref{pp29}. The sections \ref{5.2}, \ref{5.3}, \ref{5.4} and \ref{5.5} deal with non-flat symmetric parabolic contact geometries in detail and provide classification of those with semisimple group (locally) generated by symmetries, assuming there is no straight complex factor. The section \ref{5.6} contains remarks on geometric interpretation of some examples in preceding sections.


\subsection{Parabolic contact geometries and symmetries}\label{5.1}

Let $G$ be a semisimple Lie group and $P$ parabolic subgroup corresponding to a grading of $\mathfrak{g}=\mathfrak{g}_{-2}+\mathfrak{g}_{-1} +\mathfrak{g}_{0}+\mathfrak{g}_{1}+\mathfrak{g}_{2}$ such, that $ \operatorname{dim}(\mathfrak{g}_{\pm 2})=1$ and the Lie bracket $ \mathfrak{g}_{-1}\times \mathfrak{g}_{-1}\to \mathfrak{g}_{-2}$ is non-degenerate. Then the parabolic geometries of type $(G,P)$ are called parabolic contact geometries.  The full list of parabolic contact geometries of non-exceptional type is in appendix \ref{appB}.

The following proposition is a similar result to the proposition \ref{2.5.1}.

\begin{prop}\label{5.1.1}
For each regular locally symmetric parabolic contact geometry $\mathcal{G}\to M$ of type $(G,P)$, the following conditions are equivalent:
\begin{enumerate}
\item the involution $\sigma$ of $\mathfrak{g}$ defined as $(-1)^i$ on $\mathfrak{g}_i$ is inner i.e. $\sigma=Ad(g_0)$ for some $g_0\in G_0$ and, at the same time, $T^-M=T^{-1}M$;
\item the parabolic contact geometry is torsion-free and of maximal torsion;
\item the parabolic contact geometry is an extension of a flat Cartan geometry, which is locally a fibre bundle over symmetric space with one dimensional fiber.
\end{enumerate}
\end{prop}
\begin{proof}
There are only two possible components of harmonic curvature of the parabolic contact geometries in the general dimensions. Torsion $\tau: \mathfrak{g}_{-1}\times \mathfrak{g}_{-1}\to \mathfrak{g}_{-1}$ and curvature $\zeta: \mathfrak{g}_{-1}\times \mathfrak{g}_{-1}\to \mathfrak{g}_{0}$ (except some low dimensional geometries, where the harmonic component is of the type $\mathfrak{g}_{-2}\times \mathfrak{g}_{-1}\to \mathfrak{g}_{1}$).

If $Ad(g_0)=\sigma$ and $T^-M=T^{-1}M$, then $-\tau(X,Y)=\tau(-X,-Y)=\tau(X,Y)=0$. Thus 1. implies 2.

The corresponding torsion-free Cartan geometry of maximal torsion from corollary \ref{3.4.4} is a flat Cartan geometry, which is locally fibre bundle over symmetric space with one dimensional fiber, due to proposition \ref{3.4.3} and corollary \ref{2.4.6}.

Let $(i,\alpha)$ be an extension of a flat Cartan geometry of type $(K,H,h)$ such, that $L/H$ is one dimensional, where $L$ is the centralizer of $h$ in $K$. Then there are some $g_0\in G_0$ and $Z\in \mathfrak{g}_1\oplus \mathfrak{g}_2$ such that $i(h)=g_0\exp(Z)$. Since $\underline{Ad}(i(h))=Ad(g_0)$ acts as $-id$ on $T^{-1}M$, it acts as $-id$ on $\mathfrak{g}_{-1}$. Then we obtain , that $Ad(g_0)=\sigma$, again from the bracket generating property. Thus 3. implies 1. and the proof is concluded.
\end{proof}

Consequently, many types of the symmetric parabolic contact geometries are immediately flat, because the torsion is the only component of the harmonic curvature (see Appendix \ref{appB}):

\begin{cor}
Extensions to torsion-free locally symmetric regular normal parabolic contact geometries of maximal torsion with non-trivial curvature have $\mathfrak{g}$ equal to $\mathfrak{sl}(n,\mathbb{R})$, $\mathfrak{su}(p,q)$ or $\mathfrak{sp}(2n,\mathbb{R})$.
\end{cor}

As in the one graded case, we will assume that there is $g_0\in G_0$ such, that $Ad(g_0)=\sigma$. Then the following holds for the same reasons as in the case of one graded geometries.

\begin{cor}
Every flat parabolic contact geometry of type $(G,P)$ is locally symmetric.
\end{cor}

In the following theorem, we characterize all extensions to torsions-free symmetric regular parabolic contact geometries of maximal torsion with semisimple group (locally) generated by symmetries.

\begin{thm}\label{5.2.1}
Let $K/H$ be a connected homogeneous space such, that
 \begin{itemize}
\item there is $h\in K$ such, that $h^2=id$, $H$ is contained in the centralizer $L$ of $h$ in $K$ and $\operatorname{dim}(L/H)=1$;
\item the maximal normal subgroup of $K$ contained in $H$ is trivial;
\item $K$ is semisimple.
\end{itemize} 
Let $(i,\alpha)$ be (up to equivalence) an extension of $(K,H)$ to a regular parabolic contact geometry of type $(G,P)$ with $\mathfrak{g}$ equal to $\mathfrak{sl}(n,\mathbb{R})$, $\mathfrak{su}(p,q)$ or $\mathfrak{sp}(2n,\mathbb{R})$. Then:
\begin{enumerate}
\item $i(h)=g_0$, $\alpha(\mathfrak{l})\subset \mathfrak{g}_{-2}+ \mathfrak{g}_{0}+ \mathfrak{g}_{2}$, $\alpha(\mathfrak{k}/\mathfrak{l})\subset \mathfrak{g}_{-1}+ \mathfrak{g}_{1}$;

\item $\kappa(\mathfrak{g}_{-1},\mathfrak{g}_{-1})\subset \mathfrak{g}_{0}+\mathfrak{g}_{2}$, $\kappa(\mathfrak{g}_{-1},\mathfrak{g}_{-2})\subset \mathfrak{g}_{1}$ and $\kappa(\mathfrak{g}_{-2},\mathfrak{g}_{-2})=0$, i.e. $\alpha$ restricted to $\mathfrak{l}$ is a Lie algebra homomorphism;

\item $i(H)\subset G_0$, the geometry $K/H$ is reductive and $\mathfrak{l/h}$ is in center of $\mathfrak{l}$;

\item the semisimple symmetric space $(K/N,L/N,h)$ has only pseudo--hermi\-tian or para--pseudo--hermitian simple factors, where $N$ is the maximal normal  (discrete) subgroup of $K$ contained in $L$;

\item $H$ has trivial intersection with center of $L$ restricted to real factors and at most one dimensional  intersection, when restricted to (straight) complex factors;

\item the normality conditions are $$\sum_i [Z_i,\kappa(X^{-1},X_i)]=0,\ \sum_i \kappa([Z_i,X^{-2}],X_i)=0$$ for any $X=X^{-1}+X^{-2}\in \mathfrak{g}_{-1}+\mathfrak{g}_{-2}$, where $X_i$ is frame of $\mathfrak{g}/\mathfrak{p}$ and $Z_i$ frame of $\mathfrak{p}_+$  dual to $X_i$ (with respect to the Killing form of $\mathfrak{g}$).
\end{enumerate}
\end{thm}
\begin{proof}
Clearly, $\alpha(\mathfrak{k}/\mathfrak{l})\subset \mathfrak{g}^{-1}$ and we can assume, that $i(h)=g_0$, for the same reasons as in the case of one graded geometries. Then the first claim is result of computation with $Ad(i(h))$-action.

The second claim is consequence of $Ad(i(h))$--action, regularity and \linebreak torsion--freeness. Last part follows, because $\mathfrak{g}_{-2}$ is one dimensional.

Now, $\alpha(\mathfrak{h})\subset \mathfrak{g}_{0}+\mathfrak{g}_{2}$. For each parabolic contact geometry, $\mathfrak{g}_{-2}+\mathfrak{g}_{2}$ generates subalgebra $\mathfrak{z}$ isomorphic to $\mathfrak{sl}(2,\mathbb{R})$ or $\mathfrak{su}(2)$, and these are the only parts of $\alpha(\mathfrak{l})\subset \mathfrak{g}_{-2}+\mathfrak{g}_{0}+\mathfrak{g}_{2}$ with nontrivial action on $\mathfrak{g}_{-2}$. Since $\mathfrak{k}$ is semisimple, $\mathfrak{l}$ contains only semisimple or abelian simple factors. We investigate all possible cases of $\alpha(\mathfrak{l})\cap \mathfrak{z}$:

a) $\alpha(\mathfrak{l})\cap \mathfrak{z}$ is nilpotent, then the third claim holds.

b) $\alpha(\mathfrak{l})\cap \mathfrak{z}=\mathfrak{z}$. Thus preimage of $\mathfrak{z}$ contains subalgebra isomorphic to $\mathfrak{z}$. Then since $\mathfrak{z}$ is not factor of $\mathfrak{k}$, the root space in $\mathfrak{z}\cap \mathfrak{h}$ has nontrivial action on $\mathfrak{k}/\mathfrak{l}$ and its image in $\mathfrak{g}_{2}$ has trivial action. Contradiction.

c) $\alpha(\mathfrak{l})\cap \mathfrak{z}$ is solvable. Since $\mathfrak{l}$ does not contain solvable factors, there is subalgebra of $\mathfrak{l}$ isomorphic to $\mathfrak{z}$ with an solvable subalgebra mapped onto $\alpha(\mathfrak{l})\cap \mathfrak{z}$. The image of root space in $\mathfrak{z}\cap \mathfrak{h}$ maps $\mathfrak{g}_{-2}$ to $\mathfrak{g}_{0}$. Contradiction.

The Lie bracket $\mathfrak{g}_{-1}\times \mathfrak{g}_{-1}\to \mathfrak{g}_{-2}\cong \mathfrak{l}/\mathfrak{h}$ is a non-degenerate antisymmetric bilinear map, which is $Ad(P)$-invariant and $Ad(\exp(\mathfrak{g}_{-2}))$ acts trivially on $\mathfrak{g}_{-1}$ and $\mathfrak{g}_{-2}$. Thus the Lie bracket is $ad(\alpha(\mathfrak{l}))$-invariant. The curvature $\kappa(\mathfrak{g}_{-1},\mathfrak{g}_{-2})\subset \mathfrak{g}_{1}$ measures  the difference between the $ad(\alpha(\mathfrak{l}))$-action and the action induced by left multiplication by elements of $L$. Thus the Lie bracket $\mathfrak{g}_{-1}\times \mathfrak{g}_{-1}\to \mathfrak{g}_{-2}$ is $L$-invariant and $L$ acts trivially on the $\mathfrak{g}_{-2}$-part. If we compose it with $\alpha$, then we obtain a non-degenerate antisymmetric $L$-invariant bilinear form on the symmetric space $(K/N,L/N,h)$ i.e. a symplectic form (up to choice of scale) and we know, from example \ref{2.6.16}, that only the semisimple symmetric spaces in the claim are possible.

The preimage of $\mathfrak{g}_{-1}$ in the $-1$ eigenspace of $Ad(H)$ generates $\mathfrak{l/h}$ by the Lie bracket and the brackets coincide due to regularity. Thus $H$ has trivial intersection with center of $L$, if the center is one dimensional, or at most one dimensional, if  the center is two dimensional.

We will use notation and formula from lemma \ref{3.3.2} in the last claim. So for $X_i\in \mathfrak{g}_{-2}$ we obtain $[Z_i,X]\in \mathfrak{p}$ and $\kappa([Z_i,X],X_i)=0$. For $X_i\in \mathfrak{g}_{-1}$ and for $X \in \mathfrak{g}_{-1}$ we get $\kappa([Z_i,X],X_i)=0$. So $\sum_i \kappa([Z_i,X],X_i)\in \mathfrak{g}_{0}+\mathfrak{g}_{-2}$. For $X_i\in \mathfrak{g}_{-2}$ we obtain $[Z_i,\kappa(X,X_i)]\in \mathfrak{g}_{1}$. For $X_i\in \mathfrak{g}_{-1}$ and $X\in \mathfrak{g}_{-2}$ we get $[Z_i,\kappa(X,X_i)]\in \mathfrak{g}_{0}+\mathfrak{g}_{2}$. For $X_i\in \mathfrak{g}_{-1}$ and $X\in \mathfrak{g}_{-1}$ we get $[Z_i,\kappa(X,X_i)]\in \mathfrak{g}_{1}$. Then normality conditions looks like as in the proposition.
\end{proof}

So we are interested in construction of such geometries. The construction starting with semisimple symmetric space with only pseudo-hermitian or para-pseudo-hermitian simple factors (cf. \ref{2.6.8}) is summarized in the following proposition.

\begin{prop}\label{pp29}
Let $(K,L,h)$ be a semisimple symmetric space with only pseudo-hermitian or para-pseudo-hermitian simple factors. Let $h\in H \subset L$ be the subgroup of dimension $\operatorname{dim}(L)-1$, whose Lie algebra contains the semisimple part of $\mathfrak{l}$. Let $(G,P)$ be a parabolic contact geometry with $\mathfrak{g}$ equal to $\mathfrak{sl}(n,\mathbb{R})$, $\mathfrak{su}(p,q)$ or $\mathfrak{sp}(2n,\mathbb{R})$.

Let $i: H\to G_0$ be a Lie group homomorphism such that the adjoint representations of $H$ on $\mathfrak{k}/\mathfrak{h}$ and $i(H)$ on $\mathfrak{g}_{-1}$ are isomorphic. 

Let $\alpha$ have the following components:

1) $i'$ on $\mathfrak{h}$ with values in $\mathfrak{g}_{0}$

2) induced by the isomorphism of adjoint representations on $\mathfrak{k}/\mathfrak{h}$ with values in $\mathfrak{g}_{-1}$ and induced by some morphism of adjoint representations on $\mathfrak{k}/\mathfrak{h}$ with values in $\mathfrak{g}_{1}$

3) $\alpha$ is arbitrary on $\mathfrak{l/h}$ with values in $\mathfrak{g}_{-2}$ (non-zero), $\mathfrak{g}_{2}$ or in the centralizer of $i'(\mathfrak{h})$ in $\mathfrak{g}_{0}$.

Then $(i,\alpha)$ is an extension to a symmetric parabolic contact geometry of type $(G,P)$ and all extensions $\alpha$ (for fixed $i$) are of this form. 

If $K$ is simple and non-complex, then the extended geometry is regular if and only if the value of $\alpha$ in $\mathfrak{g}_{-2}$ is determined by the Lie bracket on $\mathfrak{g}_{-1}$.

If $K$ is semisimple and without (straight) complex factors, then the extended geometry is regular if and only if the value of $\alpha$ in $\mathfrak{g}_{-2}$ is determined by the Lie bracket on $\mathfrak{g}_{-1}$ on one simple factor and does not depend on the choice of the simple factor.
\end{prop}
\begin{proof}
Since the decomposition $\mathfrak{k}=\mathfrak{h}+\mathfrak{l/h}+\mathfrak{k}/\mathfrak{h}$ is $Ad(H)$-invariant the $\alpha$ is well-defined. The first three conditions from definition for $\alpha$ to be extension hold by definition of $\alpha$; the last one holds, because the adjoint representations are identified by $\alpha$. Defining $\alpha$ in another way breaks some of the defining conditions of the extension. 

The last two claims follows, because we have shown in the proof of the previous theorem, that the Lie bracket on $\mathfrak{g}_{-1}$ coincides with the Lie bracket generating $\mathfrak{l/h}$ on each real simple factor.
\end{proof}

The regularity in the complex case is more difficult and we will not deal with this case.

We will need the following proposition to show, how many $i:H\to P$ can exist up to equivalence.

\begin{prop}\label{pp30}
Let $P$ be one of $Sl(n,\mathbb{R})$, $SU(p,q)$ or $Sp(2n,\mathbb{R})$ and let $H$ be a semisimple Lie group. Let $i,\ j: H\to P$ be two homomorphisms of Lie groups with discrete kernels such, that restrictions of standard representations $\mathbb{R}^n$ to $i(H)$ and $j(H)$ are isomorphic and irreducible. Then there is $C\in P$ such, that $i(k)=Cj(k)C^{-1}$ for all $k\in H$.
\end{prop}
\begin{proof}
We will use the general concept described in \cite{odk16?}. After complexification to $P_{\mathbb{C}},\ H_{\mathbb{C}}$, we are in situation of \cite{odk16?}[Chapter 6, proposition 3.2]. Thus there is $C\in P_{\mathbb{C}}$ such, that $i(k)=Cj(k)C^{-1}$ for all $k\in H_{\mathbb{C}}$. Let $\theta$ be the involutive automorphism fixing the real form $P$, then $Cj(k)C^{-1}=i(k)=\theta(i(k))=\theta(Cj(k)C^{-1})=\theta(C)j(k)\theta(C)^{-1}$ for all $k\in H$. Thus $C^{-1}\theta(C)$ commutes with all elements in $j(H)$. Since $i(H)$ acts irreducibly on $\mathbb{R}^n$, $C^{-1}\theta(C)$ has to act as multiple of identity by Schur's lemma, thus $\theta(C)C^{-1}=e$ and $\theta(C)=C$ i.e. $C\in P$.
\end{proof} 

\subsection{Extensions to parabolic contact structures of dimension $3$}\label{5.2}

We treat the dimension $3$ separately, because on both sides of parabolic contact geometries and symmetric spaces exceptional phenomena arise.

There are only two types of simple symmetric spaces of dimension two to start with, i.e. $\mathfrak{so}(3)/\mathfrak{so}(2)$ and $\mathfrak{so}(2,1)/\mathfrak{so}(1,1)$. Thus $H$ is discrete in this situation, i.e. $K\cong \mathbb{Z}_2$ it consists only of the symmetry $h$.

The parabolic contact structures of dimension $3$ we are interested in, are those having $\mathfrak{g}$ one of $\mathfrak{sl}(3,\mathbb{R})$, $\mathfrak{su}(2,1)$ and $\mathfrak{sp}(4,\mathbb{R})$.

\begin{lem}
For any choice of $\mathfrak{g}$ and symmetric space $\mathfrak{so}(3)/\mathfrak{so}(2)$ or $\mathfrak{so}(2,1)/\mathfrak{so}(1,1)$ there is unique (up to equivalence) $i$ satisfying assumptions of proposition \ref{pp29}.
\end{lem}
\begin{proof}
To define $i: H \to G_0$ for the extensions, it suffices to give the image of $h$, which will be unique, because we assume that $(G,P)$ is effective. We map $h$ to element \[ \left( \begin{array}{ccc}
-1 & 0 & 0 \\
0 & 1 & 0 \\
0 & 0 & -1 \end{array} \right)\] in $G_0$ for $\mathfrak{sl}(3,\mathbb{R})$, $\mathfrak{su}(2,1)$ and map $h$ to element 
\[ \left( \begin{array}{cccc}
-1 & 0 & 0 & 0 \\
0 & 1 & 0& 0 \\
0 & 0 & 1 & 0\\
0& 0& 0& -1 \end{array} \right)\] in $G_0$ for $\mathfrak{sp}(4,\mathbb{R})$.

Any linear isomorphism is isomorphism of representations $H$ and $i(H)$, thus $i$ satisfies assumptions of proposition \ref{pp29}.
\end{proof}

Then following the proposition \ref{pp29} we can construct (regular) $\alpha$ as follows. First, we write $(e,x_1,x_2)$ for the following matrices
\[ \left( \begin{array}{ccc}
0 & e & -x_1  \\
-ce & 0 & -cx_2  \\
x_1 & x_2 & 0    \end{array} \right)\]
in $\mathfrak{so}(2+c,1-c)$. Further, $b_1, b_2, b_3, b_4, a_1, a_2, c_1, d_1, d_2, d_3, d_4$ are real numbers such, that $b_1b_4-b_2b_3\neq 0$.

For $\mathfrak{g}=\mathfrak{sl}(3,\mathbb{R})$, the proposition \ref{pp29} implies
\[ \alpha(e,x_1,x_2)=
\left( \begin{array}{ccc}
a_1e & d_1x_1+d_2x_2 & c_1e \\
b_1x_1+b_2x_2 & a_2e  & d_3x_1+d_4x_2 \\
(b_1b_4-b_2b_3)e & b_3x_1+b_4x_2 & -(a_1+a_2)e \end{array} \right).\] 

Similarly, in the case $\mathfrak{g}=\mathfrak{su}(2,1)$ 
\[\alpha(e,x_1,x_2)=\]
\[
\left( \begin{array}{ccc}
a_1e+a_2ei & * & c_1ei \\
b_1x_1+b_2x_2+(b_3x_1+b_4x_2)i & -2a_2ei  & d_1x_1+d_2x_2+(d_3x_1+d_4x_2)i \\
2(b_1b_4-b_2b_3)ei & * & -a_1e+a_2ei \end{array} \right)y\] 

where $*$ means that, the entry is determined by the structure of Lie algebra $\mathfrak{su}(2,1)$.

For $\mathfrak{g}=\mathfrak{sp}(4,\mathbb{R})$, 
\[\alpha(e,x_1,x_2)=\]
\[
\left( \begin{array}{cccc}
a_1e & d_1x_1+d_2x_2& d_3x_1+d_4x_2 & c_1e \\
b_1x_1+b_2x_2 & a_2e & a3e & d_3x_1+d_4x_2 \\
b_3x_1+b_4x_2 & a4e &-a_2e  & -d_1x_1-d_2x_2 \\
2(b_1b_4-b_2b_3)e & b_3x_1+b_4x_2& -b_1x_1-b_2x_2 & -a_1e \end{array} \right).\] 

We skip computations of normality conditions and automorphisms, which can be easily done due to the dimension. But we look on equivalence classes of extensions in detail. We shall employ morphisms from proposition \ref{1.4.7} to construct suitable canonical forms of the morphisms $\alpha$, and thus we shall classify all equivalence classes of $\alpha$ for fixed $i$.

In $\mathfrak{g}=\mathfrak{sl}(3,\mathbb{R})$ case we can use morphisms to get
$$(b_1b_4-b_2b_3)'=\frac{(b_1b_4-b_2b_3)}{n_2^2n_3}, b_1'=\frac{b_1n_3}{n_2},b_2'=\frac{b_2n_3}{n_2}, b_3'=\frac{b_3}{n_3^2n_2}, b_4'=\frac{b_4}{n_3^2n_2}$$
so one can choose $b_1b_4-b_2b_3=1$ and one of $b_1, b_2, b_3, b_4=1$.

In the $c=1$ case we can use morphisms to get
\[b_1'=b_1cos(n_1)-b_2sin(n_1),\]
\[b_2'=b_1sin(n_1)+b_2cos(n_1),\]
\[b_3'=b_3cos(n_1)-b_4sin(n_1),\]
\[b_4'=b_3sin(n_1)+b_4cos(n_1),\]
so we can choose $b_2=0$. So $b_3'=\frac{b_1b_3+b_2b_4}{b_1b_4-b_2b_3}$, then in proposition \ref{1.4.7} we can get $b_3'=-b_3$ and rest the same.

In the $c=-1$ case we can use morphisms to get
\[b_1'=b_1cosh(n_1)-b_2sinh(n_1),\]
\[ b_2'=-b_1sinh(n_1)+b_2cosh(n_1)\]
\[b_3'=b_3cosh(n_1)-b_4sinh(n_1),\]
\[ b_4'=-b_3sinh(n_1)+b_4cosh(n_1).\]

Since we can use morphisms exchange $b_1, b_3$ with $b_2, b_4$, we can choose $b_1^2\geq b_2^2$. If $b_1^2>b_2^2$, then we can choose $b_2=0$, and then $t:=b_3'=\frac{-b_1b_3+b_2b_4}{b_1b_4-b_2b_3}$. If $b_1^2=b_2^2$, then we can choose $b_3^2\leq b_4^2$, if $b_3^2<b_4^2$, then we can choose $b_1=1, b_3=0$, if $b_3^2=b_4^2$, then we can choose $b_1=1, b_2=1, b_3=-1, b_4=1$. Again we can get $b_3'=-b_3$, if we use morphisms.

\begin{thm}
Up to equivalences, all regular normal extensions for  \linebreak $\mathfrak{so}(3)/\mathfrak{so}(2)$ to $\mathfrak{sl}(3,\mathbb{R})$ are given by the following one parameter classes with $t\geq 0$:
\[\alpha(e,x_1,x_2)=
\left( \begin{array}{ccc}
\frac{t}{4}e & -\frac{3t^2+4}{4}x_1+\frac{t}{4}x_2 & -\frac{15t^2+16}{16}e \\
x_1 & -\frac{t}{2}e  & -\frac{3t}{4}x_1-x_2 \\
e & tx_1+x_2 & \frac{t}{4}e \end{array} \right)\] 
with curvature
\[ \kappa((e,x_1,x_2),(h,y_1,y_2))=\]
\[
\left( \begin{array}{ccc}
0 & \frac{3(t^3+t)}{2}(hx_1-ey_1) & 0 \\
0 & 0  & -\frac{3t^2}{2}(hx_1-ey_1)-\frac{3t}{2}(hx_2-ey_2) \\
0 & 0 & 0 \end{array} \right).\]

Up to equivalences, all regular normal extensions for $\mathfrak{so}(2,1)/\mathfrak{so}(1,1)$ to $\mathfrak{sl}(3,\mathbb{R})$ are given by the following one parameter classes:

a) for $b_1^2>b_2^2$, there is one parameter class for $t\geq 0$
\[\alpha(e,x_1,x_2)=
\left( \begin{array}{ccc}
-\frac{t}{4}e & \frac{3t^2-4}{4}x_1-\frac{t}{4}x_2 & \frac{16-15t^2}{16}e \\
x_1 & \frac{t}{2}e  & \frac{3t}{4}x_1+x_2 \\
e & tx_1+x_2 & -\frac{t}{4}e \end{array} \right)\] 
with curvature
\[ \kappa((e,x_1,x_2),(h,y_1,y_2))=\]
\[
\left( \begin{array}{ccc}
0 & \frac{3(t^3-t)}{2}(hx_1-ey_1) & 0 \\
0 & 0  & -\frac{3t^2}{2}(hx_1-ey_1)-\frac{3t}{2}(hx_2-ey_2) \\
0 & 0 & 0 \end{array} \right);\]

b) for $b_1^2=b_2^2$ and $b_3^2<b_4^2$
\[\alpha(e,x_1,x_2)=
\left( \begin{array}{ccc}
\frac14e & -x_1-\frac34x_2 & \frac{1}{16}e \\
x_1+x_2 & -\frac12e  & \frac14x_1+\frac14x_2 \\
e & x_2 & \frac14e \end{array} \right)\] 
with curvature
\[ \kappa((e,x_1,x_2),(h,y_1,y_2))=\]
\[
\left( \begin{array}{ccc}
0 & \frac32(hx_1-ey_1)+\frac32(hx_2-ey_2) & 0 \\
0 & 0  & 0 \\
0 & 0 & 0 \end{array} \right);\]

c) for $b_1^2=b_2^2$ and $b_3^2=b_4^2$
\[\alpha(e,x_1,x_2)=
\left( \begin{array}{ccc}
\frac14e & -\frac18x_1+\frac18x_2 & \frac{1}{32}e \\
x_1+x_2 & -\frac12e  & \frac18x_1+\frac18x_2 \\
2e & -x_1+x_2 & \frac14e \end{array} \right)\] 
with is flat.
\end{thm}

In $\mathfrak{g}=\mathfrak{su}(2,1)$ case we can use morphisms to get $b_1b_4-b_2b_3>0$ and
\[ (b_1b_4-b_2b_3)'=2(b_1b_4-b_2b_3)(cosh(n_2)+sinh(n_2))^2,\]
so we can choose $b_1b_4-b_2b_3=1$. The actions of other morphisms are quite complicated, so we won't state them explicitly, but using morphisms we can get $b_2'=b_3'=0, b_1'=t,b_4'=\frac{1}{t}$, where $$t:=\sqrt{\frac{s+c\sqrt{s^2-4c}}{2c}},\ s=\frac{cb_1^2+b_2^2+cb_3^2+b_4^2}{b_1b_4-b_2b_3}.$$

\begin{thm}
Up to equivalences, all regular normal extensions for \linebreak $\mathfrak{so}(3)/\mathfrak{so}(2)$ to $\mathfrak{su}(2,1)$ are given by the following one parameter classes for $s\geq 2$:
\[\alpha(e,x_1,x_2)=
\left( \begin{array}{ccc}
\frac{1+t^4}{8t^2}ie & * & \frac{-(15t^8-34t^4+15)}{128t^4}ie \\
t x_1+\frac{i}{t}x_2 & -\frac{1+t^4}{4t^2}ie  & \frac{-3t^4+5}{16t}x_1+\frac{5t^4-3}{16t^3}ix_2 \\
2ie & * & \frac{1+t^4}{8t^2}ie \end{array} \right),\] 
where $*$ means that, the entry is determined by the structure of $\mathfrak{su}(2,1)$, with curvature 
\[ \kappa((e,x_1,x_2),(h,y_1,y_2))=\]
\[
\left( \begin{array}{ccc}
0 & * & 0 \\
0 & 0  & \frac{3(1-t^8)}{16t^5}(hx_2-ey_2)+\frac{3(1-t^8)}{16t^3}i(hx_1-ey_1)\\
0 & 0 & 0 \end{array} \right),\]
where $*$ means that, the entry is determined by the structure of $\mathfrak{su}(2,1)$.

Up to equivalences, all regular normal extensions for $\mathfrak{so}(2,1)/\mathfrak{so}(1,1)$ to $\mathfrak{su}(2,1)$ are given by the following one parameter classes for $s>-2$:
\[\alpha(e,x_1,x_2)=
\left( \begin{array}{ccc}
\frac{1-t^4}{8t^2}ie & * & \frac{-(15t^8+34t^4+15)}{128t^4}ie \\
t x_1+\frac{i}{t}x_2 & -\frac{1-t^4}{4t^2}ie  & \frac{3t^4+5}{16t}x_1+\frac{-5t^4-3}{16t^3}ix_2 \\
2ie & * & \frac{1-t^4}{8t^2}ie \end{array} \right),\] 
where $*$ means that, the entry is determined by the structure of $\mathfrak{su}(2,1)$, with curvature 
\[ \kappa((e,x_1,x_2),(h,y_1,y_2))=\]
\[
\left( \begin{array}{ccc}
0 & * & 0 \\
0 & 0  & \frac{3(1-t^8)}{16t^5}(hx_2-ey_2)+\frac{3(1-t^8)}{16t^3}i(hx_1-ey_1)\\
0 & 0 & 0 \end{array} \right),\]
where $*$ means that, the entry is determined by the structure of $\mathfrak{su}(2,1)$.
\end{thm}

The $\mathfrak{g}=\mathfrak{sp}(4,\mathbb{R})$ case is flat and all $\alpha$ are equivalent.

\begin{thm}
Up to equivalence, there is the unique regular normal extension for $\mathfrak{so}(2,1)/\mathfrak{so}(1,1)$ to $\mathfrak{sp}(4,\mathbb{R})$ with

\[\alpha(e,x_1,x_2)=
\left( \begin{array}{cccc}
0 & -\frac14x_1& \frac14x_2 & \frac1/8e \\
x_1 & 0 & -\frac12e & \frac14x_2 \\
x_2 & -\frac12e &0  & \frac14x_1 \\
2e & x_2& -x_1 & 0 \end{array} \right),\] 
which is flat.

Up to equivalence, there is the unique regular normal extension for \linebreak $\mathfrak{so}(3)/\mathfrak{so}(2)$ to $\mathfrak{sp}(4,\mathbb{R})$ with
\[\alpha(e,x_1,x_2)=
\left( \begin{array}{cccc}
0 & -\frac14x_1& -\frac14x_2 & -\frac18e \\
x_1 & 0 & \frac12e & -\frac14x_2 \\
x_2 & -\frac12e &0  & \frac14x_1 \\
2e & x_2& -x_1 & 0 \end{array} \right),\] 
which is flat.
\end{thm}

\subsection{Extensions to Lagrangean contact structures}\label{5.3}

In this section we construct examples of symmetric Lagrangean contact structures. We want to find extension to Cartan geometry of type $(\mathfrak{sl}(n+2,\mathbb{R}),P)$ with the following gradation, where the blocks are $(1,n,1)$:
\[ \left( \begin{array}{ccc}
\mathfrak{g}_{0} & \mathfrak{g}_{1} & \mathfrak{g}_{2} \\
\mathfrak{g}_{-1} & \mathfrak{g}_{0} & \mathfrak{g}_{1} \\
\mathfrak{g}_{-2} & \mathfrak{g}_{-1} & \mathfrak{g}_{0} \end{array} \right)\]

The representation of the semisimple part of $\mathfrak{g}_{0}$ on $\mathfrak{g}_{-1}$ is $V\oplus V^*$, where $V$ is standard representation of $\mathfrak{sl}(n,\mathbb{R})$ and $V^*$ is its dual.

Firstly we look on Lagrangean contact structures for simple symmetric spaces.

\begin{prop}\label{odw1}
The only non-complex simple symmetric spaces allowing extensions to regular Lagrangean contact structures are simple para--pseudo--hermitian symmetric space and the pseudo--hermitian symmetric spaces \linebreak $\mathfrak{so}(p+2,q)/ \mathfrak{so}(p,q)+\mathfrak{so}(2)$. For the latter cases, the inclusion $i$ from proposition \ref{pp29} is unique up to equivalence.
\end{prop}
\begin{proof}
Let $(K,L,h)$ be a non-complex simple homogeneous symmetric space and let $H$ be the simisimple part of $L$ extended by the symmetry $h$. Since in the para-pseudo-hermitian case, the $\mathfrak{h}$ has representation $W\oplus W^*$ for some irreducible representation $W: \mathfrak{h}\to \mathfrak{sl}(n,\mathbb{R})$, we define $i$ by $W$. Then $H$ and $i(H)$ are isomorphic, because $(V\oplus V^*)\circ W=V\circ W\oplus V^*\circ W=W\oplus W^*$. For the pseudo-hermitian symmetric spaces the same is possible only in the case of type $\mathbb{R}$ and $W^*\cong \bar{W}$.

Since semisimple part of $G_0$ is simple, we can use proposition \ref{pp30} and we see that $i$ is $W$ or $W^*$, up to equivalence. Then we define morphism $K\times_WP\to K\times_{W^*}P$ as $(k,p)\mapsto ((k^{-1})^T,p)$, which maps extension $(W,\alpha)$ to $(W^*,-\alpha^T)$, and the claim follows from proposition \ref{pp29}.
\end{proof}

Now we explicitly compute one flat example.

\begin{example}
Extension from $(PGl(n+1,\mathbb{R}),Gl(n,\mathbb{R}))$ to $(PGl(n+2,\mathbb{R}),P)$:

The subgroup $Gl(n,\mathbb{R})$ is represented by the following matrices, where the blocks are $(1,n)$ and $B\in Gl(n,\mathbb{R})$
\[ \left( \begin{array}{cc}
1 & 0  \\
0 & B  \end{array} \right).\] 

The symmetry at $o$ is a left multiplication by the following matrix in $Gl(n,\mathbb{R})$, where $E$ is the identity matrix
\[ \left( \begin{array}{cc}
1 & 0  \\
0 & -E  \end{array} \right).\] 

$H$ is the following subgroup, where $A\in Sl(n,\mathbb{R})$
\[ \left( \begin{array}{cc}
1 & 0  \\
0 & \pm A  \end{array} \right).\] 

Now $i$ is the following injective homomorphism, which maps $H$ into $P$
\[ \left( \begin{array}{ccc}
1 & 0 & 0 \\
0 & \pm A & 0 \\
0 & 0 & 1 \end{array} \right).\] 

Since both adjoint representations are $\lambda_1 \oplus \lambda_{n-1}$, the only possible homomorphisms are nonzero multiples. Thus the only possible $\alpha$ are the following, where $a=-Tr(A)$ and $b_1,b_2\in \mathbb{R}$ are nonzero and $c_1,c_2,d_1,d_2,e_1\in \mathbb{R}$
\[ \left( \begin{array}{cc}
a & Y^T \\ 
X & A   \end{array} \right) 
\mapsto 
\left( \begin{array}{ccc}
c_1a & d_1Y^T & e_1a \\ 
b_1X & A+\frac{c2}{n}Ea & d_2X \\
b_1b_2a & b_2Y^T & (1-c_1-c_2)a   \end{array} \right).\]

For fixed $b_1,b_2$ the normality conditions are equivalent to $c_2=0$, $e_1=d_1d_2$, $(n+2)b_1d_1+nd_2b_2-2c_1=n$ and $nb_1d_1-2c_1+(n+2)b_2d_2=n+2$. Thus there are four conditions on five variables and the solution is $d_1=\frac{c_1}{b_1}, d_2=-\frac{c_1-1}{b_2}, c_2=0,e_1=-\frac{c_1-1}{b_2}\frac{c_1}{b_1}$ and $c_1$ free parameter. Thus we can choose $c_1=\frac12$ and then the $\alpha$ extending to normal geometry for fixed $b_1,b_2$ is
\[ \left( \begin{array}{cc}
a & Y^T \\ 
X & A   \end{array} \right) 
\mapsto 
\left( \begin{array}{ccc}
\frac{1}{2}a & \frac{1}{2b_1}Y^T & \frac{1}{4b_1b_2}a \\ 
b_1X & A & \frac{1}{2b_2}X \\
b_1b_2a & b_2Y^T & \frac{1}{2}a   \end{array}\right).\]

Further $\kappa_{\alpha}(X,Y)=0$ for any of these $\alpha$. So they are all equivalent and locally isomorphic to homogeneous model. We can summarize the results in the following proposition.

\begin{prop}
Up to equivalence, there is the unique regular normal extension from $(PGl(n+1,\mathbb{R}),H)$ to Lagrangean contact geometry, which is flat.
\end{prop}
\end{example}

In the case $W$ and $W^*$ are not isomorphic as the representations of $\mathfrak{h}$, then by the Schur lemma only the multiples of identity are isomorphisms. After identification of the representations of $\mathfrak{h}$ and $i(\mathfrak{h})$ via $W$, we are in situation of the previous example. Since the symmetric space has now different curvature $R(X,Y)$, and $\kappa_{\alpha}(X,Y)=[\alpha(X),\alpha(Y)]-\alpha(R(X,Y))$, the resulting contact geometry will not be flat. But using morphism from proposition \ref{1.4.7} we get that again they are all isomorphic. Thus we get the following theorem.

\begin{thm}
Up to equivalence, there is an unique regular normal extension for any non-complex simple para-pseudo-hermitian symmetric space with $W\neq W^*$ to Lagrangean contact structure. The extended geometry is flat only in the case of the previous example.
\end{thm}

We investigate two remaining cases with simple group generated by symmetries, where the representation $W$ is self dual.

\begin{example}
Extension from $(O(p+2,q),O(p,q)\times O(2))$ and $O(p+1,q+1),O(p,q)\times O(1,1))$ to $(PGl(n+2,\mathbb{R}),P)$:

The subgroups $O(p,q)\times O(2)$ and $O(p,q)\times O(1,1)$ are represented by the following matrices, where the blocks are $(2,n)$ and $B\in O(p,q)$ and $b\in O(2)$ or $b\in O(1,1)$
\[ \left( \begin{array}{cc}
b & 0  \\
0 & B  \end{array} \right)\] 

The symmetry at $o$ is represented by a left multiplication by the following matrix in $O(p,q)\times O(2)$ or $O(p,q)\times O(1,1)$, where $E$ are the identity matrices
\[ \left( \begin{array}{cc}
E & 0  \\
0 & -E  \end{array} \right).\] 

The $H$ is following subgroup, where $A\in O(p,q)$
\[ \left( \begin{array}{cc}
E & 0  \\
0 & A  \end{array} \right).\] 

Now $i$ is the following injective homomorphism, which maps $H$ into $P$
\[ \left( \begin{array}{ccc}
1 & 0 & 0 \\
0 & A & 0 \\
0 & 0 & 1 \end{array} \right).\] 

The adjoint representation of $K$ is $\lambda_1 \oplus \lambda_1$ and $i(K)$ is $\lambda_1 \oplus \lambda_{n-1}$, since $H=O(p,q)$, the $\lambda_{n-1}\cong \lambda_1$ as representation of $H$. Now the possible isomorphisms are maps $(X,Y)\mapsto (b_1X+b_2Y,b_3X+b_4Y)$ for $b_1b_4-b_2b_3\neq 0$. Thus the only possible $\alpha$ are the following, where $a\in \mathbb{R}$ and $c$ is $1$ in the $O(2)$ case and $-1$ in the $O(1,1)$ case, $I$ is matrix with $p$ entries on diagonal $1$ and remaining $q$ entries $-1$ and $c_1,c_2,d_1,d_2,d_3,d_4,e_1\in \mathbb{R}$
\[ \left( \begin{array}{ccc}
0 & a & -X^TI \\
-c a & 0 & -c Y^TI \\ 
X & Y & A   \end{array} \right) 
\\
\mapsto \]
\[
\\
\left( \begin{array}{ccc}
c_1a & d_1X^TI+d_2Y^TI & e_1a \\ 
b_1X+b_2Y & A+\frac{c_2a}{n}E & d_3X+d_4Y \\
(b_1b_4-b_2b_3)a & b_3X^TI+b_4Y^TI & (-c_1-c_2)a   \end{array} \right).\]

We denote $\gamma=cb_1b_3+b_2b_4$ and $\delta=b_1b_4-b_2b_3$. For fixed $b_1, b_2, \gamma, \delta$ the normality conditions are $c_2=\frac{n}{n+1}\frac{-\gamma}{\delta}$,$e_1=d_2d_3-d_1d_4$, $b_4d_1-b_3d_2=\frac{-b_4^2-cb_3^2}{\delta}$, $b_2d_3-b_1d_4=\frac{b_2^2+cb_1^2}{\delta}$, $b_2d_1-b_1d_2-b_4d_3+b_3d_4=\frac{n}{n+1}\frac{-\gamma}{\delta}$ and $b_2d_1-b_1d_2+b_4d_3-b_3d_4+2c_1=\frac{n}{n+1}\frac{\gamma}{\delta}$. Thus there are six conditions on seven variables and we compute the solution for $d_1, d_2, d_3, d_4, e_1$ and $c_2$ and let $c_1$ as a free parameter. Thus we can choose $c_1=\frac{n\gamma}{2(n+1)\delta}$ and the $\alpha$ extending to normal geometry for fixed $b_1, b_2, \gamma, \delta$ is
\[ \left( \begin{array}{ccc}
0 & a & -X^TI \\
-c a & 0 & -c Y^TI \\ 
X & Y & A   \end{array} \right) 
\mapsto 
\]
\[
\left( \begin{array}{ccc}
\frac{n\gamma}{2(n+1)\delta} a& V_1 &-(\frac{((3n+2)(n+2)\gamma^2}{4(n+1)^2\delta^3}+\frac{c}{\delta})a \\ 
b_1X+b_2Y & A-\frac{1}{n+1}\frac{\gamma}{\delta}Ea & V_2\\
\delta a & b_3X^TI+b_4Y^TI & \frac{n\gamma}{2(n+1)\delta} a   \end{array}\right),  \]
where 
\[ V_1=-(\frac{(n+2)\gamma b_3}{2(n+1)\delta^2}+\frac{b_4}{\delta})X^TI-(\frac{c(n+2)\gamma b_4}{2(n+1)\delta^2}-\frac{cb_3}{\delta})Y^TI, \] 
\[ V_2=-(\frac{(n+2)\gamma b_1}{2(n+1)\delta^2}-\frac{b_2}{\delta})X-(\frac{c(n+2)\gamma b_2}{2(n+1)\delta^2}+\frac{cb_1}{\delta})Y.\]

The curvature of the extended geometry by this $\alpha$ is:
\[ \kappa_{\alpha}\left( \left( \begin{array}{ccc}
0 & a & -X^TI \\
-c a & 0 & -c Y^TI \\ 
X & Y & 0   \end{array} \right) ,
\left( \begin{array}{ccc}
0 & b & -Z^TI \\
-c b & 0 & -c W^TI \\ 
Z & W & 0   \end{array} \right)  \right)=
\]
\[
\left( \begin{array}{ccc}
0 & -\frac{n+2}{n+1}\frac{(cb_3^2+b_4^2)\gamma}{\delta^3}V_3 & 0\\
0 & \gamma \frac{R_1\delta-\frac{(c+1)(n+2)}{2(n+1)}R_2+\frac{n+2}{2(n+1)}R_3}{\delta^2} -\frac{\gamma (W^TIX-Y^TIZ)}{(n+1)\delta} E & -\frac{n+2}{n+1}\frac{(cb_1^2+b_2^2)\gamma}{\delta^3}V_4 \\ 
0 & 0 & 0   \end{array} \right),\] 
where 
\[ R_1=XW^TI+WX^TI-YZ^TI-ZY^TI,\] 
\[ R_2=b_1b_4(XW^TI-YZ^TI)-b_2b_3(WX^TI-ZY^TI),\]
\[ R_3=b_1b_3(ZX^TI-XZ^TI)+b_2b_4(WY^TI-YW^TI\]
are $n\times n$ matrices and 
\[ V_3=b_1bX^TI+b_2bY^TI-b_1aZ^TI-b_2aW^TI,\] 
\[ V_4=b_3bX+b_4bY-b_3aZ-b_4aW\]
are matrices $1\times n$ and $n\times 1$.

For $\gamma=cb_1b_3+b_2b_4=0$ the extended geometry is flat. Using algorithm in proposition \ref{lab_1} we compute, that the infinitesimal automorphisms for $\gamma \neq 0$ are of the form $\alpha(\mathfrak{k})$. The equivalence classes are determined by $(b_1,b_2,b_3,b_4)$ in the same way as in dimension $3$. In particular, $t=\frac{\gamma}{\delta}$.


\begin{thm}\label{5.3.6}
Up to equivalence, all regular normal extensions from \linebreak $(O(p+2,q),O(p,q))$ and $O(p+1,q+1),O(p,q))$ in the case $b_1^2>b_2^2$ to a Lagrangean contact geometry are given by the following one parameter classes for $t\geq 0$:
\[\alpha \left( \begin{array}{ccc}
0 & a & -X^TI \\
-c a & 0 & -c Y^TI \\ 
X & Y & A   \end{array} \right) 
= 
\]
\[
\left( \begin{array}{ccc}
\frac{n}{2(n+1)}t a& -(\frac{(n+2)}{2(n+1)}ct^2+1)X^TI-\frac{n}{2(n+1)}tY^TI&-\frac{((3n+2)(n+2)}{4(n+1)^2}t^2a-ca \\ 
X & A-\frac{1}{n+1}t Ea &-\frac{(n+2)}{2(n+1)}tX-cY\\
a & c t X^TI+Y^TI & \frac{n}{2(n+1)} t a   \end{array} \right)  \]
with curvature
\[ \kappa_{\alpha}\left( \left( \begin{array}{ccc}
0 & a & -X^TI \\
-c a & 0 & -c Y^TI \\ 
X & Y & 0   \end{array} \right) ,
\left( \begin{array}{ccc}
0 & b & -Z^TI \\
-c b & 0 & -c W^TI \\ 
Z & W & 0   \end{array} \right)  \right)=
\]
\[
\left( \begin{array}{ccc}
0 & \frac{(n+2)t}{n+1}(1+ct^2)(bX^TI-aZ^TI) & 0\\
0 & \frac{t}{(n+1)}((n+1)R_1-R_2+(n+2)ctR_3) -\frac{tr_1}{(n+1)}E& -\frac{(n+2)t}{n+1}V_1 \\ 
0 & 0 & 0   \end{array} \right), \]
where 
\[ R_1=WX^TI-ZY^TI,\]
\[ R_2=XW^TI-YZ^TI,\]
\[ R_3=ZX^TI-XZ^TI,\]
\[ r_1=W^TIX-Y^TIZ,\]
\[ V_1=t(bX-aZ)+c(bY-aW).\]

For $b_1^2=b_2^2$ and $O(p+1,q+1),O(p,q))$:

a) for $b_3^2<b_4^2$
\[\alpha \left( \begin{array}{ccc}
0 & a & -X^TI \\
a & 0 & Y^TI \\ 
X & Y & A   \end{array} \right) 
= 
\]
\[
\left( \begin{array}{ccc}
\frac{n}{2(n+1)} a& -\frac{(n+2)}{2(n+1)}Y^TI-X^TI&-\frac{(3n+2)(n+2)}{4(n+1)^2}a+a \\ 
X+Y & A-\frac{1}{n+1} Ea &\frac{n}{2(n+1)}(X+Y)\\
a & Y^TI & \frac{n}{2(n+1)} a   \end{array} \right)  \]
with curvature
\[ \kappa_{\alpha}\left( \left( \begin{array}{ccc}
0 & a & -X^TI \\
a & 0 & Y^TI \\ 
X & Y & 0   \end{array} \right) ,
\left( \begin{array}{ccc}
0 & b & -Z^TI \\
b & 0 & W^TI \\ 
Z & W & 0   \end{array} \right)  \right)=
\]
\[
\left( \begin{array}{ccc}
0 & \frac{(n+2)}{n+1}(bY^TI-aZ^TI)+\frac{(n+2)}{n+1}(bX^TI-aW^TI) & 0\\
0 & \frac{1}{(n+1)}((n+1)R_1-R_2-(n+2)R_3) -\frac{r_1}{(n+1)}E& 0 \\ 
0 & 0 & 0   \end{array} \right), \]
where 
\[ R_1=WX^TI-YZ^TI,\]
\[ R_2=XW^TI-ZY^TI,\]
\[ R_3=YW^TI-WY^TI,\]
\[ r_1=W^TIX-Y^TIZ.\]

b) for $b_3^2=b_4^2$
\[\alpha \left( \begin{array}{ccc}
0 & a & -X^TI \\
a & 0 & Y^TI \\ 
X & Y & A   \end{array} \right) 
= 
\]
\[
\left( \begin{array}{ccc}
\frac{n}{2(n+1)} a& -\frac{n}{4(n+1)}X^TI+\frac{n}{4(n+1)}Y^TI&\frac{n^2}{8(n+1)^2}a \\ 
X+Y & A-\frac{1}{n+1} Ea &\frac{n}{4(n+1)}(X+Y)\\
2a & -X^TI+Y^TI & \frac{n}{2(n+1)} a   \end{array} \right)  \]
with curvature
\[ \kappa_{\alpha}\left( \left( \begin{array}{ccc}
0 & a & -X^TI \\
a & 0 & Y^TI \\ 
X & Y & 0   \end{array} \right) ,
\left( \begin{array}{ccc}
0 & b & -Z^TI \\
b & 0 & W^TI \\ 
Z & W & 0   \end{array} \right)  \right)=
\]
\[
\left( \begin{array}{ccc}
0 & 0 & 0\\
0 & \frac{1}{2(n+1)}(nR_1+nR_2+(n+2)R_3-(n+2)R_4) -\frac{r_1}{(n+1)}E& 0 \\ 
0 & 0 & 0   \end{array} \right), \]
where 
\[ R_1=XW^TI-ZY^TI,\]
\[ R_2=WX^TI-YZ^TI,\]
\[ R_3=XZ^TI-ZX^TI,\]
\[ R_4=YW^TI-WY^TI,\]
\[ r_1=W^TIX-Y^TIZ.\]
\end{thm}
\end{example}

The classification in the semisimple case is the following:
 
\begin{thm}
The only semisimple non-simple symmetric spaces without complex factors allowing extensions to regular Lagrangean contact structures are semisimple para-pseudo-hermitian symmetric spaces. For the latter cases, the inclusion $i$ from proposition \ref{pp29} is unique up to equivalence.
\end{thm}
\begin{proof}
For semisimple para--pseudo--hermitian symmetric spaces without \linebreak complex factors, the extension can be done in two steps. First we take extension from the sum of symmetric spaces to the structure group $(Gl(n,\mathbb{R})\times Gl(n,\mathbb{R}))\cap O(n,n)$, which acts as standard and dual to standard representation and is unique up to para-complex multiple. Then the claim follows in the same way as proposition \ref{odw1}.

Now assume the extension exists. Then since the representation of $i(H)$ is completely reducible, the simple factors have extension to Lagrangean contact geometry, when we restrict to the submatrix (in frame compatible with factors) with values in this factor. This defines extension to Lagrangean contact geometry of lower dimension. Assume that one factor is pseudo-hermitian and not para-hermitian, then the eigenvalues of its center are $\pm i$ and $L/H$ has to be this center, which is contradiction since due to regularity the $L/H$ intersects all factors.
\end{proof}

\subsection{Extensions to CR structures}\label{5.4}

In this section we construct examples of symmetric partially integrable almost CR structures. This means, due to the torsion freeness we construct the CR structures. So we want to find extension to Cartan geometry of type $(\mathfrak{su}(p+1,q+1),P)$ with the following gradation:,

\[ \left( \begin{array}{ccc}
\mathfrak{g}_{0} & \mathfrak{g}_{1} & \mathfrak{g}_{2} \\
\mathfrak{g}_{-1} & \mathfrak{g}_{0} & \mathfrak{g}_{1} \\
\mathfrak{g}_{-2} & \mathfrak{g}_{-1} & \mathfrak{g}_{0} \end{array} \right),\]

where the blocks are $(1,n,1)$ and $AJ+JA^*=0$ for $A\in \mathfrak{su}(p+1,q+1)$, where $J$ is representing the pseudo hermitian form $$(x_0,x_i,x_{n+1})J(y_0,y_i,y_{n+1})^*=x_0\bar{y}_{n+1}+x_{n+1}\bar{y}_{0}+\sum_{i=1}^p x_i\bar{y}_{i}-\sum_{i=p+1}^n x_i\bar{y}_{i}.$$

The representation of the semisimple part of $\mathfrak{g}_{0}$ on $\mathfrak{g}_{-1}$ is $V$, where $V$ is standard representation of $\mathfrak{su}(p,q)$.

\begin{prop}\label{odw2}
The only non-complex simple symmetric spaces allowing extensions to regular CR structures are simple pseudo-hermitian symmetric spaces and simple para-pseudo-hermitian symmetric spaces $\mathfrak{so}(p+1,q+1)/ \mathfrak{so}(p,q)+\mathfrak{so}(1,1)$. For the latter cases, the inclusion $i$ from proposition \ref{pp29} is unique up to equivalence.
\end{prop}
\begin{proof}
Let $(K,L,H)$ be a non-complex simple homogeneous symmetric space and let $H$ be the simisimple part of $L$ extended by the symmetry $h$. In the pseudo-hermitian case, the $\mathfrak{h}$ has representation $W$ for some representation $W: \mathfrak{k}\to \mathfrak{su}(p,q)$ and we can define $i$ by $W$. Then $H$ and $i(H)$ are isomorphic, because $V\circ W=W$. In the para-pseudo-hermitian case, the same is possible only if $W^*\cong \bar{W}$.

Since semisimple part of $G_0$ is simple, we can use proposition \ref{pp30} and we see that $i$ is up to equivalence $W$ or $\bar{W}$. Then we define morphism $K\times_WP\to K\times_{\bar{W}}P$ as $(k,p)\mapsto ((k^{-1})^*,p)$, which maps extension $(W,\alpha)$ on $(\bar{W},-\alpha^*)$, and the claim follows from proposition \ref{pp29}.
\end{proof}

Now we explicitly compute one flat example.

\begin{example}
Extension from  $(PSU(p+1,q),U(p,q))$ to $(PSU(p+1,q+1),P)$:

The subgroup $U(p,q)$ is represented by the following matrixes, where the blocks are $(1,n)$ and $B\in U(p,q)$
\[ \left( \begin{array}{cc}
1 & 0  \\
0 & B  \end{array} \right).\] 

The symmetry at $o$ is a left multiplication by the following matrix in $PSU(p+1,q)$, where $E$ is the identity matrix
\[ \left( \begin{array}{cc}
1 & 0  \\
0 & -E  \end{array} \right).\] 

$H$ is the following subgroup, where $A\in SU(p,q)$

\[ \left( \begin{array}{cc}
1 & 0  \\
0 & \pm A  \end{array} \right)\] 

Now $i$ is the following injective homomorphism, which maps $H$ into $P$
\[ \left( \begin{array}{ccc}
1 & 0 & 0 \\
0 & \pm A & 0 \\
0 & 0 & 1 \end{array} \right).\] 

Since both adjoint representations are $\lambda_1$, the only possible homomorphisms are nonzero (complex) multiples. Thus the only possible $\alpha$ are the following, where $a=-Tr(A)$ and $b\in \mathbb{C}$ is nonzero, $c,d \in \mathbb{C}$ and $e\in \mathbb{R}$.
\[ \left( \begin{array}{cc}
a & -\bar{X}^TI \\ 
X & A   \end{array} \right) 
\mapsto 
\left( \begin{array}{ccc}
ca & -\bar{d}\bar{X}^TI & ea \\ 
bX & A+\frac{1-2\operatorname{Re}(c)}{n}Ea & dX \\
b\bar{b}a & -\bar{b}\bar{X}^TI & \bar{c}a   \end{array} \right).\]

For fixed $b$, the normality conditions are equivalent to $\operatorname{Re}(c)=1/2$, $e=d\bar{d}$, $c=\bar{b}d$. Thus there are four conditions on five variables and the solution is $d=\frac{c}{\bar{b}}, \operatorname{Re}(c)=0,e=\frac{c\bar{c}}{b\bar{b}}$ and $\operatorname{Im}(c)$ free parameter. Thus if we choose $\operatorname{Im}(c)=0$, the resulting $\alpha$ for fixed $b$ is:
\[ \left( \begin{array}{cc}
a & -\bar{X}^TI \\ 
X & A   \end{array} \right) 
\mapsto 
\left( \begin{array}{ccc}
\frac12a & -\frac{1}{2b}\bar{X}^TI & \frac{1}{4b\bar{b}}a \\ 
bX & A & \frac{1}{2\bar{b}}X \\
b\bar{b}a & -\bar{b}\bar{X}^TI & \frac12a   \end{array}\right).\]

Further $\kappa_{\alpha}(X,Y)=0$ for any of these $\alpha$. So they are all isomorphic and locally isomorphic to homogeneous model. We can summarize the result in the following proposition.

\begin{prop}
Up to equivalence, there is an unique regular normal extension from $(PSU(p+1,q),H)$ to CR structure, which is flat.
\end{prop}
\end{example}

In the case $W$ and $\bar{W}$ are not isomorphic as the representations of $H$, then by the Schur lemma only the multiples of identity are isomorphisms. After identification of the representations of $\mathfrak{h}$ and $i'(\mathfrak{h})$ via $W$, we are in situation of the previous example. Since the symmetric space has now different curvature $R(X,Y)$, and $\kappa_{\alpha}(X,Y)=[\alpha(X),\alpha(Y)]-\alpha(R(X,Y))$, the resulting parabolic contact geometry will not be flat. But using morphism from proposition \ref{1.4.7} we get that again they are all isomorphic. Thus we get the following theorem.

\begin{thm}
Up to equivalence, there is an unique regular normal extension for any non-complex simple pseudo-hermitian symmetric space with $W\neq\bar{W}$ to CR structure. The extended geometry is flat only in the case of the previous example.
\end{thm}

Now we investigate the remaining cases with simple group generated by symmetries, where $W$ is self conjugate.

\begin{example}
Extension from $(O(p+2,q),O(p,q)\times O(2))$, $(O(p+1,q+1),O(p,q)\times O(1,1))$ to $(PSU(p+1,q+1),P)$:

The symmetric space and the $i$ are the same as in Lagrangean contact case.

Since there is no complex structure on $H$, we choose two identifications of $W=\lambda_1+\lambda_1$ with complex numbers, i.e. $(X_1,X_2)\mapsto X_1+iX_2=X$ and $(X_1,X_2)\mapsto X_2+iX_1=-i\bar{X}$. Then the isomorphisms of representations are given by complex multiples of those two identifications by $b_1,b_2\neq 0$ such, that $|b_1|\neq |b_2|$. So all the possible $\alpha$ are the following, where $a\in \mathbb{R}$ and $c$ is $1$ in the $O(2)$ case and $-1$ in the $O(1,1)$ case, $I$ is matrix with $p$ entries on diagonal $1$ and remaining $q$ entries $-1$, $c_1,d_1,d_2\in \mathbb{C}$ and $e_1\in \mathbb{R}$:
\[ \left( \begin{array}{ccc}
0 & a & -X_1^TI \\
-c a & 0 & -c X_2^TI \\ 
X_1 & X_2 & A   \end{array} \right) 
\mapsto \]
\[
\left( \begin{array}{ccc}
c_1a & -(\bar{d}_1\bar{X}^T+\bar{d}_2iX^T)I  & e_1ai \\ 
b_1X-b_2i\bar{X} & A-\frac{2\operatorname{Im}(c_1)}{n}Eai & d_1X-d_2i\bar{X} \\
2(|b_1|^2-|b_2|^2)ai & -(\bar{b}_1\bar{X}^T+\bar{b}_2iX^T)I & -\bar{c_1}a   \end{array} \right).\]

For fixed $b_1,b_2$, the normality conditions are different for $c=1$ and $c=-1$, so we skip the exact form of them. We only mention, that $\operatorname{Re}(c_1)$ is a free parameter and we choose $\operatorname{Re}(c_1)=0$. The resulting $\alpha$ for fixed $b_1,b_2$ is:

For $c=-1$
\[ \left( \begin{array}{ccc}
0 & a & -X_1^TI \\
a & 0 & X_2^TI \\ 
X_1 & X_2 & A   \end{array} \right) \mapsto
\]
\[ 
\left( \begin{array}{ccc}
\frac{n}{2(n+1)}tai & * & (\frac{-1}{2(|b_1|^2-|b_2|^2)}-\frac{(n+2)(3n+2)}{8(n+1)^2}\frac{t^2}{|b_1|^2-|b_2|^2})ai \\ 
b_1X-b_2i\bar{X} & A-\frac{2}{2(n+1)}tEai & V_1 \\
2(|b_1|^2-|b_2|^2)ai & * & \frac{n}{2(n+1)}tai   \end{array} \right)\]
where $t:=\frac{(b_1,b_2)}{|b_1|^2-|b_2|^2}=2\frac{\operatorname{Re}(b_1)\operatorname{Im}(b_2)-\operatorname{Re}(b_2)\operatorname{Im}(b_1)}{|b_1|^2-|b_2|^2}$, entry on $*$ comes from structure of Lie algebra $\mathfrak{su}(p+1,q+1)$ and 
\[V_1=(\frac{ib_2}{2(|b_1|^2-|b_2|^2)}-\frac{n+2}{4(n+1)}\frac{tb_1}{|b_1|^2-|b_2|^2})X\]\[+(\frac{ib_1}{2(|b_1|^2-|b_2|^2)}+\frac{n+2}{4(n+1)}\frac{tb_2}{|b_1|^2-|b_2|^2})i\bar{X}.\]

For $c=1$
\[ \left( \begin{array}{ccc}
0 & a & -X_1^TI \\
-a & 0 & -X_2^TI \\ 
X_1 & X_2 & A   \end{array} \right) 
\mapsto 
\]
\[
\left( \begin{array}{ccc}
\frac{n}{2(n+1)}tai & * & (\frac{1}{2(|b_1|^2-|b_2|^2)}-\frac{(n+2)(3n+2)}{8(n+1)^2}\frac{t^2}{|b_1|^2-|b_2|^2})ai \\ 
b_1X-b_2i\bar{X} & A-\frac{1}{n+1}tEai & V_2 \\
2(|b_1|^2-|b_2|^2)ai & * & \frac{n}{2(n+1)}tai   \end{array} \right),\]
where $t:=\frac{(b_1,b_2)}{|b_1|^2-|b_2|^2}=\frac{|b_1|^2+|b_2|^2}{|b_1|^2-|b_2|^2}$, entry on $*$ comes from structure of Lie algebra $\mathfrak{su}(p+1,q+1)$ and 
\[V_2=(\frac{b_1}{2(|b_1|^2-|b_2|^2)}-\frac{n+2}{4(n+1)}\frac{tb_1}{|b_1|^2-|b_2|^2})X\]\[+(\frac{b_2}{2(|b_1|^2-|b_2|^2)}+\frac{n+2}{4(n+1)}\frac{tb_2}{|b_1|^2-|b_2|^2})i\bar{X}. \]

Explicit computation of the curvature using Maple reveals, that $\kappa=0$ for $t=0$, and $\kappa\neq 0$ otherwise. Using the algorithm from proposition \ref{lab_1} we compute, that the infinitesimal automorphisms for $t\neq 0$ are of the form $\alpha(\mathfrak{g})$. Further using morphisms from proposition \ref{1.4.7} we get that the $\alpha$ can be chosen with for $c=-1$ with $b_1=\sqrt{\frac{1+\sqrt{t^2+1}}{2}}, b_2=i\sqrt{\frac{-1+\sqrt{t^2+1}}{2}},t>-1$ and for $c=1$ with $b_1=\sqrt{\frac{1+t}{2}}, b_2=\sqrt{\frac{t-1}{2}}i, t>1$.

\begin{thm}
Up to equivalence, all regular normal extensions from \linebreak $(O(p+2,q),O(p,q))$ to CR structures form one parameter class for $t\geq 1$.

Up to equivalence, all regular normal extensions from \linebreak $(O(p+1,q+1),O(p,q))$ to CR structures form one parameter class for $t>-1$.
\end{thm}
\end{example}

\begin{example}
Extension from $(SO^*(2n+2),SO^*(2n)\times SO^*(2))$ to \linebreak $(PSU(n,n),P)$:

We will not give the explicit form of $i$, the symmetric spaces and explicit computations, which were done using Maple,, but we start already with the $\alpha$. The representation $\lambda_1$ of $SO^*(2n)$ is quaternionic and the isomorphism are of the form 
\[ (f_1,f_2): X=X_1+iX_2+jX_3+kX_4\mapsto (X_1+iX_2,X_3+iX_4)\]
up to right quaternionic multiple. We also skip details on the computation of normality conditions and present the $\alpha$ leading the regular normal extension: 
\[ \left( \begin{array}{cccc}
0 & -X_1^T-iX_2^T & ai & -X_3^T+iX_4^T \\
X_1+iX_2 & A+iB & X_3+iX_4& C+iD \\ 
ai & X_3^T+iX_4^T & 0 & -X_1^T-iX_2^T \\
-X_3+iX_4 & -C+iD & X_1-iX_2& A-iB \\   \end{array} \right) 
\mapsto\]
\[
\left( \begin{array}{cccc}
\frac{nt}{(2n+1)|b|}ai & -f_1(\bar{X}\bar{d})^T & f_2(\bar{X}\bar{d})^T &|d|ai \\
f_1(Xb) & A-Di-\frac{1t}{(2n+1)|b|}aiE & B-Ci& f_1(Xd) \\ 
f_2(Xb) & -B-Ci & A+iD-\frac{1t}{(2n+1)|b|}aiE & f_2(Xd) \\
|b|ai & -f_1(\bar{X}\bar{b})^T & f_2(\bar{X}\bar{b})^T& \frac{nt}{(2n+1)|b|}ai \\   \end{array} \right), 
\]
where $b=b_1+ib_2+jb_3+kb_4\neq 0$, $t=b_1^2-b_2^2-b_3^2+b_4^2$ and 
\[d=\frac{(b_1+kb_4)((2n+1)|b|-(n+1)t)}{(2n+1)|b|^2}+\frac{(ib_2+jb_3)((n+1)|b|-(2n+1)t)}{(2n+1)|b|^2}.\]

The extension is flat for $t=0$ and non-flat otherwise. Using algorithm from proposition \ref{lab_1} we compute, that the infinitesimal automorphisms for $t\neq 0$ are of the form $\alpha(\mathfrak{g})$. Further using morphism M1) and M2) we get that the $\alpha$ can be chosen with $b=\sqrt{\frac{1+t}{2}}+\sqrt{\frac{1-t}{2}}j$.

\begin{thm}
Up to equivalence, all regular normal extensions from\linebreak $(SO^*(2n+2),SO^*(2n))$ to CR structures form one parameter class for $t\geq 0$. They are non flat for $t\neq 0$.
\end{thm}
\end{example}

In the semisimple case is the situation following:

\begin{thm}
The only semisimple non-simple symmetric spaces without complex factors allowing extensions to regular CR structures are semisimple pseudo-hermitian symmetric spaces. For the latter cases, the inclusion $i$ from proposition \ref{pp29} is unique up to equivalence.
\end{thm}
\begin{proof}
For semisimple pseudo-hermitian symmetric space without complex factors, the extension can be done in two steps. First we take extension from the sum of symmetric spaces to the structure group $U(p,q)$, which acts as standard representation. Then the claim follows in the same way as proposition \ref{odw2}.

Now assume the extension exists. Then for the same reasons as in the Lagrangian case, the simple factors have extension to integrable almost CR structures. Assume that one factor is para-pseudo-hermitian and not pseudo-hermitian, then the eigenvalues of its center are $\pm 1$ and $L/H$ has to be this center, which is contradiction since due to regularity the $L/H$ intersects all factors.
\end{proof}

\subsection{Extension to contact projective structures}\label{5.5}

In the case of contact projective structure, we need to impose more conditions on symmetries to be automorphisms of the corresponding Cartan geometry. In particular, we need to assume that symmetries preserve a contact projective class of partial connections on $T^-M$. Any two partial connections $\nabla$ and $\hat{\nabla}$ in such a class differ by
$$\hat{\nabla}_\nu\xi=\nabla_\nu\xi+\Upsilon(\nu)\xi+\Upsilon(\xi)\nu+ \Upsilon'(\mathcal{L}(\nu,\xi)),$$
where $\Upsilon$ is a smooth section of $(T^-M)^*$ and $\Upsilon': TM/T^-M\to T^-M$ is characterized by $\mathcal{L}(\Upsilon'(\beta),\xi))=\Upsilon(\xi)\beta$.

In this section we construct examples of symmetric contact projective structures. This means, we find extensions to Cartan geometry of type $(\mathfrak{sp}(2n+2,\mathbb{R}),P)$ with the following gradation:
\[ \left( \begin{array}{ccc}
\mathfrak{g}_{0} & \mathfrak{g}_{1} & \mathfrak{g}_{2} \\
\mathfrak{g}_{-1} & \mathfrak{g}_{0} & \mathfrak{g}_{1} \\
\mathfrak{g}_{-2} & \mathfrak{g}_{-1} & \mathfrak{g}_{0} \end{array} \right),\]
where the blocks are $(1,2n,1)$ and $AJ+JA^T=0$ for $A\in \mathfrak{sp}(2n+2,\mathbb{R})$, where $J$ is representing the symplectic form $$(x_0,x_i,x_{2n+1})J(y_0,y_i,y_{2n+1})^*=x_0y_{2n+1}+x_{2n+1}y_{0}+\sum_{i=1}^n (x_iy_{n+i}- x_{n+i}y_{i}).$$

The representation of the semisimple part of $\mathfrak{g}_{0}$ on $\mathfrak{g}_{-1}$ is the standard representation of $\mathfrak{sp}(2n,\mathbb{R})$.

\begin{prop}
The only non-complex simple symmetric spaces allowing extensions to regular contact projective structures are simple para-pseudo-hermitian or pseudo-hermitian symmetric spaces. For these cases, the inclusion $i$ from proposition \ref{pp29} is unique up to equivalence.
\end{prop}
\begin{proof}
Let $(K,L,h)$ be a non-complex simple homogeneous symmetric space and $H$ simisimple part of $L$ extended by $h$. For simple pseudo-hermitian symmetric spaces, the $i'$ is 
\[
\left( \begin{array}{cc}
0& 0 \\ 
0 & A+iB   \end{array} \right) 
\mapsto
 \left( \begin{array}{cccc}
0 & 0 & 0 & 0 \\
0 & A & -BI& 0 \\ 
0 & IB & IAI & 0 \\
0 & 0 & 0& 0 \\   \end{array} \right), 
\]
where $I$ is diagonal matrix given by the signature of the metric as before, $IA+A^TI=0$ and $IB-B^TI=0$. 

For simple para-pseudo-hermitian symmetric spaces, the $i'$ is inclusion of $\mathfrak{so}(n,n)$ as a subgroup.

The element $i(h)$ is 
\[ \left( \begin{array}{cccc}
-1 & 0 & 0 & 0 \\
0 & E & 0& 0 \\ 
0 & 0 & E & 0 \\
0 & 0 & 0& -1 \\   \end{array} \right).
\]

Then the representations of $H$ and $i(H)$ are isomorphic and the extension exists from proposition \ref{pp29}. Since semisimple part of $G_0$ is simple, we can use proposition \ref{pp30} and we see that $i$ is unique up to equivalence.
\end{proof}

In the same way as for the previous types of geometries, we conclude the following theorem. We consider representation $W$ as in Theorems \ref{odw1} or \ref{odw2}.

\begin{thm}
If the representation $W$ is not self dual in the para-pseudo-hermitian case or not self-conjugate in the pseudo-hermitian case, then there is (up to equivalence) unique regular normal contact projective structure for this non-complex simple (para)-pseudo-hermitian symmetric space.
\end{thm}

Now we compute the simple examples, where $W$ is self-dual or self-conjugate.

\begin{example}
Extension from $(O(p+2,q),O(p,q)\times O(2))$, $(O(p+1,q+1),O(p,q)\times O(1,1))$ to $(PSp(2n+2,\mathbb{R}),P)$:

The symmetric spaces are the same as in the case of the previous structures. The all possible $\alpha$ are: 
\[ \left( \begin{array}{ccc}
0 & a & -X_1^TI \\
-ca & 0 & -cX_2^TI \\ 
X_1 & X_2 & A   \end{array} \right) 
\mapsto\]
\[
\left( \begin{array}{cccc}
c_1a & * & * & e_1a \\
b_1X_1+b_2X_2 & A+c_2aE & gaI& d_3X_1+d_4X_2 \\ 
b_3X_1I+b_3X_2I & haI & IAI-c_2aE & -d_1X_1I-d_2X_2I \\
2(b_1b_4-b_2b_3)a & * & *& -c_1a \\   \end{array} \right),
\]
where entries on $*$ comes from structure of the Lie algebra $\mathfrak{sp}(2n+2,\mathbb{R})$ and all coefficients are real numbers such, that $b_1b_4-b_2b_3\neq 0$.

For fixed $b$'s, the normality conditions give us, that $c_1$ can be chosen as free parameter and remaining parameters are dependent. Using the morphisms from proposition \ref{1.4.7}, we get, that all choices of $b$'s are isomorphic. So we get the following result:

\begin{thm}
Up to equivalence, there is the unique regular normal extension from $(O(p+2,q),O(p,q))$ or $(O(p+1,q+1),O(p,q))$ to contact projective structures given by:
\[ \left( \begin{array}{ccc}
0 & a & -X_1^TI \\
-ca & 0 & -cX_2^TI \\ 
X_1 & X_2 & A   \end{array} \right) 
\mapsto
\left( \begin{array}{cccc}
0 & \frac{-n}{2(n+1)}(X_1I)^T & \frac{-cn}{2(n+1)}X_2^T & \frac{-2cn^2}{4(n+1)^2}a \\
X_1 & A & \frac{1}{n+1}aI& \frac{-cn}{2(n+1)}X_2 \\ 
X_2I & \frac{-1}{n+1}aI & IAI & \frac{n}{2(n+1)}X_1I \\
2a & (X_2I)^T & -X_1^T& 0 \\   \end{array} \right)
\]
with curvature 
\[ \kappa(\left( \begin{array}{ccc}
0 & a & -X_1^TI \\
-ca & 0 & -cX_2^TI \\ 
X_1 & X_2 & A   \end{array} \right),
 \left( \begin{array}{ccc}
0 & b & -Y_1^TI \\
-cb & 0 & -cY_2^TI \\ 
Y_1 & Y_2 & B   \end{array} \right) )
=\]
\[
\left( \begin{array}{cccc}
0 & 0 & 0 & 0 \\
0 & R_1 & -cR_3I-R_2I& 0 \\ 
0 &  IR_3+IR_2 & IR_1I & 0 \\
0 & 0 & 0& 0 \\   \end{array} \right),
\]
where \[ R_1=\frac{n+2}{2(n+1)}(X_1Y_1^T-Y_1X_1^T+cX_2Y_2^T-cY_2X_2^T),  \]
\[R_2=\frac{1}{(n+1)}(X_2^TY_1-X_1^TY_2),\]
\[R_3=\frac{n}{2(n+1)}(X_1Y_2^T+Y_2X_1^T-X_2Y_1^T-Y_1X_2^T). \]
\end{thm}
\end{example}

\begin{example}
Extension from  $(SO^*(2n+2),SO^*(2n)\times SO^*(2))$ to\linebreak $(PSp(2n+2,\mathbb{R}),P)$:

Technical computations using Maple lead to the following theorem. We skip the exact form of the symmetric spaces. The representation $\lambda_1$ of $SO^*(2n)$ is quaternionic and the isomorphism is 
\[ X_1+iX_2+jX_3+kX_4\mapsto (X_1,X_2,X_3,X_4) \]
up to a quaternionic multiple. We also skip the details on computation of normality conditions and computation of automorphisms and isomorphisms here.

\begin{thm}
Up to equivalence, there is unique regular normal extensions from $(SO^*(2n+2),SO^*(2n))$ to contact projective structures given by:
\[ \left( \begin{array}{cccc}
0 & -X_1^T-iX_2^T & ai & -X_3^T+iX_4^T \\
X_1+iX_2 & A+iB & X_3+iX_4& C+iD \\ 
ai & X_3^T+iX_4^T & 0 & -X_1^T-iX_2^T \\
-X_3+iX_4 & -C+iD & X_1-iX_2& A-iB \\   \end{array} \right) 
\mapsto\]
\[
\left( \begin{array}{cccccc}
0 & *& * & *& * & -\frac{4n^2}{(2n+1)^2}a \\
X_1 & A& -B& -D_1& -C& -\frac{2n}{(2n+1)}X_3 \\ 
X_2 & B& A& -C& D_1& -\frac{2n}{(2n+1)}X_4 \\
X_3 & D_1& C& A& -B&\frac{2n}{(2n+1)}X_1\\ 
-X_4 & C& -D_1& B& A & \frac{-2n}{(2n+1)}X_2 \\
a & * & *& * & *& 0 \\   \end{array} \right),
\]
where $D_1=D-\frac{a}{(2n+1)}E$ and entries on $*$ comes from structure of the Lie algebra $\mathfrak{sp}(2n+2,\mathbb{R})$, with curvature
\[ \kappa(\left( \begin{array}{cccc}
0 & -X_1^T-iX_2^T & ai & -X_3^T+iX_4^T \\
X_1+iX_2 & 0 & X_3+iX_4& 0 \\ 
ai & X_3^T+iX_4^T & 0 & -X_1^T-iX_2^T \\
-X_3+iX_4 & 0 & X_1-iX_2& 0 \\   \end{array} \right),\]
\[
\left( \begin{array}{cccc}
0 & -Y_1^T-iY_2^T & bi & -Y_3^T+iY_4^T \\
Y_1+iY_2 & 0 & Y_3+iY_4& 0 \\ 
bi & Y_3^T+iY_4^T & 0 & -Y_1^T-iY_2^T \\
-Y_3+iY_4 & 0 & Y_1-iY_2& 0 \\   \end{array} \right) )
=\]
\[
\left( \begin{array}{cccccc}
0 & 0 & 0& 0 & 0 & 0 \\
0 & R1 & R3& R5 & R7 & 0 \\
0 & R4 & R2& R8 & R6 & 0 \\
0 & -R5 & R7& R1 & -R3 & 0 \\
0 & R8 & -R6 & -R4 & R2 & 0 \\
0 & 0 & 0& 0 & 0 & 0 \\  \end{array} \right),
\]
where 
\begin{align*}
 R1=&\frac{1}{(2n+1)}(X_1Y_1^T-Y_1X_1^T+X_4Y_4^T-Y_4X_4^T)\\&-(X_2Y_2^T-Y_2X_2^T+X_3Y_3^T-Y_3X_3^T),\\
 R2=&(X_1Y_1^T-Y_1X_1^T+X_4Y_4^T-Y_4X_4^T)\\
&-\frac{1}{(2n+1)}(X_2Y_2^T-Y_2X_2^T+X_3Y_3^T-Y_3X_3^T),\\
 R3=&\frac{-1}{(2n+1)}(X_1Y_3^T-Y_1X_3^T+X_4Y_2^T-Y_4X_2^T)\\&-(X_2Y_4^T-Y_2X_4^T+X_3Y_1^T-Y_3X_1^T), \\
 R4=&(X_1Y_3^T-Y_1X_3^T+X_4Y_2^T-Y_4X_2^T)\\&+\frac{1}{(2n+1)}(X_2Y_4^T-Y_2X_4^T+X_3Y_1^T-Y_3X_1^T), \\
 R5=&\frac{1}{(2n+1)}(X_1Y_4^T-Y_1X_4^T-X_4Y_1^T+Y_4X_1^T)-(X_2Y_3^T-Y_2X_3^T\\&-X_3Y_2^T+Y_3X_2^T)-\frac{2}{(2n+1)}(X_1^TY_4-X_4^TY_1+X_2^TY_3-X_3^TY_2)E,\\ 
 R6=&-(X_1Y_4^T-Y_1X_4^T-X_4Y_1^T+Y_4X_1^T)+\frac{1}{(2n+1)}(X_2Y_3^T-Y_2X_3^T\\&-X_3Y_2^T+Y_3X_2^T)-\frac{2}{(2n+1)}(X_1^TY_4-X_4^TY_1+X_2^TY_3-X_3^TY_2)E, \\
 R7=&\frac{1}{(2n+1)}(X_1Y_2^T-Y_1X_2^T-X_4Y_3^T+Y_4X_3^T)\\&+(X_2Y_1^T-Y_2X_1^T-X_3Y_4^T+Y_3X_4^T), \\
 R8=&-(X_1Y_2^T-Y_1X_2^T-X_4Y_3^T+Y_4X_3^T)\\&-\frac{1}{(2n+1)}(X_2Y_1^T-Y_2X_1^T-X_3Y_4^T+Y_3X_4^T).
\end{align*}
\end{thm}
\end{example}

The classification in the semisimple case is the following:

\begin{thm}
The only semisimple symmetric spaces without complex factors allowing extensions to regular contact projective structures are sums of simple (para)-pseudo-hermitian symmetric spaces. For the latter cases, the inclusion $i$ from proposition \ref{pp29} is unique up to equivalence.
\end{thm}
\begin{proof}
Apart the center of $\mathfrak{l}$ the extension can be taken as in previous examples. If we have in mind, that any multiplication on invariant subspaces of $\mathfrak{g}_{-1}$ can be obtained by bracket with an element of  $\mathfrak{g}_{0}$, which commutes with image of semisimple part of $\mathfrak{l}$, then image of center of $\mathfrak{l}$ can be chosen to be such elements with appropriate action. The $\mathfrak{l}/\mathfrak{h}$ is then a sum of preimages of $\mathfrak{g}_{-2}$ parts of the relevant previous examples.
\end{proof}

\subsection{Remarks on geometric interpretation}\label{5.6}

As described in \cite{odk14?} one can relate Lagrangean contact geometry with system of differential equations. In our case (cf. \ref{5.3}) the relation is as follows.

Let $(i,\alpha)$ be an extension of $(K,H)$ to Lagrangean contact geometry $(p: K\times_i P\to K/H, \omega_\alpha)$. Let $E$ be
\[ Tp \circ \omega_\alpha^{-1} \left( \begin{array}{ccc}
\mathfrak{g}_{0} & \mathfrak{g}_{1} & \mathfrak{g}_{2} \\
\mathfrak{g}_{-1} & \mathfrak{g}_{0} & \mathfrak{g}_{1} \\
0 & 0 & \mathfrak{g}_{0} \end{array} \right)\]
and let $V$ be
\[ Tp \circ \omega_\alpha^{-1} \left( \begin{array}{ccc}
\mathfrak{g}_{0} & \mathfrak{g}_{1} & \mathfrak{g}_{2} \\
0 & \mathfrak{g}_{0} & \mathfrak{g}_{1} \\
0 & \mathfrak{g}_{-1} & \mathfrak{g}_{0} \end{array} \right).\]

Then, since the latter Cartan geometry is torsion-free, the distributions $E,V$ are integrable. Since $i(H)\subset G_0$, these distributions are invariant with respect to $K$ action i.e. they are given by $\mathfrak{e},\mathfrak{v}\subset \mathfrak{k}/\mathfrak{h}$ and the leaf space corresponding to $V$ is homogeneous space $M=K/exp(\mathfrak{v})$. Now the Cartan geometry corresponds to system of differential equations on $M$. The space of solutions is then a homogeneous space $S=K/exp(\mathfrak{e})$ and the correspondence is as follows: For any point $k\cdot exp(\mathfrak{e})$ of $S$, the fiber containing $k$ in $K/H$ projects to hypersurface in $M$. Thus the symmetry group of the differential equation is $K$ (if the geometry is not flat).

\begin{example}
Extension from $(O(p+2,q),O(p,q)\times O(2))$ to $(PGl(n+2,\mathbb{R}),P)$. If $\alpha$ is given as in theorem \ref{5.3.6}, then $\mathfrak{e}$ is given by $a=0, t X^TI+Y^TI=0$ and $\mathfrak{v}$ is given by $a=0, X=0$. Thus both $M$ and $S$ are $O(p+2,q)/O(p+1,q)$ i.e. quadric in $R^{n+2}$. The correspondence is as follows: The point $k\cdot exp(\mathfrak{e})$ is associated with the intersection of quadric with hyperplane through $k\cdot O(p+1,q)$ orthogonal (in the metric defining the quadric) to $k\cdot (t,1,0,\dots,0)$.
\end{example}

\begin{example}
Extension from $(O(p+1,q+1),O(p,q)\times O(1,1))$ to\linebreak $(PGl(n+2,\mathbb{R}),P)$. There are three possible non-equivalent $\alpha$.

a) $b_1^2>b_2^2$

Then $\mathfrak{e}$ is given by $a=0, -t X^TI+Y^TI=0$ and $\mathfrak{v}$ is given by $a=0, X=0$. Now $M$ is $O(p+1,q+1)/O(p+1,q)$, if $t>1$ then $S$ is $O(p+1,q+1)/O(p+1,q)$, if $t<1$ then $S$ is $O(p+1,q+1)/O(p,q+1)$ and if $t=1$ then $S$ is $O(p+1,q+1)/(O(p,q)\ltimes R^n)$ i.e. again quadric in $R^{n+2}$. The correspondence is as follows: The point $k\cdot exp(\mathfrak{e})$ is associated with the intersection of quadric with hyperplane through $k\cdot O(p,q+1)$ orthogonal (in the metric defining the quadric) to $k\cdot (-t,1,0,\dots,0)$.

b) $b_1^2=b_2^2$ and $b_3^2<b_4^2$

Then $\mathfrak{e}$ is given by $a=0, Y=0$ and $\mathfrak{v}$ is given by $a=0, Y=-X$. Now $M$ is $O(p+1,q+1)/(O(p,q)\ltimes R^n)$ and $S$ is $O(p+1,q+1)/O(p+1,q)$. The correspondence is as follows: The point $k\cdot exp(\mathfrak{e})$ is associated with the intersection of quadric with hyperplane through $k\cdot O(p+1,q)$ orthogonal (in the metric defining quadric) to $k\cdot (0,1,0,\dots,0)$.

c) $b_1^2=b_2^2$ and $b_3^2=b_4^2$

Then $\mathfrak{e}$ is given by $a=0, X=Y$ and $\mathfrak{v}$ is given by $a=0, Y=-X$. Now $M$ is $O(p+1,q+1)/(O(p,q)\ltimes R^n)$ and $S$ is $O(p+1,q+1)/(O(p,q)\ltimes R^n)$. The correspondence is as follows: The point $k\cdot exp(\mathfrak{e})$ is associated with the intersection of quadric with hyperplane through $k\cdot O(p,q)\ltimes R^n$ orthogonal (in the metric defining the quadric) to $k\cdot (-1,1,0,\dots,0)$.
\end{example}

Further, we remark that in dimension three all homogeneous CR--geome\-tries were found by Cartan in \cite{odk15?}. As generalization of the defining functions found by Cartan, we conjecture that in $(O(p+2,q),O(p,q))$ case, the CR-hypersurface is given by equation $$1+\sum_{i=1}^p|z_i|^2-\sum_{i=p+1}^n|z_i|^2+|w|^2=t|1+\sum_{i=1}^pz_i^2-\sum_{i=p+1}^nz_i^2+w^2|$$ in $\mathbb{C}^{n+1}$, and in $O(p+1,q+1),O(p,q))$ case, the CR-hypersurface is given by equation $$1+\sum_{i=1}^p|z_i|^2-\sum_{i=p+1}^n|z_i|^2-|w|^2=t|1+\sum_{i=1}^pz_i^2-\sum_{i=p+1}^nz_i^2-w^2|$$ in $\mathbb{C}^{n+1}$.

\newpage
\appendix

\section{Tables of simple symmetric spaces}\label{appA}

Here we present the classification of pairs $(\mathfrak{k},\mathfrak{h})$ of Lie algebras, such that there is a simple homogeneous symmetric space $(K,H,h)$ with those Lie algebras. The main inputs for the tables are classification of Berger \cite{odk9} and section \ref{2.6}. We use the following symbols and abbreviations in the table below:

\begin{itemize}
\item[$\mathfrak{k}$] is a simple Lie group, not exceptional
\item[$\mathfrak{h}$] is a Lie subgroup of $\mathfrak{k}$, where $p+q=n$ or $k+l+p+q=n$
\item[$ad(\mathfrak{h})$] is the adjoint representation of Lie algebra $\mathfrak{h}$. Each irreducible part is separated by $\oplus$ and the action of simple factors of $\mathfrak{h}$ is given in terms of fundamental representations in the order of the factors separated by $\otimes$. The action of center depends if it is $h$ or $ph$ and is visible from the $1$-grading. The $\bar{\lambda}$ means representation conjugate to $\lambda$ and $\lambda^*$ representation dual to $\lambda$. 
\item[signature] is the signature of the invariant pseudo-Riemannian metric.
\item[prop.] indicates whether there exists such $P_0$-structure on the homogeneous symmetric space using the following parameters:
\begin{itemize}
\item[h] = pseudo-Hermitian structure
\item[ph] = pseudo-para-Hermitian structure
\item[q] = pseudo-quaternionic-K\"ahler structure
\item[pq] = pseudo-para-quaternionic-K\"ahler structure
\end{itemize}
\end{itemize}

\begin{tabular}{|c|c|c|}
\cline{1-3}
\multicolumn{1}{|c|}{\multirow{2}{*}{$\mathfrak{k}$}} &
\multicolumn{1}{|c|}{$\mathfrak{h}$} &
\multicolumn{1}{|c|}{$ad(\mathfrak{h})$}      \\ \cline{2-3}
     & signature & prop.    \\ \cline{1-3}

\multicolumn{1}{|c|}{\multirow{2}{*}{$\mathfrak{sl}(n,\mathbb{C})$}} &
\multicolumn{1}{|c|}{$\mathfrak{sl}(p,\mathbb{C})+\mathfrak{sl}(q,\mathbb{C})+\mathbb{C}$}&
\multicolumn{1}{|c|}{$\lambda_1\otimes \lambda_{1}^*\oplus \lambda_{1}^*\otimes \lambda_{1}$}      \\ \cline{2-3}
      & $(2pq,2pq)$ & h,ph  \\ \cline{1-3}

\multicolumn{1}{|c|}{\multirow{2}{*}{$\mathfrak{sl}(n,\mathbb{C})$}} &
\multicolumn{1}{|c|}{$\mathfrak{so}(n,\mathbb{C})$}&
 \multicolumn{1}{|c|}{$2\lambda_1$}      \\ \cline{2-3}
 & $(\frac12(n^2+n-2),\frac12(n^2+n-2))$ &    \\ \cline{1-3}

\multicolumn{1}{|c|}{\multirow{2}{*}{$\mathfrak{sl}(2n,\mathbb{C})$}} &
\multicolumn{1}{|c|}{$\mathfrak{sp}(2n,\mathbb{C})$}&
 \multicolumn{1}{|c|}{$\lambda_2$}      \\ \cline{2-3}
 & $(2n^2-n-1,2n^2-n-1)$ &    \\ \cline{1-3}

\multicolumn{1}{|c|}{\multirow{2}{*}{$\mathfrak{sl}(n,\mathbb{C})$}} &
\multicolumn{1}{|c|}{$\mathfrak{su}(p,q)$}&
 \multicolumn{1}{|c|}{$\lambda_1+\lambda_{n-1}$}      \\ \cline{2-3}
 & $(p^2+q^2-1,2pq)$ &    \\ \cline{1-3}

\multicolumn{1}{|c|}{\multirow{2}{*}{$\mathfrak{sl}(n,\mathbb{C})$}} &
\multicolumn{1}{|c|}{$\mathfrak{sl}(n,\mathbb{R})$}&
 \multicolumn{1}{|c|}{$\lambda_1+\lambda_{n-1}$}      \\ \cline{2-3}
 & $(\frac12(n^2-n),\frac12(n^2+n-2))$ &    \\ \cline{1-3}

\multicolumn{1}{|c|}{\multirow{2}{*}{$\mathfrak{sl}(2n,\mathbb{C})$}} &
\multicolumn{1}{|c|}{$\mathfrak{sl}(n,\mathbb{H})$}&
 \multicolumn{1}{|c|}{$\lambda_1+\lambda_{n-1}$}      \\ \cline{2-3}
 & $(2n^2+n,2n^2-n-1)$ &    \\ \cline{1-3}
\end{tabular}

\resizebox{\textwidth-18pt}{!}{
{
\renewcommand{\arraystretch}{1.2}
\begin{tabular}{|c|c|c|}
\cline{1-3}
\multicolumn{1}{|c|}{\multirow{2}{*}{$\mathfrak{k}$}} &
\multicolumn{1}{|c|}{$\mathfrak{h}$} &
 \multicolumn{1}{|c|}{$ad(\mathfrak{h})$}      \\ \cline{2-3}
 & signature & properties    \\ \cline{1-3}

\multicolumn{1}{|c|}{\multirow{2}{*}{$\mathfrak{su}(k+p,l+q)$}} &
\multicolumn{1}{|c|}{$\mathfrak{su}(k,l)+\mathfrak{su}(p,q)+\mathfrak{so}(2)$}&
 \multicolumn{1}{|c|}{$\lambda_1\otimes \lambda_{1}^*\oplus \lambda_{1}^*\otimes \lambda_{1}$}      \\ \cline{2-3}
 & $(2kq+2lp,2kp+2lq)$ & h, q (k=2,l=0), pq (k=1,l=1) \\ \cline{1-3}

\multicolumn{1}{|c|}{\multirow{2}{*}{$\mathfrak{su}(p,q)$}} &
\multicolumn{1}{|c|}{$\mathfrak{so}(p,q)$}&
 \multicolumn{1}{|c|}{$2\lambda_1$}      \\ \cline{2-3}
 & $(4pq,2p^2+2q^2-p-q-1)$ &    \\ \cline{1-3}

\multicolumn{1}{|c|}{\multirow{2}{*}{$\mathfrak{su}(2p,2q)$}} &
\multicolumn{1}{|c|}{$\mathfrak{sp}(p,q)$}&
 \multicolumn{1}{|c|}{$\lambda_2$}      \\ \cline{2-3}
 & $(2n^2-n-1,2n^2-n-1)$ &   \\ \cline{1-3}

\multicolumn{1}{|c|}{\multirow{2}{*}{$\mathfrak{su}(n,n)$}} &
\multicolumn{1}{|c|}{$\mathfrak{sl}(n,\mathbb{C})+\mathfrak{so}(1,1)$}&
 \multicolumn{1}{|c|}{$\lambda_1+ \bar{\lambda}_1 \oplus \lambda_{n-1}+ \bar{\lambda}_{n-1}$}      \\ \cline{2-3}
 & $(n^2,n^2)$ &  ph  \\ \cline{1-3}

\multicolumn{1}{|c|}{\multirow{2}{*}{$\mathfrak{su}(n,n)$}} &
\multicolumn{1}{|c|}{$\mathfrak{so}^\star(2n)$}&
 \multicolumn{1}{|c|}{$2\lambda_1$}      \\ \cline{2-3}
 & $(n^2+n,n^2-1)$ &    \\ \cline{1-3}

\multicolumn{1}{|c|}{\multirow{2}{*}{$\mathfrak{su}(n,n)$}} &
\multicolumn{1}{|c|}{$\mathfrak{sp}(2n,\mathbb{R})$}&
 \multicolumn{1}{|c|}{$\lambda_2$}      \\ \cline{2-3}
 & $(n^2-n,n^2-1)$ &    \\ \cline{1-3}

\multicolumn{1}{|c|}{\multirow{2}{*}{$\mathfrak{sl}(n,\mathbb{R})$}} &
\multicolumn{1}{|c|}{$\mathfrak{sl}(p,\mathbb{R})+\mathfrak{sl}(q,\mathbb{R})+\mathfrak{so}(1,1)$}&
 \multicolumn{1}{|c|}{$\lambda_1\otimes \lambda_{1}^*\oplus \lambda_{1}^*\otimes \lambda_{1}$}      \\ \cline{2-3}
 & $(pq,pq)$ & ph, , pq (p=2)  \\ \cline{1-3}

\multicolumn{1}{|c|}{\multirow{2}{*}{$\mathfrak{sl}(n,\mathbb{R})$}} &
\multicolumn{1}{|c|}{$\mathfrak{so}(p,q)$}&
 \multicolumn{1}{|c|}{$2\lambda_1$}      \\ \cline{2-3}
 & $(\frac12(p^2+q^2+p+q-2),pq)$ &    \\ \cline{1-3}

\multicolumn{1}{|c|}{\multirow{2}{*}{$\mathfrak{sl}(2n,\mathbb{R})$}} &
\multicolumn{1}{|c|}{$\mathfrak{sl}(n,\mathbb{C})+\mathfrak{so}(2)$}&
 \multicolumn{1}{|c|}{$\lambda_1+ \bar{\lambda}_1 \oplus \lambda_{n-1}+ \bar{\lambda}_{n-1}$}      \\ \cline{2-3}
 & $(n^2+n,n^2-n)$ &  h  \\ \cline{1-3}

\multicolumn{1}{|c|}{\multirow{2}{*}{$\mathfrak{sl}(2n,\mathbb{R})$}} &
\multicolumn{1}{|c|}{$\mathfrak{sp}(2n,\mathbb{R})$}&
 \multicolumn{1}{|c|}{$\lambda_2$}      \\ \cline{2-3}
 & $(n^2-1,n^2-n)$ &    \\ \cline{1-3}

\multicolumn{1}{|c|}{\multirow{2}{*}{$\mathfrak{sl}(n,\mathbb{H})$}} &
\multicolumn{1}{|c|}{$\mathfrak{sl}(p,\mathbb{H})+\mathfrak{sl}(q,\mathbb{H})+\mathfrak{so}(1,1)$}&
 \multicolumn{1}{|c|}{$\lambda_1\otimes \lambda_{1}^*\oplus \lambda_{1}^*\otimes \lambda_{1}$}      \\ \cline{2-3}
 & $(4pq,4pq)$ & ph, q (p=1)  \\ \cline{1-3}

\multicolumn{1}{|c|}{\multirow{2}{*}{$\mathfrak{sl}(n,\mathbb{H})$}} &
\multicolumn{1}{|c|}{$\mathfrak{sp}(p,q)$}&
 \multicolumn{1}{|c|}{$\lambda_2$}      \\ \cline{2-3}
 & $(2p^2+2q^2-p-q-1,4pq)$ &   \\ \cline{1-3}

\multicolumn{1}{|c|}{\multirow{2}{*}{$\mathfrak{sl}(n,\mathbb{H})$}} &
\multicolumn{1}{|c|}{$\mathfrak{sl}(n,\mathbb{C})+\mathfrak{so}(2)$}&
 \multicolumn{1}{|c|}{$\lambda_1+ \bar{\lambda}_1 \oplus \lambda_{n-1}+ \bar{\lambda}_{n-1}$}      \\ \cline{2-3}
 & $(n^2-n,n^2+n)$ &  h  \\ \cline{1-3}

\multicolumn{1}{|c|}{\multirow{2}{*}{$\mathfrak{sl}(n,\mathbb{H})$}} &
\multicolumn{1}{|c|}{$\mathfrak{so}^\star(2n)$}&
 \multicolumn{1}{|c|}{$2\lambda_1$}      \\ \cline{2-3}
 & $(n^2-1,n^2+n)$ &    \\ \cline{1-3}
\end{tabular}
}}

{
\renewcommand{\arraystretch}{1.2}
\begin{tabular}{|c|c|c|}
\cline{1-3}
\multicolumn{1}{|c|}{\multirow{2}{*}{$\mathfrak{k}$}} &
\multicolumn{1}{|c|}{$\mathfrak{h}$} &
 \multicolumn{1}{|c|}{$ad(\mathfrak{h})$}      \\ \cline{2-3}
 & signature & properties    \\ \cline{1-3}

\multicolumn{1}{|c|}{\multirow{2}{*}{$\mathfrak{so}(n,\mathbb{C})$}} &
\multicolumn{1}{|c|}{$\mathfrak{so}(p,q)$}&
 \multicolumn{1}{|c|}{$\lambda_2$}      \\ \cline{2-3}
 & $(\frac12(p^2+q^2-p-q),pq)$ &  \\ \cline{1-3}

\multicolumn{1}{|c|}{\multirow{2}{*}{$\mathfrak{so}(2n,\mathbb{C})$}} &
\multicolumn{1}{|c|}{$\mathfrak{so}^\star(2n)$}&
 \multicolumn{1}{|c|}{$\lambda_2$}      \\ \cline{2-3}
 & $(n^2,n^2-n)$ &    \\ \cline{1-3}

\multicolumn{1}{|c|}{\multirow{2}{*}{$\mathfrak{so}(n,\mathbb{C})$}} &
\multicolumn{1}{|c|}{$\mathfrak{so}(p,\mathbb{C})+\mathfrak{so}(q,\mathbb{C})$}&
 \multicolumn{1}{|c|}{$\lambda_1 \otimes \lambda_{1}$}      \\ \cline{2-3}
 & $(pq,pq)$ &    \\ \cline{1-3}

\multicolumn{1}{|c|}{\multirow{2}{*}{$\mathfrak{so}(n,\mathbb{C})$}} &
\multicolumn{1}{|c|}{$\mathfrak{so}(n-2,\mathbb{C})+\mathfrak{so}(2,\mathbb{C})$}&
 \multicolumn{1}{|c|}{$\lambda_1 \oplus \lambda_{1}$}      \\ \cline{2-3}
 & $(2n-4,2n-4)$ &  h, ph  \\ \cline{1-3}

\multicolumn{1}{|c|}{\multirow{2}{*}{$\mathfrak{so}(2n,\mathbb{C})$}} &
\multicolumn{1}{|c|}{$\mathfrak{sl}(n,\mathbb{C})+\mathbb{C}$}&
 \multicolumn{1}{|c|}{$\lambda_2\oplus \lambda_{2}^*$}      \\ \cline{2-3}
 & $(n^2-n,n^2-n)$ &  h,ph  \\ \cline{1-3}

\cline{1-3}
\multicolumn{1}{|c|}{\multirow{2}{*}{$\mathfrak{so}(n,n)$}} &
\multicolumn{1}{|c|}{$\mathfrak{sl}(n,\mathbb{R})+\mathfrak{so}(1,1)$}&
 \multicolumn{1}{|c|}{$\lambda_2\oplus \lambda_{2}^*$}      \\ \cline{2-3}
 & $(\frac12(n^2-n),\frac12(n^2-n))$ & ph \\ \cline{1-3}

\multicolumn{1}{|c|}{\multirow{2}{*}{$\mathfrak{so}(n,n)$}} &
\multicolumn{1}{|c|}{$\mathfrak{so}(n,\mathbb{C})$}&
 \multicolumn{1}{|c|}{$\lambda_1+ \bar{\lambda}_{1}$}      \\ \cline{2-3}
 & $(\frac12(n^2+n),\frac12(n^2-n))$ &    \\ \cline{1-3}

\multicolumn{1}{|c|}{\multirow{2}{*}{$\mathfrak{so}(k+p,l+q)$}} &
\multicolumn{1}{|c|}{$\mathfrak{so}(k,l)+\mathfrak{so}(p,q)$}&
 \multicolumn{1}{|c|}{$\lambda_1 \otimes \lambda_{1}$}      \\ \cline{2-3}
 & $(kq+lp,kp+lq)$ &  q (k=4,l=0), pq (k=2,l=2)  \\ \cline{1-3}

\multicolumn{1}{|c|}{\multirow{2}{*}{$\mathfrak{so}(k+2,l)$}} &
\multicolumn{1}{|c|}{$\mathfrak{so}(k,l)+\mathfrak{so}(2)$}&
 \multicolumn{1}{|c|}{$\lambda_1 \oplus \lambda_{1}$}      \\ \cline{2-3}
 & $(2l,2k)$ &  h  \\ \cline{1-3}

\multicolumn{1}{|c|}{\multirow{2}{*}{$\mathfrak{so}(k+1,l+1)$}} &
\multicolumn{1}{|c|}{$\mathfrak{so}(k,l)+\mathfrak{so}(1,1)$}&
 \multicolumn{1}{|c|}{$\lambda_1 \oplus \lambda_{1}$}      \\ \cline{2-3}
 & $(k+l,k+l)$ &  ph  \\ \cline{1-3}

\multicolumn{1}{|c|}{\multirow{2}{*}{$\mathfrak{so}(2p,2q)$}} &
\multicolumn{1}{|c|}{$\mathfrak{su}(p,q)+\mathfrak{so}(2)$}&
 \multicolumn{1}{|c|}{$\lambda_2\oplus \lambda_{2}^*$}      \\ \cline{2-3}
 & $(2pq,p^2+q^2-p-q)$ &  h  \\ \cline{1-3}

\multicolumn{1}{|c|}{\multirow{2}{*}{$\mathfrak{so}^\star(2n)$}} &
\multicolumn{1}{|c|}{$\mathfrak{su}(p,q)+\mathfrak{so}(2)$}&
 \multicolumn{1}{|c|}{$\lambda_2\oplus \lambda_{2}^*$}      \\ \cline{2-3}
 & $(p^2+q^2-p-q,2pq)$ & h \\ \cline{1-3}

\multicolumn{1}{|c|}{\multirow{2}{*}{$\mathfrak{so}^\star(2n)$}} &
\multicolumn{1}{|c|}{$\mathfrak{so}(n,\mathbb{C})$}&
 \multicolumn{1}{|c|}{$\lambda_1 + \bar{\lambda}_{1}$}      \\ \cline{2-3}
 & $(\frac12(n^2-n),\frac12(n^2+n))$ &    \\ \cline{1-3}

\multicolumn{1}{|c|}{\multirow{2}{*}{$\mathfrak{so}^\star(2n)$}} &
\multicolumn{1}{|c|}{$\mathfrak{so}^\star(2p)+\mathfrak{so}^\star(2q)$}&
 \multicolumn{1}{|c|}{$\lambda_1 \otimes \lambda_{1}^*$}      \\ \cline{2-3}
 & $(2pq,2pq)$ & q, pq (p=2)  \\ \cline{1-3}

\multicolumn{1}{|c|}{\multirow{2}{*}{$\mathfrak{so}^\star(2n+2)$}} &
\multicolumn{1}{|c|}{$\mathfrak{so}^\star(2n)+\mathfrak{so}^\star(2)$}&
 \multicolumn{1}{|c|}{$\lambda_1 \oplus \lambda_{1}^*$}      \\ \cline{2-3}
 & $(2pq,2pq)$ & h  \\ \cline{1-3}

\multicolumn{1}{|c|}{\multirow{2}{*}{$\mathfrak{so}^\star(4n)$}} &
\multicolumn{1}{|c|}{$\mathfrak{sl}(n,\mathbb{H})+\mathfrak{so}(1,1)$}&
 \multicolumn{1}{|c|}{$\lambda_2\oplus \lambda_{2}^*$}      \\ \cline{2-3}
 & $(2n^2-n,2n^2-n)$ &  ph  \\ \cline{1-3}
\end{tabular}
}

{
\renewcommand{\arraystretch}{1.2}
\begin{tabular}{|c|c|c|}
\cline{1-3}
\multicolumn{1}{|c|}{\multirow{2}{*}{$\mathfrak{k}$}} &
\multicolumn{1}{|c|}{$\mathfrak{h}$} &
 \multicolumn{1}{|c|}{$ad(\mathfrak{h})$}      \\ \cline{2-3}
 & signature & properties    \\ \cline{1-3}

\multicolumn{1}{|c|}{\multirow{2}{*}{$\mathfrak{sp}(n,\mathbb{C})$}} &
\multicolumn{1}{|c|}{$\mathfrak{sp}(p,q)$}&
 \multicolumn{1}{|c|}{$2\lambda_1$}      \\ \cline{2-3}
 & $(2p^2+2q^2+p+q,4pq)$ &  \\ \cline{1-3}

\multicolumn{1}{|c|}{\multirow{2}{*}{$\mathfrak{sp}(2n,\mathbb{C})$}} &
\multicolumn{1}{|c|}{$\mathfrak{sp}(2n,\mathbb{R})$}&
 \multicolumn{1}{|c|}{$2\lambda_1$}      \\ \cline{2-3}
 & $(n^2,n^2+n)$ &    \\ \cline{1-3}

\multicolumn{1}{|c|}{\multirow{2}{*}{$\mathfrak{sp}(n,\mathbb{C})$}} &
\multicolumn{1}{|c|}{$\mathfrak{sp}(p,\mathbb{C})+\mathfrak{sp}(q,\mathbb{C})$}&
 \multicolumn{1}{|c|}{$\lambda_1 \otimes \lambda_{1}^*$}      \\ \cline{2-3}
 & $(4pq,4pq)$ &    \\ \cline{1-3}

\multicolumn{1}{|c|}{\multirow{2}{*}{$\mathfrak{sp}(n,\mathbb{C})$}} &
\multicolumn{1}{|c|}{$\mathfrak{sl}(n,\mathbb{C})+\mathbb{C}$}&
 \multicolumn{1}{|c|}{$2\lambda_1\oplus 2\lambda_{1}^*$}      \\ \cline{2-3}
 & $(n^2+n,n^2+n)$ &  h,ph  \\ \cline{1-3}

\multicolumn{1}{|c|}{\multirow{2}{*}{$\mathfrak{sp}(n,n)$}} &
\multicolumn{1}{|c|}{$\mathfrak{sl}(n,\mathbb{H})+\mathfrak{so}(1,1)$}&
 \multicolumn{1}{|c|}{$2\lambda_1\oplus 2\lambda_{1}^*$}      \\ \cline{2-3}
 & $(2n^2+n,2n^2+n)$ & ph \\ \cline{1-3}

\multicolumn{1}{|c|}{\multirow{2}{*}{$\mathfrak{sp}(n,n)$}} &
\multicolumn{1}{|c|}{$\mathfrak{sp}(n,\mathbb{C})$}&
 \multicolumn{1}{|c|}{$\lambda_1+ \bar{\lambda}_{1}$}      \\ \cline{2-3}
 & $(2n^2-n,2n^2+n)$ &    \\ \cline{1-3}

\multicolumn{1}{|c|}{\multirow{2}{*}{$\mathfrak{sp}(k+p,l+q)$}} &
\multicolumn{1}{|c|}{$\mathfrak{sp}(k,l)+\mathfrak{sp}(p,q)$}&
 \multicolumn{1}{|c|}{$\lambda_1 \otimes \lambda_{1}^*$}      \\ \cline{2-3}
 & $(4kq+4lp,4kp+4lq)$ &  q (k=1,l=0)  \\ \cline{1-3}

\multicolumn{1}{|c|}{\multirow{2}{*}{$\mathfrak{sp}(p,q)$}} &
\multicolumn{1}{|c|}{$\mathfrak{su}(p,q)+\mathfrak{so}(2)$}&
 \multicolumn{1}{|c|}{$2\lambda_1\oplus 2\lambda_{1}^*$}      \\ \cline{2-3}
 & $(2pq,p^2+q^2+p+q)$ &  h  \\ \cline{1-3}

\multicolumn{1}{|c|}{\multirow{2}{*}{$\mathfrak{sp}(2n,\mathbb{R})$}} &
\multicolumn{1}{|c|}{$\mathfrak{su}(p,q)+\mathfrak{so}(2)$}&
 \multicolumn{1}{|c|}{$\lambda_2\oplus \lambda_{2}^*$}      \\ \cline{2-3}
 & $(p^2+q^2+p+q,2pq)$ & h \\ \cline{1-3}

\multicolumn{1}{|c|}{\multirow{2}{*}{$\mathfrak{sp}(2n,\mathbb{R})$}} &
\multicolumn{1}{|c|}{$\mathfrak{sp}(n,\mathbb{C})$}&
 \multicolumn{1}{|c|}{$\lambda_1 +\bar{\lambda}_{1}$}      \\ \cline{2-3}
 & $(2n^2+n,2n^2-n)$ &    \\ \cline{1-3}

\multicolumn{1}{|c|}{\multirow{2}{*}{$\mathfrak{sp}(2n,\mathbb{R})$}} &
\multicolumn{1}{|c|}{$\mathfrak{sp}(2p,\mathbb{R})+\mathfrak{sp}(2q,\mathbb{R})$}&
 \multicolumn{1}{|c|}{$\lambda_1 \otimes \lambda_{1}^*$}      \\ \cline{2-3}
 & $(2pq,2pq)$ &  pq (p=1)  \\ \cline{1-3}

\multicolumn{1}{|c|}{\multirow{2}{*}{$\mathfrak{sp}(2n,\mathbb{R})$}} &
\multicolumn{1}{|c|}{$\mathfrak{sl}(n,\mathbb{R})+\mathfrak{so}(1,1)$}&
 \multicolumn{1}{|c|}{$2\lambda_1\oplus 2\lambda_{1}^*$}      \\ \cline{2-3}
 & $(\frac12(n^2+n),\frac12(n^2+n))$ &  ph  \\ \cline{1-3}
\end{tabular}
}

\section{Tables of gradings}\label{appB}

The following tables comes form \cite{odk3}. First table contains $1$-gradings of simple Lie algebras. We use the following symbols and abbreviations in the table below:

\begin{itemize}
\item[$\mathfrak{g}$] is simple Lie algebra with one grading $\mathfrak{g}=\mathfrak{g}_{-1}+\mathfrak{g}_0+\mathfrak{g}_1$ of non-exceptional type
\item[$\mathfrak{g}_{-1}$] is given as a representation of $ad(\mathfrak{g}_0)$ on $\mathfrak{g}_{-1}$ in terms of fundamental representations.
\end{itemize}

{
\renewcommand{\arraystretch}{1.2}
\begin{center}
\begin{tabular}{|c|c|c|}
\hline
$\mathfrak{g}$ & $\mathfrak{g}_0$ & $\mathfrak{g}_{-1}$ \\
\hline
$\mathfrak{sl}(n+1,\mathbb{R})$ & $\mathfrak{sl}(n,\mathbb{R})+\mathbb{R}$ & $\lambda_1$ \\
\hline
$\mathfrak{sl}(n,\mathbb{R})$ & $\mathfrak{sl}(p,\mathbb{R})+\mathfrak{sl}(q,\mathbb{R})+\mathbb{R}$ & $\lambda_{p-1}\otimes \lambda_{1}$ \\
\hline
$\mathfrak{sl}(n,\mathbb{C})$ & $\mathfrak{sl}(p,\mathbb{C})+\mathfrak{sl}(q,\mathbb{C})+\mathbb{C}$ & $\lambda_{p-1}\otimes \lambda_{1}$ \\
\hline
$\mathfrak{sl}(n,\mathbb{H})$ & $\mathfrak{sl}(p,\mathbb{H})+\mathfrak{sl}(q,\mathbb{H})+\mathbb{R}$ & $\lambda_{p-1}\otimes \lambda_{1}$ \\
\hline
$\mathfrak{su}(n,n)$ & $\mathfrak{sl}(n,\mathbb{C})+\mathbb{R}$ & $\lambda_{1}+ \bar{\lambda}_{1}$ \\
\hline
$\mathfrak{sp}(2n,\mathbb{R})$ &  $\mathfrak{sl}(n,\mathbb{R})+\mathbb{R}$ & $2\lambda_{1}$ \\
\hline
$\mathfrak{sp}(2n,\mathbb{C})$ &  $\mathfrak{sl}(n,\mathbb{C})+\mathbb{C}$ & $2\lambda_{1}$ \\
\hline
$\mathfrak{sp}(n,n)$ &  $\mathfrak{sl}(n,\mathbb{H})+\mathbb{R}$ & $2\lambda_{1}$ \\
\hline
$\mathfrak{so}(p+1,q+1)$ & $\mathfrak{so}(p,q)+\mathbb{R}$ & $\lambda_{1}$ \\
\hline
$\mathfrak{so}(n+2,\mathbb{C})$ & $\mathfrak{so}(n,\mathbb{C})+\mathbb{C}$ & $\lambda_{1}$ \\
\hline
$\mathfrak{so}(p,p)$ & $\mathfrak{sl}(n,\mathbb{R})+\mathbb{R}$ & $\lambda_{2}$ \\
\hline
$\mathfrak{so}(2n,\mathbb{C})$ & $\mathfrak{sl}(n,\mathbb{C})+\mathbb{C}$ & $\lambda_{2}$ \\
\hline
$\mathfrak{so}^\star(4n)$ & $\mathfrak{sl}(n,\mathbb{H})+\mathbb{R}$ & $\lambda_{2}$ \\
\hline
\end{tabular}
\end{center}
}

This table contains complex contact gradings of simple complex Lie algebras. We use the following symbols and abbreviations in the table below:

\begin{itemize}
\item[$\mathfrak{g}$] is simple complex Lie algebra with complex contact two grading $\mathfrak{g}=\mathbb{C}+\mathfrak{g}_{-1}+\mathfrak{g}_0+\mathfrak{g}_1+\mathbb{C}$ of non-exceptional type
\item[$\mathfrak{g}_{-1}$] is given as a representation of $ad(\mathfrak{g}_0)$ on $\mathfrak{g}_{-1}$ in terms of fundamental representations.
\end{itemize}

\begin{center}
\begin{tabular}{|c|c|c|}
\hline
$\mathfrak{g}$ & $\mathfrak{g}_0$ & $\mathfrak{g}_{-1}$ \\
\hline
$\mathfrak{sl}(n+2,\mathbb{C})$ & $\mathfrak{sl}(n,\mathbb{C})+\mathbb{C}^2$  & $\lambda_1\oplus \lambda_{n-1}$ \\
\hline
$\mathfrak{so}(n+4,\mathbb{C})$ & $\mathfrak{so}(n,\mathbb{C})+\mathfrak{sl}(2,\mathbb{C})+\mathbb{C}$ & $\lambda_1\otimes \lambda_1$ \\
\hline
$\mathfrak{sp}(2n+2,\mathbb{C})$ & $\mathfrak{sp}(2n,\mathbb{C})+\mathbb{C}$ & $\lambda_1$ \\
\hline
\end{tabular}
\end{center}

This table contains contact gradings of simple Lie algebras. We use the following symbols and abbreviations in the table below:

\begin{itemize}
\item[$\mathfrak{g}$] is simple Lie algebra with contact two grading $\mathfrak{g}=\mathbb{R}+\mathfrak{g}_{-1}+\mathfrak{g}_0+\mathfrak{g}_1+\mathbb{R}$ of non-exceptional type
\item[$\mathfrak{g}_{-1}$] is given as a representation of $ad(\mathfrak{g}_0)$ on $\mathfrak{g}_{-1}$ in terms of fundamental representations.
\end{itemize}

\begin{center}
\begin{tabular}{|c|c|c|}
\hline
$\mathfrak{g}$ & $\mathfrak{g}_0$ & $\mathfrak{g}_{-1}$ \\
\hline
$\mathfrak{sl}(n+2,\mathbb{R})$ & $\mathfrak{sl}(n,\mathbb{R})+\mathbb{R}^2$ & $\lambda_1\oplus \lambda_{n-1}$ \\
\hline
$\mathfrak{su}(p+1,q+1)$ & $\mathfrak{su}(p,q)+\mathbb{R}^2$ &  $\lambda_1$\\
\hline
$\mathfrak{so}(p+2,q+2)$ & $\mathfrak{so}(p,q)+\mathfrak{sl}(2,\mathbb{R})+\mathbb{R}$ & $\lambda_1\otimes \lambda_1$ \\
\hline
$\mathfrak{sp}(2n+2,\mathbb{R})$ & $\mathfrak{sl}(2n,\mathbb{R})+\mathbb{R}$ & $\lambda_1$ \\
\hline
$\mathfrak{so}^\star(2n+2)$ & $\mathfrak{so}^\star(2n)+\mathfrak{su}(2)+\mathbb{R}$ & $\lambda_1\otimes \lambda_1$ \\
\hline
\end{tabular}
\end{center}

\end{document}